\documentclass[11pt]{article}  
\usepackage{amsthm}
\usepackage{amssymb, amsmath}
\usepackage{verbatim}
\newtheorem{theorem}{Theorem}
\newtheorem{definition}{Definition}
\newtheorem{lemma}{Lemma}
 \newtheorem{prop}{Proposition}
 \newtheorem{corollary}{Corollary}

\usepackage[margin=.5in,footskip=0.25in]{geometry}
\newgeometry{left=0.5in,right=0.5in,top=.6in,bottom=1in}

\begin{document}
 \title{Generic Multilinear Multipliers Associated to Degenerate Simplexes}
 \author{Robert M. Kesler}
 \maketitle

 \abstract{For each $1 \leq  p \leq \infty$, let $W_{p}(\mathbb{R}) = \left\{ f \in L^p(\mathbb{R}): \hat{f} \in L^{p^\prime}(\mathbb{R}) \right\}$ with norm $||f||_{W_{p}(\mathbb{R})} = ||\hat{f}||_{L^{p^\prime}(\mathbb{R})}$. Moreover, let $ \Gamma = \left\{ \xi_1 + \xi_2 =0\right\} \subset \mathbb{R}^2$ and $a_1,a_2 : \mathbb{R}^2 \rightarrow \mathbb{C}$ satisfy the H\"{o}rmander-Mikhlin condition
 \begin{eqnarray*}
 \left| \partial^{\vec{\alpha}} a_j \left(\vec{\xi}\right) \right| \lesssim_{\vec{\alpha}} \frac{1}{dist(\vec{\xi}, \Gamma)^{|\vec{\alpha}|}}~~~\forall \vec{\xi} \in \mathbb{R}^2, j \in \{1, 2\}
 \end{eqnarray*}
for sufficiently many multi-indices $\vec{\alpha} \in (\mathbb{N} \bigcup \{0\})^2$. Our main result is that the generic degenerate trilinear simplex multiplier defined on $ \mathcal{S}^3(\mathbb{R})$ by 
 
 \begin{eqnarray*}
 B[a_1, a_2]   : (f_1, f_2, f_3) \rightarrow \int_{\mathbb{R}^3} a_1(\xi_1, \xi_2) a_2(\xi_2, \xi_3) \left[ \prod_{j=1}^3 \hat{f_j} (\xi_j) e^{2 \pi ix \xi_j} \right] d\xi_1 d\xi_2 d\xi_3
 \end{eqnarray*}
extends to a map $L^{p_1}(\mathbb{R}) \times W_{p_2}(\mathbb{R}) \times L^{p_3}(\mathbb{R}) \rightarrow L^{\frac{1}{\frac{1}{p_1} + \frac{1}{p _2} +\frac{1}{p_3}}}(\mathbb{R})$ 
  provided

\begin{eqnarray*}
1 < p_1, p_3 \leq \infty, \frac{1}{p_1} + \frac{1}{p_2} <1, \frac{1}{p_2} + \frac{1}{p_3} <1, 2 < p_2 <\infty.
 \end{eqnarray*}
  


}

\section{Introduction}
Several recent articles have investigated singular integral operators associated to simplexes from a time-frequency perspective, e.g. \cite{2015arXiv151104948B, MR1491450,MR2221256,MR1887641,MR2127984, MR2127985,MR3329857}.  Such objects arise naturally in the asymptotic expansions of AKNS systems, where particular $L^p$ estimates are sought for the maps defined for any $n \in \mathbb{N}$ and $\vec{f} \in \mathcal{S}^n(\mathbb{R})$ by
\begin{eqnarray*}
C_n: (f_1, ..., f_n) \mapsto  \sup_t \left| \int_{\xi_1 < ... <\xi_n<t}\left[ \prod_{j=1}^n f_j(\xi_j)e^{2 \pi i  x(-1)^j\xi_j}  \right] d \vec{\xi} \right|.
\end{eqnarray*} 
 For details on the connection between $\left\{ C_n \right\}_{n \geq 1}$ and AKNS, consult \cite{MR1809116}. 
Moreover, there is a growing literature aimed at understanding the collection of Fourier multipliers defined for any $n \in \mathbb{N}, \vec{\epsilon} \in \mathbb{R}^n$ and $\vec{f} \in \mathcal{S}^n(\mathbb{R})$ by 

\begin{eqnarray*}
C^{\vec{\epsilon}}: (f_1, ..., f_n)  \mapsto \int_{\xi_1 < ... < \xi_n} \hat{f}_1(\xi_1)... \hat{f}_n(\xi_n) e^{2 \pi i x \vec{\epsilon} \cdot \vec{\xi}} d\vec{\xi}. 
\end{eqnarray*}
A fundamental goal in this area of research is to understand how the mapping properties of $C^{\vec{\epsilon}}$ vary with the parameter $\vec{\epsilon}$. For instance, C. Muscalu, T. Tao,  and C. Thiele have proven a wide range of H\"{o}lder-type $L^p$ estimates for $C^{\vec{\epsilon}}$ under the assumption that $\vec{\epsilon}$ is non-degenerate in \cite{MR3329857} and also observed in \cite{MR1981900} that $C^{\vec{\epsilon}}$ satisfies no $L^p$ estimates whenever $\vec{\epsilon} \in \mathbb{R}^n$ is degenerate and $n \geq 3$. The meaning of degenerate is made precise in

\begin{definition}
A vector $\vec{\epsilon} \in \mathbb{R}^n$ is degenerate whenever there exists  $i \in \{ 1, ..., n-1\}$ such that $\epsilon_i + \epsilon_{i+1} =0$. 
\end{definition}
 However, degenerate multilinear simplex multipliers that fail to satisfy any $L^p$ estimates may nonetheless have mixed estimate available, by which we mean the following: 

\begin{definition}
Let $n \in \mathbb{N}$ and $\vec{p} = (p_1, ..., p_n) \in [1, \infty]^n$. For each $i \in \{1, ..., n\}$, let 

\begin{eqnarray*}
W_{p_i}(\mathbb{R}) := \left\{ f \in L^{p_i} (\mathbb{R}) : \hat{f} \in L^{p^\prime_i}(\mathbb{R}) \right\}
\end{eqnarray*}
 with $||f||_{W_{p_i}(\mathbb{R})}:= \left|\left|\hat{f}\right|\right|_{L^{p_i^\prime}(\mathbb{R})}$. In addition, let $X_{p_i}(\mathbb{R}) \in \{ L^{p_i}(\mathbb{R}), W_{p_i} (\mathbb{R})\}$ for all $i \in \{1, ..., n\}$. An n-(sub)linear operator $T$ satisfies the mixed estimate $ \prod_{i=1}^n X_{p_i}(\mathbb{R}) \rightarrow L^{\frac{1}{\sum_{i=1}^n \frac{1}{p_i}}} (\mathbb{R})$ whenever 
\begin{eqnarray*}
\left|\left| T(\vec{f}) \right|\right|_{L^{\frac{1}{\sum_{i=1}^n \frac{1}{p_i}}}(\mathbb{R})} \lesssim C_{\vec{p}} \prod_{i=1}^n || f_i||_{X_{p_i}(\mathbb{R})}
\end{eqnarray*}
for all $(f_1, ..., f_n)\in \mathcal{S}^n(\mathbb{R})$.
\end{definition}
\emph{Remark:} $W_p(\mathbb{R}) \subsetneq L^p(\mathbb{R})$ for $2 < p \leq \infty$ and $W_p(\mathbb{R}) =L^p(\mathbb{R})$ for $1 \leq p \leq 2$ by Hausdorff-Young.

 Mixed estimates for $C^{-1,1,1}$ and $C^{-1,1,1}$ are already known via arguments in \cite{MR3286565} using martingale structure decompositions similar to those in \cite{MR1809116,MR3052499} combined with a Littlewood-Paley result for arbitrary intervals due to Rubio de Francia in \cite{MR2293255}, and the Bi-Carleson estimates of Muscalu, Tao, and Thiele  from \cite{MR2221256}. Additionally, C. Benea and C. Muscalu have independently obtained in \cite{2015arXiv151104948B} the same estimates as a consequence of vector-valued arguments.  One of the principal results from \cite{MR3286565} is 

\begin{theorem}
 The trilinear multiplier defined on $\mathcal{S}^3(\mathbb{R})$ by the formula
 
 \begin{eqnarray*}
 C^{-1,1,-1}  : (f_1, f_2, f_3) \rightarrow \int_{\xi_1 < \xi_2 < \xi_3}  \left[ \prod_{j=1}^3 \hat{f_j} (\xi_j) e^{2 \pi ix (-1)^j} \right] d\xi_1 d\xi_2 d\xi_3
 \end{eqnarray*}
satisfies the mixed estimate $L^{p_1}(\mathbb{R}) \times W_{p_2}(\mathbb{R}) \times L^{p_3}(\mathbb{R}) \rightarrow L^{\frac{1}{\frac{1}{p_1} + \frac{1}{p _2} +\frac{1}{p_3}}}(\mathbb{R})$ 
  provided

\begin{eqnarray*}
1 < p_1, p_3 <\infty, \frac{1}{p_1} + \frac{1}{p_2} <1, \frac{1}{p_2} + \frac{1}{p_3} <1, 2 < p_2 \leq \infty.
 \end{eqnarray*}
 In particular, $C^{-1,1,-1}$
 has mixed estimates into $L^r(\mathbb{R})$ for all $\frac{1}{2} < r<\infty$.
  \end{theorem}

The main purpose of this article is to adopt a time-frequency perspective in proving mixed estimates for the class of bilinear multipliers adapted to the degenerate line $ \{\xi_1 + \xi_2=0\} \subset \mathbb{R}^2$ as well as the class of trilinear multipliers of the form $m(\xi_1, \xi_2 , \xi_3) =a_1( \xi_1, \xi_2 ) a_2(\xi_2, \xi_3)$ where both $a_1$ and $a_2$ are adapted to $\{\xi_1 + \xi_2 =0\}$. Before listing the main theorems, we include the following for readers' convenience:

\begin{definition}
$\Gamma = \left\{ (\xi_1,\xi_2 ) \in \mathbb{R}^2: \xi_1 + \xi_2 =0 \right\}$. 
\end{definition}

\begin{definition}
\begin{eqnarray*}
\mathcal{M}_\Gamma (\mathbb{R}^2) = \left\{ m : \mathbb{R}^2 \rightarrow \mathbb{C}: |\partial^{\vec{\alpha}} m(\vec{\xi})| \lesssim_{\vec{\alpha}} \frac{1}{dist(\vec{\xi}, \Gamma)^{|\vec{\alpha}|}}~ \text{for sufficiently many multi-indices}~ \vec{\alpha} \in (\mathbb{N} \cup \{0 \})^2 \right\}.
\end{eqnarray*}
\end{definition}
\begin{definition}
Given $n \in \mathbb{N}$ and $m : \mathbb{R}^n \rightarrow \mathbb{C}$ for which $|| m ||_{L^\infty(\mathbb{R}^n)} <\infty$, let $T_m$ be the multilinear map on $\mathcal{S}^n(\mathbb{R})$ given by 

\begin{eqnarray*}
T_m : (f_1, ..., f_n) \mapsto \int_{\mathbb{R}^n} m(\vec{\xi}) \prod_{j=1}^n \left[ \hat{f}_j(\xi_j) e^{2 \pi i x \xi_j} \right] d \vec{\xi}. 
\end{eqnarray*}

\end{definition}
The main theorems are as follows:

\newtheorem*{PT1}{Theorem \ref{PT1}}
\begin{PT1}
Let $m \in \mathcal{M}_\Gamma (\mathbb{R}^2)$. Then  $T_m:L^{p_1}(\mathbb{R}) \times W_{p^\prime_2}(\mathbb{R}) \rightarrow L^{\frac{p_1 p^\prime_2}{p_1 + p^\prime_2}}(\mathbb{R})$ whenever $\frac{1}{p_1} + \frac{1}{p_2} <1$ and $p_2>2$. 
\end{PT1}

 \newtheorem*{PT2}{Theorem \ref{PT2}}
 \begin{PT2}
Let $a_1, a_2 \in \mathcal{M}_\Gamma(\mathbb{R}^2)$ and construct the trilinear operator $T^{a_1, a_2}$ defined on $(f_1, f_2, f_3) \in \mathcal{S}(\mathbb{R})^3 $ by

\begin{eqnarray*}
B[a_1, a_2] (f_1, f_2, f_3) (x) := \int_{\mathbb{R}^3} a_1(\xi_1, \xi_2) a_2(\xi_2, \xi_3) \hat{f}_1(\xi_1) \hat{f}_2(\xi_3) \hat{f}_3(\xi_3) e^{2 \pi i x (\xi_1+ \xi_2 + \xi_3)} d \xi_1 d\xi_2 d \xi_3.
\end{eqnarray*}
Then $B[a_1, a_2]$ extends to a bounded operator from $W_{p_1}(\mathbb{R}) \times L^{p_2}(\mathbb{R}) \times W_{p_3}(\mathbb{R}) \rightarrow L^{\frac{1}{\frac{1}{p_1} + \frac{1}{p_2} + \frac{1}{p_3}}}(\mathbb{R})$ for all $\frac{1}{p_1} + \frac{1}{p_2}+\frac{1}{p_3} <1$ and $p_1, p_3 >2$. Specifically,  $T^{a_1, a_2}$ can be defined on all $(f_1, f_2, f_3)$ such that $\hat{f}_1 \in L^{p_1^\prime}(\mathbb{R}), f_2 \in L^{p_2}(\mathbb{R})$ and $\hat{f}_3 \in L^{p^\prime_3}(\mathbb{R})$ in such a way that
\begin{eqnarray*}
|| B[a_1, a_2](f_1, f_2, f_3) ||_{L^{\frac{1}{\sum \frac{1}{p_1} +\frac{1}{p_2} + \frac{1}{p_3}}}(\mathbb{R})} \lesssim _{p_1, p_2, p_3} ||\hat{f}_1||_{L^{p^\prime_1}(\mathbb{R})} ||f_2||_{L^{p_2}(\mathbb{R})} || \hat{f}_3||_{L^{p_3^\prime}(\mathbb{R})}.
\end{eqnarray*}

\end{PT2}

\newtheorem*{PT3}{Theorem \ref{PT3}}

\begin{PT3}

Let $a_1, a_2 \in \mathcal{M}_\Gamma (\mathbb{R}^2)$. Then the trilinear simplex multiplier defined on $\mathcal{S}^3(\mathbb{R})$ by the formula
 
 \begin{eqnarray*}
 B[a_1, a_2]   : (f_1, f_2, f_3) \rightarrow \int_{\mathbb{R}^3} a_1(\xi_1, \xi_2) a_2(\xi_2, \xi_3) \left[ \prod_{j=1}^3 \hat{f_j} (\xi_j) e^{2 \pi ix \xi_j} \right] d\xi_1 d\xi_2 d\xi_3
 \end{eqnarray*}
maps $L^{p_1}(\mathbb{R}) \times W_{p_2}(\mathbb{R}) \times L^{p_3}(\mathbb{R})$ into $L^{\frac{1}{\frac{1}{p_1} + \frac{1}{p _2} +\frac{1}{p_3}}}(\mathbb{R})$ 
  provided

\begin{eqnarray*}
1 < p_1, p_3 <\infty, \frac{1}{p_1} + \frac{1}{p_2} <1, \frac{1}{p_2} + \frac{1}{p_3} <1, 2 < p_2 <\infty.
 \end{eqnarray*}
 In particular, $B[a_1, a_2]$
 has mixed estimates into $L^r(\mathbb{R})$ for all $\frac{1}{2} < r<\infty$.
  \end{PT3}
Before proving estimates for model sums, we expend some effort in motivating the above theorems by showing how far the existing methods from \cite{MR3286565} can be applied in the setting of this article and why, therefore, a  detailed microlocal analysis seems to be necessary to produce the full range of mixed estimates for generic trilinear degenerate simplex multipliers. The main difficulty with the classical methods is that they are unable to produce $LWL$ type mixed estimates for $B[a_1, a_2]$. This is ultimately because the standard discretization procedure for multilinear multipliers  does not produce a paracomposition or an object that can be globally rewritten as a sum of such paracompositions. The nature of this problem becomes apparent in \S{14}-\S{16}. 

The issues arising in the proof of Theorem \ref{MT} are somewhat similar to those arising in J. Jung's proof of $L^p$ estimates for generic trilinear Biest operators in \cite{2013arXiv1311.1574J}. In fact, Jung's method involved rewriting the generic non-degenerate trilinear multiplier using a Taylor expansion with remainder term that is adapted to a non-degenerate line, where each of the main terms is a non-degenerate paracomposition and thereby handled via Biest model estimates.  In our situation, the global Taylor expansion method present in \cite{2013arXiv1311.1574J} will not work. Instead, one needs to apply the Taylor trick locally. Rewriting the multiplier as a sum of \emph{local} paracompositions with a carefully controlled error term does yield the desired $LWL$ mixed estimates. 

\subsection{Organization}
 In \S{2}, we construct an explicit counterexample for generic bilinear multipliers of Hilbert transform type before obtaining mixed estimates for these objects in \S{3}.  In \S{4}, we prove mixed estimates for a maximal variant of the generic degenerate bilinear multiplier before reestablishing mixed estimates for $C^{-1,1,-1}$ in \S{5} using Christ-Kiselev-Paley decompositions, Rubio de Francia inequalities, and familiar vector-valued C-Z bounds. We prove mixed estimates for the hybrid case in \S{6}, where the symbol $m(\xi_1, \xi_2, \xi_3) = sgn(\xi_1 + \xi_2) a(\xi_2, \xi_3)$ before showing so-called $WLW$ type mixed estimates for $B[a_1, a_2]$. In \S{8} we introduce a counterexample to explain why the methods used up to this point are not  sufficient to prove the harder $LWL$ type mixed estimates for $B[a_1, a_2]$. 
  
 In \S{9}, we prove two ``mixed" interpolation results in which one has pointwise control on certain input functions and pointwise control on the Fourier transform of other input functions (depending on the mixed estimate one wishes to prove) before applying the new machinery to prove mixed estimates for the scale-1 version of Hilbert transform and then the generic degenerate bilinear case. It is instructive to develop a time-frequency perspective in the bilinear setting before developing one for the more complicated trilinear world. For this reason, we proceed to re-introduce some important time-frequency definitions in \S{10}, to apply the time-frequency perspective in the bilinear cases in \S{11} and \S{12}, and to prove the hardest $LWL$ type mixed estimates for the generic trilinear simplex multiplier $B[a_1, a_2]$ at the end of this article in \S{13}-\S{15}.

\vspace{20mm}

\section{A Negative Result for Degenerate Bilinear Multipliers}
The following result states that there are multipliers $m : \mathbb{R}^2 \rightarrow \mathbb{C}$, which look like $H: (f_1, f_2) \mapsto H(f_1 \cdot f_2)$ but satisfy no $L^p$ estimates.  It can also be derived as a corollary from the results of Grafakos and Kalton in \cite{}, but the argument presented below is more explicit. 
 \begin{prop}
 There exists $ m \in \mathcal{M}_\Gamma(\mathbb{R}^2)$ for which the bilinear multiplier defined on $\mathcal{S}^2(\mathbb{R})$ by 
 
 \begin{eqnarray*}
 T_m: (f_1, f_2) \mapsto \int_{\mathbb{R}^2} m(\xi_1, \xi_2) \hat{f}_1(\xi_1) \hat{f}_2(\xi_2) e^{2 \pi i x (\xi_1 + \xi_2)} d \xi_1 d \xi_2
 \end{eqnarray*}
  satisfies no $L^p$ estimates. 
 \end{prop}
 
 \begin{proof}
 Pick $\Phi \in C^\infty([-1/2,1/2])$ real, symmetric such that $\hat{\Phi}(0) >0$. Then $\hat{\Phi}$ is also real, symmetric. 
Let $\Gamma >>1$.  Define a collection of  frequency squares
 
 \begin{eqnarray*}
 \mathbb{Q} &:=&  \bigcup_{k \geq 8 }~ \bigcup_{m \in \mathbb{Z}}~ \bigcup_{-2^{k-8} < \lambda < 2^{k-8}} \vec{Q}_{k,m, \lambda},
\end{eqnarray*}
where for each $ k  \geq 8, m \in \mathbb{Z}, -2^{k-8} < \lambda < 2^{k-8}$

\begin{eqnarray*}
\vec{Q}_{k, m \lambda}  := [m+ \lambda 2^{-k}-2^{-k-1}, m + \lambda 2^{-k}+2^{-k-1} ] \times [ -m - \lambda 2^{-k}-2^{-k-1} + \Gamma 2^{-k}, -m - \lambda 2^{-k} +2^{-k-1} + \Gamma 2^{-k} ] .
\end{eqnarray*}
Next, assign
 \begin{eqnarray*}
 \eta^1_{\vec{Q}_{k,m, \lambda}}(x) &:=& 2^{-k} \hat{\Phi}(x2^{-k}) e^{2 \pi i (m + \lambda 2^{-k}) x} \\ 
 \eta^2_{\vec{Q}_{k,m, \lambda}}(x)&=& 2^{-k} \hat{\Phi}(x2^{-k}) e^{-2 \pi i (m + \lambda 2^{-k}- \Gamma 2^{-k}) x} e^{ 2 \pi i \Gamma 2^{-k} m}.
 \end{eqnarray*}
 Let $m(\xi_1 , \xi_2) = \sum_{k\geq 8} \sum_{|\vec{Q}|=2^{-k}} \hat{ \eta}^1_{Q_1} (\xi_1) \hat{\eta}^2_{Q_2}(\xi_2) \in \mathcal{M}_{\{ \xi_1 +\xi_2 =0\}}(\mathbb{R}^2)$. 
 Moreover, letting $\epsilon =1/100$, choose $\psi \in \mathcal{S}(\mathbb{R})$ satisfying 
 
 \begin{eqnarray*}
  1_{[-1/2+\epsilon, 1/2-\epsilon]} \leq \hat{\psi}  \leq1_{[-1/2, 1/2]}.
 \end{eqnarray*}
 In addition, for each $N \in \mathbb{N}$, construct $f_1^N (x) = \sum_{1 \leq n \leq N} \psi(x-n) e^{2 \pi i n x}$ and $f_2^N(x) = \sum_{1 \leq n \leq N} \psi(x-n) e^{-2 \pi i n x}$.  For a given $(c_{Q_1}, c_{Q_2}) = (m + \lambda 2^{-k}, -m - \lambda 2^{-k} + \Gamma 2^{-k})$ for which $[m+\lambda 2^{-k} -2^{-k-1}, m+\lambda 2^{-k}  + 2^{-k-1}] \cap  [n_1-1/2, n_1+1/2] \not = \emptyset  $ and $ [ m+ \lambda 2^{-k} - \Gamma 2^{-k}-2^{-k-1} , m+ \lambda 2^{-k} - \Gamma 2^{-k} + 2^{-k-1}] \cap [n_2-1/2, n_2+ 1/2] \not = \emptyset$, then $m= n_1=n_2$ for all $k \gtrsim C_\Gamma$, in which case
 
 \begin{eqnarray*}
 &  [m+\lambda 2^{-k} -2^{-k-1}, m+\lambda 2^{-k}  + 2^{-k-1}] &\subset [m-1/2+\epsilon, m+1/2-\epsilon] \\ &  [m+\lambda 2^{-k} \Gamma 2^{-k} -2^{-k-1} , m+ \lambda 2^{-k} - \Gamma 2^{-k} + 2^{-k-1}]  &\subset [m-1/2+\epsilon, m+1/2-\epsilon].
 \end{eqnarray*}
Therefore, for each $k \geq C_\Gamma$, we have
 
 \begin{eqnarray*}
 && T^k_m(f_1, f_2)(x)\\&=& \sum_{m \in \mathbb{Z}} \sum_{-2^{k-8} < \lambda < 2^{k-8}}  \sum_{1\leq n_1, n_2 \leq N}  \left (\psi(\cdot -n_1) e^{2 \pi i n_1 \cdot}\right) * \eta^1_{\vec{Q}_{k,m, \lambda}} (x) \left(\psi(\cdot-n_2) e^{-2 \pi i n_2\cdot}\right) * \eta^2_{\vec{Q}_{k, m , \lambda}}(x) \\ &=&  \sum_{1 \leq m \leq N} \sum_{-2^{k-8} < \lambda < 2^{k-8}}   \left (\psi(\cdot -m) e^{2 \pi i m \cdot}\right) * \eta^1_{\vec{Q}_{k,m, \lambda}} (x) \left(\psi(\cdot-m) e^{-2 \pi i m\cdot}\right) * \eta^2_{\vec{Q}_{k, m , \lambda}}(x) \\ &=&  \sum_{1 \leq m \leq N} \sum_{-2^{k-8} < \lambda < 2^{k-8}}  2^{-2k} \hat{\Phi}((x-n)2^k) e^{2 \pi i(m+\lambda 2^{-k})  (x-m)} \hat{\Phi}((x-m)2^k) e^{-2 \pi i (m+\lambda 2^{-k} - \Gamma 2^{-k} ) (x-m)} e^{2 \pi i \Gamma 2^{-k} m} \\ &=& [ 2^{k-7}-1] 2^{-2k} \sum_{1 \leq m \leq N} ( \hat{\Phi}((x-m)2^k)))^2 e^{2 \pi i \Gamma 2^{-k} x}. 
 \end{eqnarray*}
By the assumption $\hat{\Phi}$ is real-valued with $\hat{\Phi}(0) >0$, $|T^k_m(f_1, f_2)(x)| \gtrsim 1_{[1,N]}(x)$ for all $ C_\Gamma \leq k \lesssim \log(N)$.  Lastly, by picking $\Gamma =100$, say, 

\begin{eqnarray*}
supp~ \mathcal{F}(T_m^k(f_1^N, f_2^N)) \subset [99\cdot 2^{-k}, 101\cdot 2^{-k}].
\end{eqnarray*}
Letting $1< p_1, p_2 <\infty$ satisfy $\frac{1}{p_1} + \frac{1}{p_2} <1$, note
 
 \begin{eqnarray*}
 \left| \left| T_{m} (f^N_1, f^N_2)  \right| \right|_{\frac{p_1 p_2}{p_1 + p_2}} &=&   \left| \left| \sum_{k \geq 8}T^k_m(f_1^N, f_2^N) \right| \right|_{\frac{p_1 p_2}{p_1 + p_2}}  \\ &\simeq&   \left| \left| \left( \sum_{k \geq 8} \left| T_m^k(f_1^N, f_2^N)\right|^2 \right)^{1/2} \right| \right|_{\frac{p_1 p_2}{p_1 + p_2}} \\ &\gtrsim& \log(N)^{1/2} N^{\frac{1}{p_1}+\frac{1}{p_2}}.
 \end{eqnarray*}
 However, $||f_i||_{p_i} \simeq N^{1/p_i}$ for $i \in \{1,2\}$, so taking N arbitrarily large finishes the theorem.

 \end{proof}

\section{Mixed Estimates for Hilbert Transform Type Multipliers}
The following theorem extends my result in \cite{MR3286565} from the case of $H(f_1, f_2) \mapsto H(f_1 \cdot f_2)$ to the class $\mathcal{M}_\Gamma (\mathbb{R}^2)$.
\begin{theorem}\label{PT1}
 Let $m  \in \mathcal{M}_\Gamma (\mathbb{R}^2)$. Then $T_m: W_{p_1} (\mathbb{R}) \times L^{p_2} (\mathbb{R}) \rightarrow L^{\frac{p_1 p_2}{p_1 + p_2}}(\mathbb{R})$ provided $\frac{1}{p_1} + \frac{1}{p_2} = \frac{1}{p}$, $1< p <\infty$, and $p_1 >2$.  
 
 \end{theorem}

\begin{proof}
We begin by recalling a basic fact: there is a Whitney decomposition of the region $\mathcal{R}:= \{\xi_1 + \xi_2 \geq 0\} \subset \mathbb{R}^2$ with boundary $\partial \mathcal{R} = \Gamma$ into disjoint squares $\left\{ Q_i \right\}_{i \in \mathbb{Z}}$ such that the Whitney property holds:

\begin{eqnarray*}
dist(Q_i , \Gamma) \simeq |Q_i|.
\end{eqnarray*}
 The precise constants in the above display are unimportant and therefore omitted. Our strategy is now to discretize $m \in \mathcal{M}_\Gamma(\mathbb{R}^2)$ according to the Whitney decomposition of $\mathcal{R}$, mollify the sum of characteristic functions of these disjoint Whitney squares, and then expand the product of each these mollified characteristic functions with the multiplier $m$ as a double Fourier series about the original frequency boxes. This enables us to rewrite $T_m$ in the following manner: 

\begin{eqnarray*}
T_m(f_1, f_2) &=& \int_{\mathbb{R}^2} m(\xi_1, \xi_2) \hat{f}_1(\xi_1) \hat{f}_2(\xi_2) e^{2 \pi i x (\xi_1 + \xi_2)} d \xi_1 d \xi_2 \\ &=& \int_{\mathbb{R}^2} (1_{\xi_1 + \xi_2 \geq 0} + 1_{\xi_1 + \xi_2< 0}) m(\xi_1, \xi_2) \hat{f}_1(\xi_1) \hat{f}_2(\xi_2) e^{2 \pi i x (\xi_1 + \xi_2)} d \xi_1 d \xi_2 \\ &=&\sum_{l_1, l_2 \in \mathbb{Z}}c_{l_1} d_{l_2}  \sum_{\vec{Q} \in \mathbb{Q}} \int_{\mathbb{R}^2} \eta_{Q_1}^{1,l_1}(\xi_1) \eta_{Q_2}^{2,l_2}(\xi_2) \hat{f}_1(\xi_1) \hat{f}_2(\xi_2) d \xi_1 d\xi_2 \\ &=& \sum_{l_1, l_2 \in \mathbb{Z}} c_{l_1} d_{l_2} \sum_{\vec{Q} \in \mathbb{Q}}  f_1* \eta_{Q_1}^{1, l_1} f_2 * \eta_{Q_2}^{2, l_2}
\end{eqnarray*}
where both sequences $\left\{ c_{l_1} \right\}_{l_1 \in \mathbb{Z}}$ and $\{ d_{l_2} \}_{l_2 \in \mathbb{Z}}$ are very rapidly decaying and $(\eta_{Q_1}^{1, l_1}, \eta_{Q_2}^{2, l_2})$ is uniformly adapted in frequency to $(Q_1, Q_2)$.
It therefore suffices to produce mixed estimates for $\sum_{\mathbb{Q} \in \mathbb{Q}}  f_1* \eta_{Q_1}^{1, l_1} f_2 * \eta_{Q_2}^{2, l_2}
$ uniform in the parameters $l_1, l_2$. To this end, we introduce another decomposition on top of the one already obtained by inserting a martingale structure coming from the Paley decomposition as found in \cite{MR1809116,MR3052499}. Define the distribution function $\gamma_{f_1} : \mathbb{R} \rightarrow (0,1)$ given by

\begin{eqnarray*}
\gamma_{f_1}(x) := \frac{ \int_{-\infty}^x |\hat{f}_1|^{p_1^\prime} d \bar{x} }{|| \hat{f}_1 ||_{p_1^\prime}^{p_1^\prime}}.
\end{eqnarray*}
By a simple limiting argument, we may assuming $\{ \hat{f}_1 = 0\} = \emptyset$, in which case we may partition the set of points $\xi_1 > - \xi_2$ based on the smallest dyadic interval that contains both $\xi_1$ and $-\xi_2$. Using $\gamma_{f_1}$ we may transfer the dyadic structure of $[0,1]$ back to $\mathbb{R}$ by taking preimages. So, for $m \geq 0$ and $0 \leq k \leq 2^m -2$ set 

\begin{eqnarray*}
E^m_k := \gamma_{f_1}^{-1} \left( [ 2^{-m} k , 2^{-m} (k+1)) \right)
\end{eqnarray*}
 with an obvious modification at the right end point corresponding to  $k = 2^m -1$.  Lastly, we construct 
 
 \begin{eqnarray*}
 E^m_{k,l} := \gamma_{f_1}^{-1} \left( [ 2^{-m} k , 2^{-m} (k+1/2)) \right)
 E^m_{k,r} := \gamma_{f_1}^{-1} \left( [ 2^{-m} (k+1/2) , 2^{-m} (k+1)) \right)
 \end{eqnarray*}
 with another obvious modification at the right end point of $E^m_{k,r}$ corresponding to $k = 2^m-1$. 
Hence, 

\begin{eqnarray*}
\left\{ (\xi_1, \xi_2) \in \mathbb{R}^2 : \xi_1 + \xi_2  \geq 0 \right\} &=& \coprod_{m \geq 0} \coprod_{k=0}^{2^m-1} -E^m_{k,l} \times E^m_{k,r} \bigcup \left\{ (\xi, -\xi)\in \mathbb{R}^2 : \xi \in \mathbb{R} \right\} \\ 
\left\{ (\xi_1, \xi_2) \in \mathbb{R}^2 : \xi_1 + \xi_2  < 0 \right\} &=& \coprod_{m \geq 0} \coprod_{k=0}^{2^m-1} E^m_{k,l} \times -E^m_{k,r}.
\end{eqnarray*}
 Armed with these two decompositions, we may rewrite $T_m(f_1, f_2)$ as indicated above and suppress the dependence on $l_1, l_2$ to obtain
\begin{eqnarray*}
T_m (f_1, f_2)&=&\sum_{m \geq 0} \sum_{j=0}^{2^m-1}  \sum_{\vec{Q} \in \mathbb{Q}^{m,1}_j}  f_1 * \check{1}_{-E^m_{j,l}}* \eta^1_{Q_1} f_2 * \check{1}_{E^m_{j,r}}*\eta^2_{Q_2} \\ &+& \sum_{m \geq 0} \sum_{j=0}^{2^m-1}  \sum_{\vec{Q} \in \mathbb{Q}^{m,2}_j}  f_1 * \check{1}_{E^m_{j,l}}* \eta^3_{Q_1} f_2 * \check{1}_{-E^m_{j,r}}*\eta^4_{Q_2} .
\end{eqnarray*}
WLOG, we may focus our attention exclusively on the first term. 
Then, we can decompose $\mathbb{Q}^{m,1}_{j} := \mathbb{Q}^m_j$ into $O(1)$ paraproducts. Upon dualizing, one can estimate the corresponding 3-form as follows:
\begin{eqnarray*}
\Lambda_{T_m ,1} (f_1, f_2, f_3) &:=& \int_\mathbb{R} T_m(f_1, f_2)\cdot  f_3~ dx\\ & =& \int_\mathbb{R} \sum_{m \geq 0} \sum_{j=0}^{2^m-1}  \sum_{\vec{Q} \in \mathbb{Q}^m_j}  f_1 * \check{1}_{-E^m_{j,l}}* \eta_{-Q} f_2 * \check{1}_{E^m_{j,r}}*\eta_{Q} f_3 * \psi_{|Q|} dx  \\ &=&  \int_\mathbb{R} \sum_{m \geq 0} \sum_{j=0}^{2^m-1}  \sum_{\vec{Q} \in \mathbb{Q}^{m,1}_j}  f_1 * \check{1}_{-E^m_{j,l}}* \eta_{-Q} f_2 * \check{1}_{E^m_{j,r}}*\eta_{Q} f_3 * \psi_{|Q|} dx\\ &+&  \int_\mathbb{R} \sum_{m \geq 0} \sum_{j=0}^{2^m-1}  \sum_{\vec{Q} \in \mathbb{Q}^{m,2}_j}  f_1 * \check{1}_{-E^m_{j,l}}* \eta_{-Q} f_2 * \check{1}_{E^m_{j,r}}*\eta_{Q} f_3 * \psi_{|Q|} dx \\ &:=& I + II. 
\end{eqnarray*}
The notation $\mathbb{Q}^{m,i}_j$ signifies that the set of cubes in question are lacunary in both in the ith and 3rd positions. Hence, the  first term can be satisfactorily estimated by

\begin{eqnarray*}
|I|  &\leq& \int_\mathbb{R} \sum_{m \geq 0} \sum_{j=0}^{2^m-1} \left(  \sum_{\vec{Q} \in \mathbb{Q}^{m,1}_j}  |f_1 * \check{1}_{-E^m_{j,l}}* \eta_{-Q}|^2 \right)^{1/2}  \sup_{\vec{Q} \in \mathbb{Q}^{m,1}_j} \left| f_2 * \check{1}_{E^m_{j,r}}*\eta_{Q} \right| \left( \sum_{ \vec{Q} \in \mathbb{Q}^{m,1}_j} |f_3 * \psi_{|Q|}|^2 \right)^{1/2} dx \\ &\leq&  \int_\mathbb{R} \sum_{m \geq 0} \left( \sum_{j=0}^{2^m-1} \sum_{\vec{Q} \in \mathbb{Q}^{m,1}_j}  |f_1 * \check{1}_{-E^m_{j,l}}* \eta_{-Q}|^2 \right)^{1/2} \left( \sum_{j=0}^{2^m-1} \sup_{\vec{Q} \in \mathbb{Q}^{m,1}_j} \left| f_2 * \check{1}_{E^m_{j,r}}*\eta_{Q} \right|^2 \right)^{1/2}  \left( \sum_{k\in \mathbb{Z}} |f_3 * \psi_{k}|^2 \right)^{1/2} dx.
\end{eqnarray*}
At this point, we can bring the sum over m outside the integral and then H\"{o}lderize. We may use vector-valued inequalities for CZO operators combined with the definition of the martingale structure to obtain the geometric decay $2^{m (1/2-1/p_1^\prime)} || \hat{f}_1~||_{p_1^\prime}$. The second factor can be handled using Fefferman-Stein's maximal inequality. The third requires a simple  Littlewood-Paley square function estimate. As the sum over all cubes in $\mathbb{Q}^{m,2}_j$ is similar to the sum over cubes in $\mathbb{Q}^{m,1}_j$, we omit the details. 
\end{proof}

\section{Mixed Estimates for Generic Degenerate Bi-Carleson Operators}
The next result extends mixed estimates to a maximal variant of generic bilinear degenerate simplex multipliers. \begin{theorem}\label{mhtp}
For $m\in \mathcal{M}_{\Gamma}(\mathbb{R}^2)$, construct the operator $\widetilde{\mathcal{M}H_m}$ where for each $x \in \mathbb{R}$, 
 
 \begin{eqnarray*}
 \widetilde{\mathcal{M}H_m} (f_1, f_2)(x) := \sup_{N \in \mathbb{R}} \left| \int_{ \xi_2 < \xi_1 < N}m(\xi_1, \xi_2) \hat{f}_1(\xi_1) \hat{f}_2(\xi_2) e^{2 \pi i x (\xi_1 - \xi_2)} d \xi_1 d \xi_2 \right|   . 
 \end{eqnarray*}
Then $\widetilde{\mathcal{M}H}: W_{p_1}(\mathbb{R}) \times L^{p_2}(\mathbb{R}) \rightarrow L^{\frac{p_1 p_2}{p_1 + p_2}} (\mathbb{R})$ provided $p_1 >2$ and $\frac{1}{p_1}+\frac{1}{p_2} <1$.

 \end{theorem}
 
 \begin{proof}
 It is important for the proof that $\xi_1$ is adjacent to the N over which the supremum is taken. This enables us to carve things a little more easily than in the other case. In fact, we shall reduce the study of the case when the function in the Wiener space is "opposite" the N to this first case using estimates for maximal paraproducts. Upon introducing two carvings, we have that

\begin{eqnarray*}
\widetilde{mH}(f_1, f_2)) =\sup_{N \in \mathbb{R}} \left|\sum_{ m_1, m_2 \geq 0}\sum_{j_1=0}^{2^{m_1}-1} \sum_{j_2 =0}^{2^{m_2}-1} \sum_{Q \in \mathbb{Q}^{m_1, m_2}_{j_1, j_2}} f_1 * \check{1}_{E^m_{j_1,r}}  * \check{1}_{E^{m_2}_{j_2, l}}* \eta_{-Q} \cdot f_2 * \check{1}_{-E^m_{j_1,l}}* \eta_Q \cdot 1_{\left\{N \in E^{m_2}_{j_2, r}\right\}} \right|.
\end{eqnarray*}
Dualizing with $g: || g||_{p^\prime}=1$ yields
\begin{eqnarray*}
 \int_\mathbb{R} \sum_{ m_1, m_2 \geq 0}\sum_{j_1=0}^{2^{m_1}-1} \sum_{j_2 =0}^{2^{m_2}-1} \sum_{Q \in \mathbb{Q}^{m_1}_{j_1}} f_1 * \check{1}_{E^m_{j_1,r}}  * \check{1}_{E^{m_2}_{j_2, l}}* \eta_{-Q}(x) \cdot f_2 * \check{1}_{-E^m_{j_1,l}}* \eta_Q(x) \cdot \left( 1_{\left\{N \in E^{m_2}_{j_2, r}\right\}} g\right)* \psi_{|Q|}(x)~dx.
\end{eqnarray*}
As before, we need to split $\mathbb{Q}^{m_1}_{j_1}$ into two disjoint collections labeled $\mathbb{Q}^{m_1, 1}_{j_1}$ and $\mathbb{Q}^{m_1, 2}_{j_1}$ where the first is lacunary in positions 1 and 3, and the second is lacunary with respect to positions 2 and 3. WLOG, we handle the sum only over $\mathbb{Q}^{m_1}_{j_1}$. This is done as follows:

\begin{eqnarray*}
&&  \int_\mathbb{R} \widetilde{m H} (f_1, f_2)(x) g (x)dx\\ &\lesssim& \int_\mathbb{R}\left[ \sum_{ m_1,m_2 \geq 0}\sum_{j_1=0}^{2^{m_1}-1} \sum_{j_2 =0}^{2^{m_2}-1} \left(\sum_{Q \in \mathbb{Q}^{m_1,1}_{j_1}} |f_1 * \check{1}_{E^m_{j_1,r}}  * \check{1}_{E^{m_2}_{j_2, l}}* \eta_{-Q}|^2 \right)^{1/2}  \sup_{Q \in \mathbb{Q}^{m_1,1}_{j_1}} | f_2 * \check{1}_{-E^m_{j_1,l}}* \eta_Q| \right. \\ &\times&  \left. \left( \sum_{ Q \in \mathbb{Q}^{m_1,1}_{j_1}} | (1_{\left\{N \in E^{m_2}_{j_2, r}\right\}} g)* \psi_{|Q|}|^2 \right)^{1/2}~dx\right].
\end{eqnarray*}
Again, we fix the scale, and apply Cauchy-Schwarz separately in each index $j_1, j_2$. Upon using H\"{o}lder's inequality, we proceed to estimate the three factors separately: 

The first takes the form 

\begin{eqnarray*}
\left| \left| \left( \sum_{j_1=0}^{2^{m_1}-1} \sum_{j_2 =0}^{2^{m_2}-1} \sum_{Q \in \mathbb{Q}^{m_1,1}_{j_1}} |f_1 * \check{1}_{E^m_{j_1,r}}  * \check{1}_{E^{m_2}_{j_2, l}}* \eta_{-Q}|^2 \right)^{1/2} \right| \right|_{\frac{p_1 p_2}{p_1 + p_2}} &\lesssim& \left| \left| \left( \sum_{j_1=0}^{2^{m_1}-1} \sum_{j_2 =0}^{2^{m_2}-1}  |f_1 * \check{1}_{E^m_{j_1,r}}  * \check{1}_{E^{m_2}_{j_2, l}}|^2 \right)^{1/2} \right| \right|_{\frac{p_1 p_2}{p_1 + p_2}} \\ &\lesssim& 2^{\max\{m_1, m_2\} (1/2-1/p_1^\prime)} || \hat{f}_1 ~||_{p_1^\prime}.
\end{eqnarray*}
The second factor has the same bound as before, i.e. $|~II~| \lesssim 2^{ m_1 (\max \{ 0, 1/p_2 -1/2\})}||f_2||_{p_2}$. Lastly, we  estimate 

\begin{eqnarray*}
 \left| \left| \left(\sum_{j_2=0}^{2^{m_2}-1} \sum_{ Q \in \mathbb{Q}^{m_1,1}_{j_1}} | (1_{\left\{N \in E^{m_2}_{j_2, r}\right\}} g)* \psi_{|Q|}|^2 \right)^{1/2} \right| \right|_{p^\prime}& \leq& \left| \left| \left(\sum_{j_2=0}^{2^{m_2}-1} \sum_{k \in \mathbb{Z}} | (1_{\left\{N \in E^{m_2}_{j_2, r}\right\}} g)* \psi_{k}|^2 \right)^{1/2} \right| \right|_{p^\prime}  \\ &\simeq& \left| \left| \left(\sum_{k \in \mathbb{Z}} \left( \mathbb{E}_t \left|\left( \sum_{j_2=0}^{2^{m_2}-1} r_{j_2}(t) 1_{\left\{N \in E^{m_2}_{j_2, r}\right\}} g\right)* \psi_{k} \right| \right)^2 \right)^{1/2} \right| \right|_{p^\prime} \\ &\leq&\mathbb{E}_t \left| \left| \left(\sum_{k \in \mathbb{Z}}  \left|\left( \sum_{j_2=0}^{2^{m_2}-1} r_{j_2}(t) 1_{\left\{N \in E^{m_2}_{j_2, r}\right\}} g\right)* \psi_{k} \right|^2 \right)^{1/2} \right| \right|_{p^\prime} \\ &\lesssim& \mathbb{E}_t \left| \left|  \sum_{j_2=0}^{2^{m_2}-1} r_{j_2}(t) 1_{\left\{N \in E^{m_2}_{j_2, r}\right\}} g  \right| \right|_{p^\prime} \\ &\leq& \left| \left| g \right| \right|_{p^\prime} \leq 1. 
\end{eqnarray*}
Therefore, it suffices to bound $\sum_{m_1, m_2 \geq 0} 2^{\max\{m_1, m_2\} (1/2-1/p_1^\prime)} 2^{ m_1 (\max \{ 0, 1/p_2 -1/2+\epsilon \})}$. To this end, split the sum according to $\max\{m_1, m_2 \}$ and compute for the first part

\begin{eqnarray*}
\sum_{ 0 \leq m_1 \leq m_2} 2^{\max\{m_1, m_2\} (1/2-1/p_1^\prime)} 2^{ m_1 (\max \{ 0, 1/p_2 -1/2+\epsilon\})} &=& \sum_{0 \leq m_1 \leq m_2} 2^{m_2( 1/2-1/p_1^\prime)} 2^{ m_1 (\max \{ 0, 1/p_2 -1/2+\epsilon\})}  \\ &\lesssim& \sum_{0 \leq m_2  } m_22^{m_2 (\max\{1/2, 1/p_2+\epsilon\} + 1/p_1 -1)} \lesssim_{\epsilon} 1
\end{eqnarray*}
for sufficiently small $\epsilon(p_1, p_2)>0$. 
For the second sum, note

\begin{eqnarray*}
\sum_{ m_1 > m_2 \geq 0 } 2^{m_1 (1/2-1/p_1^\prime)} 2^{ m_1 (\max \{ 0, 1/p_2 -1/2+\epsilon\})} = \sum_{m_1 \geq 0} m_1 2^{m_1 (1/2-1/p_1^\prime)} 2^{ m_1 (\max \{ 0, 1/p_2 -1/2\})}  \lesssim 1
\end{eqnarray*} 
for sufficiently small $\epsilon(p_1, p_2)>0$. 
 \end{proof}
 \begin{corollary}
Let the operator $\mathcal{M}H_1$ be given on functions $(f_1, f_2) \in \mathcal{S}(\mathbb{R})^2$ by the formula
 
 \begin{eqnarray*}
 \mathcal{M}H_1 (f_1, f_2)(x) := \sup_{N \in \mathbb{R}} \left| \int_{\xi_1 < \xi_2 < N} \hat{f}_1(\xi_1) \hat{f}_2(\xi_2) e^{2 \pi i x (\xi_1 - \xi_2)} d \xi_1 d \xi_2 \right|   . 
 \end{eqnarray*}
Then $\mathcal{M}H_1: W_{p_1}(\mathbb{R}) \times L^{p_2}(\mathbb{R}) \rightarrow L^{\frac{p_1 p_2}{p_1 + p_2}} (\mathbb{R})$ provided $p_1 >2$ and $\frac{1}{p_1}+\frac{1}{p_2} <1$. 

\end{corollary}
\begin{proof}
By the triangle inequality,

\begin{eqnarray*}
\mathcal{M}H_1(f_1, f_2)\leq   C(f_1)\cdot C(f_2) + \widetilde{\mathcal{M}H_1}(f_1, f_2).
\end{eqnarray*}
It suffices to apply the $L^p$ triangle inequality, H\"{o}lder's inequality, the Carleson estimates, and Theorem \ref{mhtp}.

\end{proof}

\section{Mixed Estimates for $sgn(\xi_1 + \xi_2) sgn(\xi_2 + \xi_3)$}
Before moving on to the case of generic trilinear degenerate simplex multipliers, it is instructive to prove mixed estimates using more elementary methods in a few special trilinear cases. In this section, we quickly reprove a theorem from \cite{MR3286565} before generalizing this result in \S{6}. 
 \begin{theorem}\label{OT*}
 $T^{sgn, sgn}$ be defined initially on functions $(f_1, f_2, f_3) \in \mathcal{S}^3(\mathbb{R})$ by 

\begin{eqnarray*}
T^{sgn, sgn}(f_1, f_2, f_3)(x) := \int_{\mathbb{R}^3} sgn(\xi_1 + \xi_2) sgn(\xi_2 + \xi_3) \hat{f}_1(\xi_1) \hat{f}_2(\xi_2) \hat{f}_3(\xi_3) e^{2 \pi i x (\xi_1 + \xi_2 + \xi_3)} d\vec{\xi}. 
\end{eqnarray*}
Then $T^{sgn, sgn}$ is a bounded operator from $L^{p_1}(\mathbb{R}) \times W_{p_2}(\mathbb{R}) \times L^{p_3}(\mathbb{R})$ into $L^{\frac{1}{\frac{1}{p_1} + \frac{1}{p_2} + \frac{1}{p_3}}}(\mathbb{R})$ provided the following exponent conditions hold: $\frac{1}{p_1} + \frac{1}{p_2} <1, \frac{1}{p_2} + \frac{1}{p_3} <1, \frac{1}{p_1} + \frac{1}{p_2} + \frac{1}{p_3}  <\frac{3}{2}, 1 < p_1 < \infty, 2 < p_2 \leq \infty,$ and $ 1 < p_3 < \infty$.
 \end{theorem}
 Remark: In fact, using a more complicated argument, one can show that $\frac{1}{p_1} + \frac{1}{p_2} + \frac{1}{p_3} < \frac{3}{2}$ is not necessary for the above mixed estimates to hold. Theorem \ref{OT*} is true even if one removes the requirement $\frac{1}{p_1} + \frac{1}{p_2} + \frac{1}{p_3} <1$ altogether. 
 \begin{proof}
 Introduce the function $\mu_{f_2} : \mathbb{R} \rightarrow (0,1)$ given by
 
 \begin{eqnarray*}
 \mu_{f_2}(x) := \frac{ \int_{-\infty}^x |\hat{f}_2(\bar{x})|^{p_2^\prime} d \bar{x}}{|| \hat{f}_2 ||_{p_2^\prime}^{p_2^\prime}}
 \end{eqnarray*}
 and construct the following family of sets: for each $m \geq 0$ and $0 \leq k \leq 2^m-1$, 
 
 \begin{eqnarray*}
 E^m_k &=& \mu^{-1}_{f_2} ( [ 2^{-m} k, 2^{-m} (k+1)))\\ 
 E^m_{k,left} &=& \mu^{-1}_{f_2} ([2^{-m} k, 2^{-m} (k+1/2))) \\ 
 E^m_{k, right}&=& \mu^{-1}_{f_2} ([2^{-m} (k+1/2), 2^{-m}(k+1))).
 \end{eqnarray*}
Just as before, 
 
 \begin{eqnarray*}
\mathbb{R}^2 \supset \left\{ \xi_1 + \xi_2 \geq 0\right \} = \coprod_{m \geq 0} \coprod_{k=0}^{2^m-1} -E^m_{k, left} \times E^m_{k, right} \bigcup \left\{ \xi, - \xi) \in \mathbb{R}^2 : \xi \in \mathbb{R} \right\}.
 \end{eqnarray*}
 Hence, modulo harmless difference terms, which satisfy the desired estimates, 
 
 \begin{eqnarray*}
 T^{sgn, sgn} (f_1, f_2, f_3)(x) \simeq \sum_{m, m^\prime \geq 0 } \sum_{k=0}^{2^m-1} \sum_{l=0}^{2^{m^\prime}-1} f_1 * \check{1}_{-E^m_{k, left}}(x) \cdot  f_2 * \check{1}_{E^m_{k, right}} * \check{1}_{-E^{m^\prime}_{l, left}}(x)\cdot f_3*\check{1}_{E^{m^\prime}_{l, right}}(x).
 \end{eqnarray*}
 
 Next, perform Cauchy-Schwarz in both $k,l$ for fixed $m, m^\prime$. If $\frac{1}{p_1} + \frac{1}{p_2} + \frac{1}{p_3} \leq 1$, then we may use the triangle inequality to pull out the sum over $m, m^\prime$. If $\frac{1}{p_1} + \frac{1}{p_2} + \frac{1}{p_3} >1$, then set $p =\frac{1}{\frac{1}{p_1} + \frac{1}{p_2} + \frac{1}{p_3}}$ and use
 
 \begin{eqnarray*}
&& \left| \left| T^{sgn, sgn} (f_1, f_2, f_3) \right| \right|_{L^{p}(\mathbb{R})}^{p}\\& \leq& \sum_{m, m^\prime \geq 0} \left| \left| \left( \sum_{k=0}^{2^m -1} \left|  f_1 * \check{1}_{-E^m_{k, left}} \right|^2 \right)^{1/2} \left( \sum_{k=0}^{2^m-1} \sum_{l=0}^{2^{m^\prime-1}} \left|  f_2 * \check{1}_{E^m_{k, right}} * \check{1}_{-E^{m^\prime}_{l, left}} \right|^2 \right)^{1/2} \right. \right. \\ &&\left. \left. \times \left( \sum_{l =0}^{2^{m^\prime}-1} \left|  f_3*\check{1}_{E^{m^\prime}_{l, right}} \right|^2 \right)^{1/2} \right| \right|_{L^p(\mathbb{R})} ^p.
 \end{eqnarray*}
 In the quasi-Banach case, one may apply the generalized H\"{o}lder's inequality, generalized Rubio de Francia estimates, and the martingale structure  mass decomposition to bound the last expression from above by
 
 \begin{eqnarray*}
\left[  \sum_{m ,m^\prime \geq 0} \left| 2^{m \max\{ 0 , 1/p_1 - 1/2 + \epsilon\}} 2^{\max\{ m, m^\prime \} (1/2 - 1/p_1^\prime)} 2^{ m^\prime \max\{ 0, 1/p_3 - 1/2 + \epsilon\}} \right|^p \right] ||f_1||_{p_1}^p ||\hat{f}_2||_{p_2^\prime}^{p} ||f_3||_{p_3}^p.
 \end{eqnarray*}
 As before, we split the sum into two parts corresponding to $\sum_{m, m^\prime \geq 0} = \sum_{0 \leq m \leq m^\prime} + \sum_{0 \leq m^\prime < m}$. This time, it is easy to see that both sums are summable for small enough $\epsilon (p_1, p_2, p_3)>0$ by the assumption $\frac{1}{p_1} + \frac{1}{p_2} + \frac{1}{p_3} <1$.

 \end{proof}

 \section{Mixed Estimates for Multipliers of $sgn(\xi_1 + \xi_2) a(\xi_2, \xi_3)$ Type}
 We now prove mixed estimates for hybrid trilinear degenerate simplex multipliers. 
 \begin{theorem}
 Let $a \in \mathcal{M}_{\Gamma}(\mathbb{R}^2)$ and construct the trilinear operator $T^{sgn, a}$ on $\mathcal{S}^3(\mathbb{R})$ given by 
 \begin{eqnarray*} 
 T^{sgn, a}(f_1, f_2, f_3)(x) := \int_{\mathbb{R}^3} sgn(\xi_1 + \xi_2) a(\xi_2, \xi_3) \hat{f}_1(\xi_1) \hat{f}_2(\xi_2) \hat{f}_3(\xi_3) e^{2 \pi i x (\xi_1 + \xi_2 + \xi_3)} d\xi_1 d \xi_2 d \xi_3.
 \end{eqnarray*}
 Then $T^{sgn, a}$ maps $L^{p_1}(\mathbb{R}) \times W_{p_2}(\mathbb{R}) \times L^{p_3}(\mathbb{R}) \rightarrow L^{\frac{1}{\frac{1}{p_1} + \frac{1}{p_2} + \frac{1}{p_3}}}(\mathbb{R})$ provided $1<p_1 < \infty, 2 < p_2 \leq \infty, 1 < p_3 < \infty,$ and $\frac{1}{p_1} + \frac{1}{p_2} + \frac{1}{p_3} <1$.

 \end{theorem}
 
 \begin{proof}
 
We begin by noting 

\begin{eqnarray*}
m(\xi_1, \xi_2, \xi_3) = sgn(\xi_1 + \xi_2) a(\xi_2, \xi_3) = sgn(\xi_2 + \xi_3) a(\xi_2, \xi_3) (1_{\xi_2 + \xi_3 \leq 0} + 1_{\xi_2 + \xi_3 >0}) = a_I + a_{II}. 
\end{eqnarray*}
It will suffice to bound $a_I$. First set up two different carvings for $1_{\xi_2 <- \xi_3}$. Introduce

\begin{eqnarray*}
a(\xi_2, \xi_3) 1_{\{ \xi_2+ \xi_3 \leq 0\}} (\xi_2, \xi_3) &=& \sum_{(\sigma , \sigma^\prime) \in \{ 0, \frac{1}{3}, \frac{2}{3} \}^2} \sum_{\kappa, \kappa^\prime} \sum_{\vec{P} \in \mathbb{P}} c(a)_{\kappa} \tilde{c}(a)_{\kappa^\prime} \hat{\eta}_{P_1, 1} ^{\sigma, \kappa } (\xi_2) \hat{\eta}_{P_2, 2} ^{\sigma^\prime, \kappa^\prime}(\xi_3)~~(a.e) .
\end{eqnarray*}
It follows that

\begin{eqnarray*}
&&\Lambda_{T^{sgn, a}} (f_0, f_1, f_2, f_3) \\&:=& \int_\mathbb{R} T^{sgn, a}(f_0, f_1, f_2) (x) f_3(x) dx \\ &=&  \sum_{(\sigma , \sigma^\prime) \in \{ 0, \frac{1}{3}, \frac{2}{3} \}^2} \sum_{m, m^\prime}  \sum_{k=0}^{2^m-1} \sum_{k^\prime=0}^{2^{m^\prime}-1} \sum_{\kappa, \kappa^\prime \in \mathbb{Z}} \sum_{\vec{P} \in \mathbb{P}} c_{\kappa} \tilde{c}_{\kappa^\prime} \\ &\times&  \int_\mathbb{R} f_0 f_1 *\check{1}_{-E^m_{k, left}} f_2* \check{1}_{E^m_{k, right}} * \eta_{P_1,1}^{\sigma , \kappa} * \check{1}_{E^{m^\prime}_{l, left}} f_3* \eta_{P_2,2}^{\sigma^\prime, \kappa^\prime} * \check{1}_{-E^{m^\prime}_{l, right}} dx .
\end{eqnarray*}
Because of rapid coefficient decay, it suffices to prove satisfactory estimates for 

\begin{eqnarray*}
\left| \sum_{m, m^\prime}  \sum_{k,l}  \sum_{P \in \mathbb{P}^{m^\prime}_{l}} \int_\mathbb{R} f_0 f_1 *\check{1}_{-E^m_{k, left}} f_2* \check{1}_{E^m_{k, right}} * \eta_{P_1,1}^{\sigma , \kappa ,\lambda} * \check{1}_{E^{m^\prime}_{l, left}} f_3* \eta_{P_2,2}^{\sigma^\prime, \kappa^\prime, \lambda^\prime} * \check{1}_{-E^{m^\prime}_{l, right}} dx \right| 
\end{eqnarray*}
with a bound that is independent of $\sigma, \sigma^\prime, \kappa, \kappa^\prime $.  To this end, set $\mathbb{P}^{m^\prime}_l = \left\{ \vec{P} \in \mathbb{P} : (P_1 , P_2)  \cap (-E^{m^\prime}_{l, left}, E^{m^\prime}_{l, right} ) \not = \emptyset \right\}$.
We organize our collection of $\mathbb{P}^{m^\prime}_l$ into $O(1)$ disjoint  paraproducts, which are either lacunary in the first index, in which case we say the paraproduct is type $A$, or lacunary in the 2nd index, in which case we say the paraproduct is type $B$. For notational simplicity, we write $\eta_{P_1, 1}^{\sigma, \kappa}$ simply as $\eta_{P_1, 1}$.

  \subsection{Estimates for $\mathbb{P}^{m^\prime}_l[A]$}
Let $\mathbb{P}^{m^\prime}_l [A] \subset \mathbb{P}^{m^\prime}_l$ be a type $A$ praraproduct. Then we may majorize $\left| \Lambda^{m,m^\prime,\mathbb{P}[A]}_{T^{sgn, a}}(f_0, f_1, f_2, f_3)\right|$ by a rapidly decaying sum over expressions of the form
    
    \begin{eqnarray*}
   && \sum_{m, m^\prime \geq 0} \left| \int_\mathbb{R} \sum_{k,l} \sum_{\vec{P} \in \mathbb{P}^{m^\prime}_l[A]} ( f_0 f_1* \check{1}_{-E^m_{k, left}}) * \psi^{lac, \vec{P}}_{|P|} f_2* \check{1}_{E^m_{k, right}} * \check{1}_{-E^{m^\prime}_{l, left}}*\eta_{P_1,1} \cdot f_3* \check{1}_{E^{m^\prime}_{l, right}} ~dx\right| \\ &\leq&\sum_{m, m^\prime \geq 0}  \int_\mathbb{R} \left( \sum_k  \sum_{\vec{P} \in \mathbb{P}^{m^\prime}_l[A]} \left|  (f_0 f_1* \check{1}_{-E^m_{k, left}}) * \psi^{lac, \vec{P}}_{|P|} \right|^2 \right)^{1/2} \left( \sum_{k,l} \sum_{\vec{P} \in \mathbb{P}^{m^\prime}_l[A]} \left|  f_2* \check{1}_{E^m_{k, right}} * \check{1}_{-E^{m^\prime}_{l, left}}*\eta_{P_1}  \right|^2 \right)^{1/2} \\ &\times& \left( \sum_l   \left| f_3* \check{1}_{E^{m^\prime}_{l, right}}  \right|^2 \right)^{1/2}~dx.
   \end{eqnarray*}
   
We may now H\"{o}lderize, use the vector-valued CZO estimate and generalized Rubio de Francia estimate for the first factor, use the martingale structure to extract exponential decay over $m, m^\prime$ from the second factor, and apply the generalized Rubio de Francia estimate for the third factor. 
 

  \subsection{Estimates for $\mathbb{P}^{m^\prime}_l[B]$}
   Let $\mathbb{P}^{m^\prime}_l [B] \subset \mathbb{P}^{m^\prime}_l$ be a collection of type $B$ frequencies. Then we may majorize $\left| \Lambda^{m,m^\prime,\mathbb{P}[B]}_{T^{sgn, a}}(f_0, f_1, f_2, f_3) \right|$ by a rapidly decaying sum over expressions of the form

   \begin{eqnarray*}
  &&\left| \int_\mathbb{R} \sum_{k,l} \sum_{\vec{P} \in \mathbb{P}^{m^\prime}_l[B]} ( f_0 f_1* \check{1}_{-E^m_{k, left}}) * \psi^{lac, \vec{P}}_{|P|} f_2* \check{1}_{E^m_{k, right}} * \check{1}_{-E^{m^\prime}_{l, left}} f_3* \check{1}_{E^{m^\prime}_{l, right}} * \eta_{P_2,2} ~dx\right|  \\ &\leq& \int_\mathbb{R} \left( \sum_k  \sum_{\vec{P} \in \mathbb{P}^{m^\prime}_l[B]} \left|(  f_0 f_1* \check{1}_{-E^m_{k, left}}) * \psi^{lac, \vec{P}}_{|P|} \right|^2 \right)^{1/2} \left( \sum_{k,l} \left|  f_2* \check{1}_{E^m_{k, right}} * \check{1}_{-E^{m^\prime}_{l, left}}  \right|^2 \right)^{1/2} \\ &\times& \left( \sum_l  \sum_{\vec{P} \in \mathbb{P}^{m^\prime}_l[B]} \left| f_3* \check{1}_{E^{m^\prime}_{l, right}} * \eta_{P_2,2} \right|^2 \right)^{1/2}~dx.
   \end{eqnarray*}

We may now H\"{o}lderize, use the vector-valued CZO estimate and generalized Rubio de Francia estimate for the third factor, use the martingale structure to extract exponential decay over $m, m^\prime$ from the second factor, and apply the generalized Rubio de Francia estimate for the first factor.

\end{proof}

 \section{$WLW$-Type Estimates for $B[ a_1,a_2]$}

 \begin{theorem}\label{PT2}

Let $a_1,a_2 \in \mathcal{M}_{\Gamma}(\mathbb{R}^2)$ and construct the trilinear operator $T^{a_1, a_2}$ on $\mathcal{S}^3(\mathbb{R}) $ given by

\begin{eqnarray*}
B[a_1, a_2] (f_1, f_2, f_3) (x) := \int_{\mathbb{R}^3} a_1(\xi_1, \xi_2) a_2(\xi_2, \xi_3) \hat{f}_1(\xi_1) \hat{f}_2(\xi_3) \hat{f}_3(\xi_3) e^{2 \pi i x (\xi_1+ \xi_2 + \xi_3)} d \xi_1 d\xi_2 d \xi_3.
\end{eqnarray*}
Then $B[a_1, a_2]$ extends to a bounded map from $W_{p_1}(\mathbb{R}) \times L^{p_2}(\mathbb{R}) \times W_{p_3}(\mathbb{R})$ into $ L^{\frac{1}{\frac{1}{p_1} + \frac{1}{p_2} + \frac{1}{p_3}}}(\mathbb{R})$ for all $\frac{1}{p_1} + \frac{1}{p_2}+\frac{1}{p_3} <1$ and $p_1, p_3 >2$. Specifically,  $T^{a_1, a_2}$ can be defined on all $(f_1, f_2, f_3)$ such that $\hat{f}_1 \in L^{p_1^\prime}(\mathbb{R}), f_2 \in L^{p_2}(\mathbb{R})$ and $\hat{f}_3 \in L^{p^\prime_3}(\mathbb{R})$ in such a way that
\begin{eqnarray*}
|| B[a_1, a_2](f_1, f_2, f_3) ||_{L^{\frac{1}{\sum \frac{1}{p_1} +\frac{1}{p_2} + \frac{1}{p_3}}}(\mathbb{R})} \lesssim _{p_1, p_2, p_3} ||\hat{f}_1||_{L^{p^\prime_1}(\mathbb{R})} ||f_2||_{L^{p_2}(\mathbb{R})} || \hat{f}_3||_{L^{p_3^\prime}(\mathbb{R})}.
\end{eqnarray*}

\end{theorem}

 \begin{proof}
 As usual, we use a Christ-Kiselev-Paley decomposition, Whitney decomposition for $\Gamma = \{\xi_1 + \xi_2 = 0\}$, and  the generalized Rubio de Francia inequality.  The interaction of the mass decomposition of  $\hat{f}_3$ with the geometric decomposition of $\Gamma$ can be handled via vector-valued inequalities for C-Z operators. 
 
 To begin, carve 
 
 \begin{eqnarray*}
&& a_1 (\xi_1, \xi_2) a_2(\xi_2, \xi_3) \\&=& a_1(\xi_1, \xi_2) a_2(\xi_2, \xi_3) (1_{0< \xi_1 +\xi_2} 1_{\xi_2 + \xi_3 \leq 0} +1_{0 <\xi_2 +\xi_3} 1_{\xi_1 + \xi_2 \leq 0} +1_{0 <\xi_2 + \xi_3 }  1_{0 <\xi_1 + \xi_2 } + 1_{\xi_2+\xi_3 \leq0} 1_{0<\xi_1 + \xi_2 }) \\ &=& a_I + a_{II} + a_{III} + a_{IV}. 
 \end{eqnarray*}
 It will suffice to bound $a_I$.  Now, set up two carvings for $1_{-\xi_1 < \xi_2}$. First, introduce the function $\mu_{f_1} : \mathbb{R} \rightarrow (0, 1)$ given by 
 
 \begin{eqnarray*}
 \mu_{f_1}(x) := \frac{ \int_{-\infty}^x  | \hat{f}_1(\bar{x}) |^{p_1^\prime}d \bar{x} }{||\hat{f}_1||_{p_1^\prime}^{p_1^\prime}}
 \end{eqnarray*}
 and construct the following family of sets: for each $m \in \mathbb{N}^+ \bigcup \{ 0 \}$ and $0 \leq k \leq 2^m-1$, 
 
 \begin{eqnarray*}
 F^m_k&:=& \mu_{f_1}^{-1} \left( [2^{-m} k, 2^{-m} (k+1))\right) \\ 
  F^m_{k, left} &:=& \mu_{f_1}^{-1} \left( [2^{-m} k, 2^{-m} (k+1/2))\right) \\ 
    F^m_{k, right} &:=& \mu_{f_1}^{-1} \left( [2^{-m} (k+1/2), 2^{-m} (k+1))\right).
 \end{eqnarray*}
 Next,  introduce the function $\gamma_{f_3} : \mathbb{R} \rightarrow (0, 1)$ given by 
 
 \begin{eqnarray*}
 \gamma_{f_3}(x) := \frac{ \int_{-\infty}^x  | \hat{f}_3(\bar{x}) |^{p_3^\prime}d \bar{x} }{||\hat{f}_3||_{p_3^\prime}^{p_3^\prime}}
 \end{eqnarray*}
 and construct the following family of sets: for each $m \in \mathbb{N}^+ \bigcup \{ 0 \}$ and $0 \leq k \leq 2^m-1$, 
 
 \begin{eqnarray*}
 E^m_k&:=& \gamma_{f_3}^{-1} \left( [2^{-m} k, 2^{-m} (k+1))\right) \\ 
  E^m_{k, left} &:=& \gamma_{f_3}^{-1} \left( [2^{-m} k, 2^{-m} (k+1/2))\right) \\ 
    E^m_{k, right} &:=& \gamma_{f_3}^{-1} \left( [2^{-m} (k+1/2), 2^{-m} (k+1))\right).
 \end{eqnarray*}
 Hence, for each $m \geq 0$, $\mathbb{R} = \coprod_{k=0}^{2^m-1} E^m_k$. Moreover, $\{ \vec{\xi} \in \mathbb{R}^2 : -\xi_1 <  \xi_2\}  = \coprod_{m =0}^{\infty} \coprod_{k=0}^{2^m-1} -E^m_{k, left} \times E^m_{k, right}.$ 
Furthermore,  we choose of Whitney decomposition for $\{ -\xi_1 <  \xi_2\} \subset \mathbb{R}^2$, namely $\{ P_j\}_{j \in \mathbb{Z}}$ satisfying for every $j \in \mathbb{Z}$
 
 \begin{eqnarray*}
| side(P_j)| \simeq dist(P_j, \Gamma_2).
\end{eqnarray*}
Performing the same tricks as in the standard discretization of the BHT (adapted instead to the degenerate line $\Gamma_2$) produces 

\begin{eqnarray*}
a_1(\xi_1, \xi_2) 1_{\{ -\xi_1 < \xi_2\}}(\xi_1, \xi_2) = \sum_{(\gamma, \gamma^\prime) \in \{ 0, \frac{1}{3}, \frac{2}{3} \}^2} \sum_{k, k^\prime \in \mathbb{Z}} \sum_{\vec{Q} \in \mathbb{Q}^{\gamma, \gamma^\prime}} c_k \tilde{c}_{k^\prime} \hat{\eta}_{Q_1, 1}^{\gamma, k} (\xi_1) \hat{\eta}_{Q_2, 2}^{\gamma^\prime , k^\prime} (\xi_2), 
\end{eqnarray*}
where the sequences $\left\{ c_m \right\}_{m \in \mathbb{Z}}, \left\{ \tilde{c}_m \right\}_{\tilde{m} \in \mathbb{Z}}$ are both rapidly decaying. 
Similarly, we have
\begin{eqnarray*}
a_2(\xi_2, \xi_3) 1_{\{ \xi_2 < - \xi_3\}} (\xi_2, \xi_3) &=& \sum_{(\sigma , \sigma^\prime) \in \{ 0, \frac{1}{3}, \frac{2}{3} \}^2} \sum_{l, l^\prime \in \mathbb{Z}} \sum_{\vec{P} \in \mathbb{P}^{\sigma, \sigma^\prime}} d_l \tilde{d}_{l^\prime}  \hat{\eta}_{P_1, 1} ^{\sigma, l } (\xi_2) \hat{\eta}_{P_2, 2} ^{\sigma^\prime, l^\prime}(\xi_3).
\end{eqnarray*}
Suppressing dependence on $\sigma, \sigma^\prime, \gamma, \gamma^\prime, k, k^\prime, l , l^\prime$, it suffices to show estimates for the form $\Lambda_{T^{a_1, a_2}}$ given by

\begin{eqnarray*}
&& \Lambda_{B[a_1, a_2]}(f_0, f_1, f_2, f_3) \\&=&   \sum_{\vec{Q} \in \mathbb{Q}} \sum_{\vec{P} \in \mathbb{P}} \int_\mathbb{R}f_0 \cdot f_1*\eta_{Q_1, 1} \cdot f_2*\eta_{Q_2, 2} * \eta_{P_1, 1} \cdot  f_3 * \eta_{P_2,2}~dx   \\ &=& \sum_{m \geq 0} \sum_{k=0}^{2^m-1} \sum_{(\vec{Q}, \vec{P}) : |\vec{Q}| > |\vec{P}|}  \int_\mathbb{R}f_0 \cdot f_1*\eta_{Q_1, 1} \cdot f_2*\eta_{Q_2, 2} * \eta_{P_1, 1} * \check{1}_{E^m_{k, left}} \cdot  f_3 * \eta_{P_2,2} * \check{1}_{-E^m_{k, right}}~dx\\ &+&  \sum_{m \geq 0} \sum_{k=0}^{2^m-1} \sum_{(\vec{Q}, \vec{P}) : |\vec{Q}| \leq |\vec{P}|}  \int_\mathbb{R}f_0 \cdot f_1*\eta_{Q_1, 1}*\check{1}_{-F^m_{k, left}}  \cdot f_2*\eta_{Q_2, 2} * \eta_{P_1, 1} * \check{1}_{F^m_{k,right}} \cdot  f_3 * \eta_{P_2,2} ~dx \\ &=& \Lambda_{B[a_1, a_2] _1}(f_0, f_1, f_2, f_3) + \Lambda_{B[a_1, a_2]_2}(f_0, f_1, f_2, f_3). 
\end{eqnarray*}
By symmetry and Camil's estimates, it suffices to prove there exists $C >0$ such that for all $ m \geq 0$

\begin{eqnarray*}
&&\Lambda^m_{B[a_1, a_2]_1}(f_0, f_1, f_2, f_3) \\&:=&\left|  \sum_{k=0}^{2^m-1} \sum_{(\vec{Q}, \vec{P}) : \vec{P} \in \mathbb{P}_k^m,  |\vec{Q}| >> |\vec{P}|}  \int_\mathbb{R}f_0 \cdot f_1*\eta_{Q_1, 1} \cdot f_2*\eta_{Q_2, 2} * \eta_{P_1, 1} * \check{1}_{E^m_{k, left}} \cdot  f_3 * \eta_{P_2,2} * \check{1}_{-E^m_{k, right}}~dx \right|\\ & \lesssim_{p_1, p_2, p_3} & 2^{-Cm} || f_1||_{p_1} ||f_2||_{p_2} || \hat{f}_3||_{p_3^\prime}.
\end{eqnarray*}
To this end, rewrite $\Lambda^m_{B[a_1, a_2]_1}(f_0, f_1, f_2, f_3)$ as

\begin{eqnarray*}
 \sum_{k=0}^{2^m-1} \sum_{(\vec{Q}, \vec{P}) : \vec{P} \in \mathbb{P}_k^m,  |\vec{Q}| > |\vec{P}|}  \int_\mathbb{R}( f_0 * \eta_{-\tilde{Q}_1,0} \cdot f_1*\eta_{Q_1, 1} ) * \psi_{|P|} \cdot f_2*\eta_{Q_2, 2} * \eta_{P_1, 1} * \check{1}_{E^m_{k, left}} \cdot  f_3 * \eta_{Q_3, 3} * \eta_{P_2,2} * \check{1}_{-E^m_{k, right}}~dx.  
\end{eqnarray*}
Recall that a paraproduct $\Pi = \{\vec{P} \}$ is type $A$ provided $\{P_1\}_{\vec{P} \in \Pi}$ is lacunary, $\Pi = \left\{ \vec{P} \right\}$ is Type $B$ provided $\left\{ P_2 \right\}_{\vec{P} \in \Pi}$ is lacunary. For each $m \geq 0$ and $0 \leq k < 2^m -1$, let $\mathbb{P}^m_k =\left \{ (P_1, P_2) \in \mathbb{P}: \vec{P} \cap \left(E^m_{k, keft} \times-E^m_{k, right} \right)\not = \emptyset \right\}$. We may clearly split each $\mathbb{P}^m_k$ into $O(1)$ disjoint collections of paraproducts of types $A$ and $B$. Hence, there is a splitting of $\tilde{\Lambda}_{T_1^{a_1, a_2}}^m(f_0, f_1, f_2, f_3)$ into a sum of two terms in the obvious way so that $|\tilde{\Lambda}_{T_1^{a_1, a_2}}^m(f_0, f_1, f_2, f_3)| \leq | \tilde{\Lambda}_{T_1^{a_1, a_2}}^{m,A}(f_0, f_1, f_2, f_3) | +| \tilde{\Lambda}_{T_1^{a_1, a_2}}^{m,B}(f_0, f_1, f_2, f_3)|. $ Furthermore, for each $m \geq 0$ and $0 \leq k \leq 2^m-1$, construct  
 \begin{eqnarray*}
 \mathbb{Q}^m_k := \left\{ (Q_1, Q_2) \in \mathbb{Q}: \left( Q_2 \times -Q_2 \right) \cap\left(  E^m_{k, left} \times -E^m_{k, right} \right) \not = \emptyset \right\}.  \end{eqnarray*}
  It is easy to see that for fixed $m,k$, there are $O(1)$ intervals  $Q_2$ of a given size and hence $O(1)$ many cubes $\vec{Q}$ of a given size in $\mathbb{Q}^m_k$. The boxes $Q_2 \times - Q_2$ in $\mathbb{Q}^m_k$ are also essentially nested, so that $\{Q_1\}_{\vec{Q} \in \mathbb{Q}^m_k}$ is a Littlewood-Paley collection.  We shall also need the following technical results: 
  \begin{lemma}\label{L10}
 Let $F, G, H \in \mathcal{S}(\mathbb{R})$. Moreover, let $\eta^{\vec{Q}}_{P_1, 1}, \eta_{P_2, 2}$ be functions which are fourier-localized onto the intervals $P_1$ and $P_2$ and such that the standard uniform decay properties hold.  Moreover, suppose $P_1 + P_2 \subset [c |P|, C|P|]$. Then 
 
 \begin{eqnarray*}
 \int_\mathbb{R} F \cdot G * \eta^{\vec{Q}}_{P_1, 1}\cdot H*\eta_{P_2, 2}~dx =\sum_{k \in\mathbb{Z}}c^{\vec{Q}, \vec{P}}_k  \int_\mathbb{R} F*\eta_{|P|,1}^{\vec{Q}, \vec{P}, k } \cdot G \cdot H*\eta^k_{P_2, 2}~dx, 
 \end{eqnarray*}
 where $\eta_{|P|,1} ^{\vec{Q}, \vec{P}, k}, \eta_{P_2, 2}^k$ are both $L^1$-normalized bump functions courier-adapted to $[c|P|, C|P|]$ and $P_2$ respectively and the sequence $c_k^{\vec{Q}, \vec{P}}$ decays uniformly in the parameters $\vec{Q}$ and $\vec{P}$, i.e. we have an implicit constant such that 
 
 \begin{eqnarray*}
\left |c_k^{\vec{Q}, \vec{P}} \right|\lesssim_N \frac{1}{1+|k|^N}.
 \end{eqnarray*}

 \end{lemma}

 \begin{proof}
 By assumption, we may include a function $\eta_{|P|,0}$ in the left hand side such that
 
 \begin{eqnarray*}
 &&\int_\mathbb{R} F(x) \cdot G * \eta^{\vec{Q}}_{P_1, 1} (x)\cdot H*\eta_{P_2, 2}(x)~dx \\&=& \int_\mathbb{R} F*\eta_{|P|, 0}(x) \cdot G * \eta^{\vec{Q}}_{P_1, 1}(x) \cdot H*\eta_{P_2, 2}(x)~dx \\ &=& \int_{\mathbb{R}^3} \delta(\xi_1 + \xi_2 + \xi_3) \hat{F}(\xi_1) \hat{\eta}_{|P|, 0}(\xi_1) \hat{G}(\xi_2) \hat{\eta}^{\vec{Q}}_{P_1, 1} (\xi_2) \hat{H}(\xi_3) \hat{\eta}_{P_2, 2}(\xi_3) ~d\xi_1 d \xi_2 d \xi_3 .
 \end{eqnarray*}
  We may expand $\hat{\eta}^{\vec{Q}}_{P_1, 1}(\xi_2)$ as a Fourier series by 
 
 \begin{eqnarray*}
 \hat{\eta}^{\vec{Q}}_{P_1, 1}(\xi_2)= \sum_{k \in \mathbb{Z}} c^{\vec{Q}, \vec{P}}_k \hat{\eta}^k_{P_1, 1} (\xi_2), 
 \end{eqnarray*} so that expanding $\hat{\eta}_{P_1}^k(-\xi_1- \xi_3)$ in double Fourier Series on $-[c|P|, C|P|] \times -P_2$ yields

 \begin{eqnarray*}
LHS  &=& \sum_{k \in \mathbb{Z}}\int_{\mathbb{R}^3} \delta(\xi_1 + \xi_2 + \xi_3)  c^{\vec{Q}, \vec{P}}_k  \hat{F}(\xi_1) \hat{\eta}_{|P|, 0}(\xi_1) \hat{G}(\xi_2) \hat{\eta}^k_{P_1, 1} (\xi_2) \hat{H}(\xi_3) \hat{\eta}_{P_2, 2}(\xi_3) ~d\xi_1 d \xi_2 d \xi_3 \\ &=&  \sum_{k \in \mathbb{Z}}\int_{\mathbb{R}^3} \delta(\xi_1 + \xi_2 + \xi_3)  c^{\vec{Q}, \vec{P}}_k  \hat{F}(\xi_1) \hat{\eta}_{|P|, 0}(\xi_1) \hat{G}(\xi_2) \hat{\eta}^k_{P_1, 1} (-\xi_1 - \xi_3) \hat{H}(\xi_3) \hat{\eta}_{P_2, 2}(\xi_3) ~d\xi_1 d \xi_2 d \xi_3 \\ &=& \sum_{k \in \mathbb{Z}} \sum_{l \in \mathbb{Z}} \int_{\mathbb{R}^3} \delta(\xi_1 + \xi_2 + \xi_3)  c^{\vec{Q}, \vec{P}}_k d_l^k  \hat{F}(\xi_1) \hat{\eta}^{k,l}_{|P|, 0}(\xi_1) \hat{G}(\xi_2) \hat{H}(\xi_3) \hat{\eta}^{k,l}_{P_2, 2}(\xi_3) ~d\xi_1 d \xi_2 d \xi_3 ,
 \end{eqnarray*}
 where $|d_l^k| \lesssim \frac{1}{1+|l|^N}$ uniformly in $k$. We may rewrite the above as 
 
 \begin{eqnarray*}
  \sum_{k \in \mathbb{Z}} \sum_{l \in \mathbb{Z}}\tilde{c}_k \tilde{d}_l  \int_{\mathbb{R}^3} \delta(\xi_1 + \xi_2 + \xi_3)  \bar{c}_k^{\vec{Q}, \vec{P}} \bar{d}_l^k  \hat{F}(\xi_1) \hat{\eta}^{k,l}_{|P|, 0}(\xi_1) \hat{G}(\xi_2) \hat{H}(\xi_3) \hat{\eta}^{k,l}_{P_2, 2}(\xi_3) ~d\xi_1 d \xi_2 d \xi_3.  
 \end{eqnarray*}
Defining $\hat{\tilde{\eta}}^{\vec{Q}, \vec{P}, k,l}_{|P|,0} (\xi_1) =   \bar{c}_k^{\vec{Q}, \vec{P}} \bar{d}_l^k\hat{\eta}^{k,l}_{|P|, 0}(\xi_1)$ gives the lemma, once we inject $\mathbb{Z} \times \mathbb{Z} \rightarrow \mathbb{Z}$.

 \end{proof} 
    \begin{lemma}\label{L11}
For each $k \in \mathbb{K}$, let 

\begin{eqnarray*}
m_k (\xi_1, \xi_2) := \sum_{\vec{P} \in \mathbb{P}} \eta_{P_1, 1, k} (\xi_1) \eta_{P_2, 2, k} (\xi_2)
\end{eqnarray*}
 be a generic multiplier of Hilbert transform type. Then, one has the following estimates for the maximal bi-sub-linear operator defined as $M_{\{ m_k\}}: (f_1, f_2) \mapsto \sup_{k \in \mathbb{K}} \left| T_{m_k}(f_1, f_2) \right|$:
 
 \begin{eqnarray*}
|| M_{\{m_k\}_{k \in \mathbb{K}}}(f_1, f_2) ||_{\frac{p_1 p_2}{p_1 + p_2}}  \lesssim_{\epsilon, p_1, p_2} |\mathbb{K}|^\epsilon ||f_1||_{p_1} || \hat{f}_2||_{p_2^\prime}
 \end{eqnarray*}
 for all $f_1 \in L^{p_1}(\mathbb{R})$ and $f_2 \in W_{p_2}(\mathbb{R})$ satisfying $\frac{1}{p_1} + \frac{1}{p_2} <1$ and $2 < p_2 \leq \infty$. 
 \end{lemma}

 \begin{proof}
 For each $k \in \mathbb{K}$, we introduce the same carving as before. So, WLOG, 
 
 \begin{eqnarray*}
 T_{\{m_k\}} (f_1, f_2) = \sum_{m \geq 0} \sum_{k=0}^{2^m-1}\sum_{\vec{P} \in \mathbb{P}} f_1*\eta_{P_1, 1, k}*\check{1}_{E^m_{k,left} } f_2* \eta_{P_2, 2, k}  * \check{1}_{-E^m_{k, right}}.
 \end{eqnarray*}
  Therefore, we may dualize and obtain for the disjoint collection $\{S_k\}_{k \in \mathbb{K}}$
  \begin{eqnarray*}
 || M_{\{m_k\}} (f_1, f_2)||_{\frac{p_1 p_2} {p_1 + p_2}}& \simeq& \int_\mathbb{R} \sum_{k \in \mathbb{K}} g1_{S_k}  \sum_{m \geq 0} \sum_{\lambda=0}^{2^m-1}\sum_{\vec{P} \in \mathbb{P}} f_1*\eta_{P_1, 1, k}*\check{1}_{E^m_{\lambda,left} } f_2* \eta_{P_2, 2, k}  * \check{1}_{-E^m_{\lambda, right}} ~dx \\ &=&  \int_\mathbb{R} \sum_{k \in \mathbb{K}} g1_{S_k}  \sum_{m \geq 0} \sum_{\lambda=0}^{2^m-1}\sum_{\vec{P} \in \mathbb{P}^m_k[A]} f_1*\eta_{P_1, 1, k}*\check{1}_{E^m_{\lambda,left} } f_2* \eta_{P_2, 2, k}  * \check{1}_{-E^m_{\lambda, right}} ~dx \\ &+&  \int_\mathbb{R} \sum_{k \in \mathbb{K}} g1_{S_k}  \sum_{m \geq 0} \sum_{\lambda=0}^{2^m-1}\sum_{\vec{P} \in \mathbb{P}^m_k[B]} f_1*\eta_{P_1, 1, k}*\check{1}_{E^m_{\lambda,left} } f_2* \eta_{P_2, 2, k}  * \check{1}_{-E^m_{\lambda, right}} ~dx\\ &:=& \sum_{m \geq 0} I_A^m + I_B^m . 
 \end{eqnarray*}
  It suffices to show that for each $m \geq 0$ we have 
 
 \begin{eqnarray*}
&& I_A^m + I_B ^m  \lesssim_{\epsilon, p_1, p_2} | \mathbb{K}|^\epsilon 2^{-C m } || f_1||_{p_1} || \hat{f}_2||_{p_2^\prime}.
 \end{eqnarray*}
 Recall $\mathbb{P}^m_k [A]$ is a paraproduct of type $A$ and so $\left\{ P_1 \right\}_{\vec{P} \in \mathbb{P}^m_k[A]}$ lacunary, while $\mathbb{P}^m_k [B]$ is a paraproduct of type $B$ and so $\left\{ P_2 \right\}_{\vec{P} \in \mathbb{P}^m_k[B]}$ is lacunary. 
 By symmetry, it suffices to handle $I^m_B$. To this end, note
 
 \begin{eqnarray*}
&& |I^m_B| \\&=& \left|  \sum_{\lambda=0}^{2^m-1}\sum_{\vec{P} \in \mathbb{P}^m_\lambda[B]}\int_\mathbb{R} \sum_{k \in \mathbb{K}} (g1_{S_k})* \eta_{|P|, k}    f_1* \check{1}_{E^m_{\lambda,left} } f_2* \eta_{P_2, 2, k}  * \check{1}_{-E^m_{\lambda, right}} ~dx \right|  \\ &\leq&  \sum_{\lambda=0}^{2^m-1}\sum_{k \in \mathbb{K}} \int_\mathbb{R}  \left( \sum_{\vec{P} \in \mathbb{P}^m_\lambda[B]}\left| (g1_{S_k})* \eta_{|P|, k}\right|^2 \right)^{1/2}   | f_1* \check{1}_{E^m_{\lambda,left} }| \left(\sum_{\vec{P} \in \mathbb{P}^m_\lambda[B]} |  f_2* \eta_{P_2, 2, k}  * \check{1}_{-E^m_{\lambda, right}}|^2 \right)^{1/2} ~dx \\ & \leq&   \sum_{\lambda=0}^{2^m-1}\sum_{k \in \mathbb{K}} \int_\mathbb{R}  \left( \sum_{\vec{P} \in \mathbb{P}^m_\lambda[B]}\left| (g1_{S_k})* \eta_{|P| }\right|^2 \right)^{1/2}   | f_1* \check{1}_{E^m_{\lambda,left} }| \left(\sum_{\vec{P} \in \mathbb{P}^m_\lambda[B]} |  f_2* \eta_{P_2, 2}  * \check{1}_{-E^m_{\lambda, right}}|^2 \right)^{1/2} ~dx \\ &\leq&  \sum_{\lambda=0}^{2^m-1} \sum_{k \in \mathbb{K}} \int_\mathbb{R} \left|  \tilde{H} \left(  (g1_{S_k})\right) \right|  \cdot |  f_1* \check{1}_{E^m_{\lambda,left} }| \cdot \left|  \tilde{H}^{\lambda} \left( f_2  * \check{1}_{-E^m_{\lambda, right}} \right) \right| dx  \\ &\leq&      \int_\mathbb{R}   \sum_{k \in \mathbb{K}} \left|  \tilde{H} \left(  g1_{S_k}\right)  \right| \cdot \left( \sum_{\lambda} \left| f_1* \check{1}_{E^m_{\lambda,left} } \right|^2 \right)^{1/2} \cdot  \left( \sum_\lambda \left|  \tilde{H}^{\lambda} \left( f_2  * \check{1}_{-E^m_{\lambda, right}} \right) \right|^2 \right)^{1/2}dx  \\ &\leq& \left| \left|  \sum_{k \in \mathbb{K}} \left| \tilde{H} ( g1_{S_k}) \right|\right| \right|_{p_0} \left| \left| \left( \sum_{\lambda} \left| f_1* \check{1}_{E^m_{\lambda,left} } \right|^2 \right)^{1/2}  \right| \right|_{p_1} \left| \left|  \left( \sum_\lambda \left|  \tilde{H}^{\lambda} \left( f_2  * \check{1}_{-E^m_{\lambda, right}} \right) \right|^2 \right)^{1/2} \right| \right|_{p_2}.
 \end{eqnarray*}
As the second and third factors have already been handled by previous estimates, it suffices to understand the first. We first perform discrete Holder, which gives us that 
 
 \begin{eqnarray*}
  \left| \left|  \sum_{k \in \mathbb{K}} \left|  \tilde{H} ( g1_{S_k}) \right| \right| \right|_{p_0} &\leq& |\mathbb{K}|^{\epsilon} \left| \left| \left(  \sum_{k \in \mathbb{K}} \left|  \tilde{H} ( g1_{S_k}) \right|^{\frac{1}{1-\epsilon}} \right)^{1-\epsilon} \right| \right|_{p_0}  \\ &\lesssim_\epsilon&  |\mathbb{K}|^{\epsilon} \left| \left| \left(  \sum_{k \in \mathbb{K}} \left|   g1_{S_k}) \right|^{\frac{1}{1-\epsilon}} \right)^{1-\epsilon} \right| \right|_{p_0}  \\ &\leq& | \mathbb{K}|^\epsilon. 
 \end{eqnarray*}
 As before, the second and third factor give us a combined $2^{m \left( \frac{1}{p_1}-\frac{1}{p_2^\prime}\right)} ||f_1||_{p_1} || \hat{f}_2 ||_{p_2^\prime}.$
 
 \end{proof}
 We now use Lemmata \ref{L10} and \ref{L11} to conclude
 
 \begin{eqnarray*}
 &&\tilde{ \Lambda}_{B[a_1, a_2]_1}^m(f_0, f_1, f_2, f_3) \\ &:= &\left|\int_\mathbb{R} \sum_k \sum_{\vec{Q} \in \mathbb{Q}}~ \sum_{\vec{P}\in \mathbb{P}^m_k : |\vec{Q}| >>  |\vec{P}| } ( f_0 * \eta_{-Q_1} f_1 * \eta_{Q_1} )* \eta_{|P|} f_2*\eta_{Q_2} *\eta_{P_1} * \check{1}_{E^m_{k, left}} f_3*\eta_{P_2}* \check{1}_{-E^m_{k, right}} ~dx \right| \\ &=&\left|\int_\mathbb{R} \sum_k \sum_{\vec{Q} \in \mathbb{Q}^m_k}~ \sum_{\vec{P}\in \mathbb{P}^m_k : |\vec{Q}| > >|\vec{P}| } ( f_0 * \eta_{-Q_1} f_1 * \eta_{Q_1} )* \eta_{|P|} f_2*\eta_{Q_2} *\eta_{P_1} * \check{1}_{E^m_{k, left}} f_3*\eta_{P_2}* \check{1}_{-E^m_{k, right}} ~dx \right|  \\ &\lesssim& \left|\int_\mathbb{R} \sum_k \sum_{\vec{Q} \in \mathbb{Q}^m_k}~ \sum_{\vec{P} \in \mathbb{P}^m_k: |\vec{Q}| >> |\vec{P}| } ( f_0 * \eta_{-Q_1} f_1 * \eta_{Q_1} )* \eta^{\vec{Q}}_{|P|} f_2  * \check{1}_{E^m_{k, left}} f_3*\eta_{P_2}* \check{1}_{-E^m_{k, right}} ~dx \right| \\ &\leq&  \int_\mathbb{R} \sum_k ~\left( \sum_{\vec{P} \in \mathbb{P}^m_k}\left| \sum_{ |\vec{Q}| >> |\vec{P}|}( f_0 * \eta_{-Q_1} f_1 * \eta_{Q_1} )* \eta^{\vec{Q}}_{|P|}\right|^2 \right)^{1/2} |f_2  * \check{1}_{E^m_{k, left}}| \left(  \sum_{\vec{P} \in \mathbb{P}^m_k } \left| f_3*\eta_{P_2}* \check{1}_{-E^m_{k, right}} \right|^2 \right)^{1/2}~dx.
 \end{eqnarray*}
 The last expression in the above display is majorized by
 \begin{eqnarray*}
&\leq& \int_\mathbb{R} \sup_k\left[  ~\left( \sum_{\vec{P} \in \mathbb{P}^m_k}\left|\sum_{\vec{Q} \in \mathbb{Q}^m_k: |\vec{Q}| >> |\vec{P}|}  ( f_0 * \eta_{-Q_1} f_1 * \eta_{Q_1} )* \eta^{\vec{Q}}_{|P|}\right|^2 \right)^{1/2}\right]  \\ &\times& \left( \sum_k |f_2  * \check{1}_{E^m_{k, left}}| ^2 \right)^{1/2} \left(  \sum_k\sum_{\vec{P} \in \mathbb{P}^m_k } \left| f_3*\eta_{P_2}* \check{1}_{-E^m_{k, right}} \right|^2 \right)^{1/2}~dx, 
 \end{eqnarray*}
 where the inequality in the third line arises from the fact that one has implicitly used triangle inequality on a countable number of terms with rapidly decaying coefficients arising from two applications of Fourier Series as described in Lemma \ref{L10}. Finish by using H\"{o}lder's inequality, linearizing with Rademacher functions, and applying Lemma \ref{L11}.

\end{proof}

\section{Counterexample for a Bilinear Operator related to $B[a_1, a_2]$}
From the proceeding proofs, it is clear that the estimate $T^{a_1, a_2} : L^{p_1}(\mathbb{R}) \times W_{p_2}(\mathbb{R}) \times L^{p_3} (\mathbb{R}) \rightarrow L^{\frac{1}{\frac{1}{p_1} + \frac{1}{p_2} + \frac{1}{p_3}}}(\mathbb{R})$ would hold if $m (\xi_1, \xi_2) := \sum_{\vec{P} \in \mathbb{P} } \sum_{\lambda \in \mathbb{Z}} \eta_{-P_1}(\xi_1) \eta_{P_1}(\xi_2) \hat{\psi}^{\vec{P}}_\lambda(\xi_1 + \xi_2)$ (where $\left\{ P \right\}$ is a lacunary sequence) is a generic bounded multilinear multiplier and $\sup_{k \in \mathbb{K}} \left| T_{m_k}(\cdot, \cdot) \right|$ has an operatorial bound growing $O_\epsilon(|\mathbb{K}|^\epsilon)$ for every $\epsilon$. However, the next proposition states that $m_k$ can be chosen to satisfy no $L^p$ estimates. 
 \begin{prop}\label{LC}
Let $P_k = [2^k - 2^{k-1}, 2^k + 2^{k-1}]$ for all $k \in \mathbb{Z}$. Then there exists a collection $\left\{ \eta^\lambda _k \right\}_{(k, \lambda ) \in \mathbb{Z}^2} $, where each $\hat{\eta}^\lambda_k$ is $k-$uniformly adapted to and supported in $P_\lambda$ so that the bilinear operator given by 

\begin{eqnarray*}
\mathcal{B} :(f_1, f_2) \mapsto \sum_{k \in \mathbb{Z}} \sum_{\lambda \in \mathbb{Z}}\left[  f_1*\eta^{k} f_2*\widetilde{\eta^{k}} \right] * \eta^\lambda_k
\end{eqnarray*}
 satisfies no $L^p$ estimates.
 \end{prop}
 
 \begin{proof}
Set $\eta^k \equiv \eta^\lambda_k \equiv 0$ if $ k <0$. For $k \geq 0$, set $\hat{\eta}^k(\xi) =2^{-k} \tilde{1}(2^{-k}\xi)$ for some $\tilde{1} \in C^\infty(\mathbb{R})$ satisfying $1_{[7/8, 9/8]}(\xi) \leq \tilde{1}(\xi) \leq 1_{[3/4, 3/2]}(\xi)$.  Furthermore, set $\hat{\eta}^{-k_0} _k(\xi)  = e^{2 \pi i 2^{-k_0}k}  \chi (2^{k_0}(\xi -2^{-k_0}))$ where $\check{\chi} \geq 1_{[-1,1]}$ and $supp~\chi \subset [-1/2, 1/2]$.

 Fix $N \in \mathbb{N}$ large. Choose $f^N_1 = \sum_{ 1 \leq n \leq N} e^{2 \pi i 2^n x} \phi_1(x-n), f^N_2 = \sum_{1 \leq n \leq N} e^{-2 \pi i 2^n x } \phi_2(x-n)$ where $\phi_1, \phi_2$ have Fourier support inside $[-1/8, 1/8]$ so that $\phi_1 \phi_2$ has Fourier support inside $[-1/4, 1/4]$. Moreover,  choose $\phi_1, \phi_2$ to ensure  $\phi_1 \phi_2$ has flat Fourier transform on $[-1/16, 1/16]$. This can be achieved with $\hat{\phi}_1$ chosen to be a skinny peak and $\hat{\phi}_2$ chosen to be a wide peak. Then the contribution at each lacunary scale $k_0: 100 \leq k_0 \leq   \log_2(N)$ is 
 
 \begin{eqnarray*}
\mathcal{B}_{k_0}(f^N_1, f^N_2)(x):=  \sum_{k \in \mathbb{Z}} \left[  f^N_1*\eta^{k} f^N_2*\widetilde{\eta^{k}} \right] * \eta_{k}^{-k_0}  =  e^{2 \pi i 2^{-k_0} x} \sum_{ 1 \leq n \leq N} 2^{-k_0} \check{\chi} (2^{-k_0} (x-n)) .
 \end{eqnarray*}
Hence,  $|\mathcal{B}_{k_0}(f^N_1, f^N_2)(x)| \gtrsim 1_{[1,N]}(x)$.
Littlewood-Paley equivalence then yields 
 
 \begin{eqnarray*}
 \left| \left| \sum_{k_0 \in \mathbb{Z}} \mathcal{B}_{k_0} (f^N_1, f^N_2) \right| \right|_{\frac{p_1 p_2}{p_1 + p_2}} \simeq  \left| \left| \left( \sum_{k_0 \in \mathbb{Z}} \left|  \mathcal{B}_{k_0} (f^N_1, f^N_2) \right|^2 \right)^{1/2} \right| \right|_{\frac{p_1 p_2}{p_1 + p_2}} \simeq \log(N)^{1/2} N^{\frac{1}{p_1} + \frac{1}{p_2}},
 \end{eqnarray*}
 whereas $||f^N_1||_{p_1} \simeq N^{1/p_1}$ and $||f^N_2||_{p_2} \simeq N^{1/p_2}$. Taking $N$ arbitrarily large yields the proposition.

 \end{proof}
Proposition \ref{LC} suggests that the perspective taken in \S{3}-\S{7} is not able to produce $LWL-$type mixed estimates for $B[a_1, a_2]$. For this reason, we shall need to adopt a more sophisticated time-frequency perspective. 
  \section{Generalized Restricted Type Mixed Estimates }
 Let us begin by recalling the setup and notation from \cite{MR2199086}. 
 \begin{definition}
 For each measurable subset $E \subset \mathbb{R}$ with finite measure let $X(E) = \left\{ f: |f| \leq 1_E~a.e. \right\}$ with respect to Lebesgue measure. 
 \end{definition}
 
 \begin{definition}
 A multisublinear form is of restricted type $\alpha = (\alpha_1, ..., \alpha_n$ with $0 \leq \alpha_j \leq 1$ if there exists a constant $C$ such that for each tuple $E=(E_1, ..., E_n)$ of measurable subsets of $\mathbb{R}$ and for each tuple $f=(f_1, ..., f_n)$ with $f_j \in X(E_j)$, we have 
 
 \begin{eqnarray*}
 |\Lambda(f)| := |\Lambda(f_1, ..., f_n)| \leq C|E|^\alpha
 \end{eqnarray*}
 where $|E|^\alpha = \prod_j |E_j|^{\alpha_j}$.  
 
 \end{definition}
 
 \begin{definition}
 Let $\alpha$ be an $n-$tuple of real numbers and assume $\alpha_j \leq 1$ for all $j$. An $n-$sublinear form is called generalized restricted type $\alpha$ if there is a constant $C$ such that for all tuples $E= (E_1, ..., E_n)$ there is an index $j$ and a major subset $\tilde{E}_j$ of $E_j$ such that for all tuples $f=(f_1, ...., f_n)$ with $f_k \in X(E_k)$ for all $k$ and in addition $f_j \in X(\tilde{E}_j)$ we have 
 
 \begin{eqnarray*}
 |\Lambda(f_1, ..., f_n)| \leq C |E|^\alpha. 
 \end{eqnarray*}

 \end{definition}
 From the standpoint of multilinear Marcinkiewicz interpolation, we may in fact allow the exceptional set $\tilde{E}_j$ in the above description to depend not just on the choice of $E=(E_1, ..., E_n)$ but also on the choice of $f_k$ for all $k \not = j$. In the case when $j=n$, say, this would mean there exists a constant $C$ such that for each $(E_1, ..., E_n)$ and all $(f_1, ..., f_{n-1})$ with $f_k \in X(E_k)$ for all $k \leq n-1$, there exists  a major subset $\tilde{E}_n(f_1, ..., f_{n-1})$ for which $|\Lambda(f_1, ..., f_n)| \leq C |E|^\alpha$ for all $f_n \in X(\tilde{E}_n(f_1, ..., f_{n-1}))$.

To prove mixed type estimates for the given multi-linear form $\Lambda$ , it suffices by Marcinekiewicz interpolation to obtain weak mixed estimates. The weak statement is that for every tuple $(E_1,..., E_n)$ and collection of functions $f_j : \mathbb{R} \rightarrow \mathbb{C}$ satisfying $|f_j| \leq 1_{E_j},  |\Lambda (\vec{f})| \lesssim_{\vec{p}} \prod_{j=1}^n  |E_j|^{\frac{1}{p_j}}$. In the mixed setting,  the usual condition $\sum_{j=1}^n \beta_j=1$ is replaced by $\sum_{j\not =i }^n \beta _j =\beta_{i}$, supposing the mixed index falls on the $ith$ position. Details of the Marcinekiewicz interpolation in this ``mixed" setting are provided in the following two lemmas.

\subsection{Marcinkiewicz Interpolation Lemmas}  The proofs of both results are essentially the same as C. Thiele arguments in \cite{MR2199086}. 
\begin{lemma}
Let $\Lambda$ be a multi-linear form which satisfies $|\Lambda(\vec{f})| \leq C \prod_{j=1}^n |E_j|^{\alpha_j}$ with uniformly bounded constant C for all tuples $\vec{\alpha}=(\alpha_1, ..., \alpha_n)$ in some neighborhood of $\vec{\beta}=(\beta_1, ..., \beta_n)$ where $\sum_{j \not = i}^n \alpha_j =  \alpha_i$, $0 < \alpha_j <1$ for all $j \in \{1, .., n\}$ and assume, in addition, $\beta_i > \frac{1}{2}$. Then for all $\vec{f}=(f_1, ..., f_n)$,

\begin{eqnarray*}
\left| \Lambda(\vec{f}) \right| \lesssim C \prod_{j=1}^n ||f_j||_{\frac{1}{\beta_j}}. 
\end{eqnarray*}
\end{lemma}

\begin{proof}
Without loss of generality, suppose $f_j \geq 0$ for all indices $1 \leq j \leq n$. For each $f_j$ appearing in the tuple $\vec{f}$, we construct 2 sequences of subsets of $\mathbb{R}$ denoted by $\left\{ \tilde{F}^j_k\right\}_{k \in \mathbb{Z}}$ and $\left\{ F^j_k\right\}_{k \in \mathbb{Z}}$ with the following properties:

\begin{eqnarray*}
|\tilde{F}^j_k| &=& |F^j_k|=2^k\\
\tilde{F}^j_k& \supset& F^j_{k-1}\\
essinf~\{ f_j(x) : x \in \tilde{F}^j_k\}& \geq& esssup~\{f_j(x): x \in F^j_k\}\\
essinf~ \{ f_j(x) : x \in F^j_k\}& \geq& esssup~\{f_j(x) : x \not \in F^j_k \cup \tilde{F}^j_k\}\\
\left\{ F^j_k\right\}_{k \in \mathbb{Z}}&&~is~a~partition.
\end{eqnarray*}
For each index $j \in \{1, .., n\}$ introduce the splitting $f_j(x) = \sum_{k \in \mathbb{Z}} f_j(x) 1_{F^j_k}(x)$ and note by multi-linearity of $\Gamma$

\begin{eqnarray*}
|\Lambda(\vec{f})| \leq C \sum_{k_1, ..., k_n \in \mathbb{Z}} ~~\prod_{j=1}^n\left[  \left| esssup\{ f_j(x): x \in F^j_{k_j}\} \right| 2^{k_j \alpha_j} \right]
\end{eqnarray*}
where $C$ is the uniform constant appearing in the statement of the lemma and $\vec{\alpha}$ is a tuple in a neighborhood of $\vec{\beta}$ for which mixed weak estimates hold. For fixed $(k_1, ..., k_n)$, we wish to choose $\vec{\alpha}$ in such a way as to guarantee for $k:= \sum_{j=1}^n |k_j|$

\begin{eqnarray*}
\sum_{j=1}^n \alpha_j k_j = \sum_{j=1}^n \beta_j k_j - \epsilon \max \{|k_1 - k|, |k_2-k|, ... , |k_i + k|, ..., |k_n-k|\}.
\end{eqnarray*}
Indeed, $\vec{\alpha}$ and $\vec{\beta}$ are both subject to the restructions $\sum_{j\not = i}^n \alpha_j -\alpha_i =\sum_{j \not =i} \beta_j - \beta_i=0$. Therefore, $(\vec{\alpha}-\vec{\beta}) \cdot (1, ...., -1, ..., 1)=0$ and our conclusion is one can always choose $\vec{\alpha}(\vec{k})$ to satisfy $|\vec{\alpha}-\vec{\beta}| <\delta$ and $(\vec{\alpha}-\vec{\beta}) \cdot \vec{k} = -\epsilon \max \{|k_1 - k|, |k_2-k|, ... , |k_i + k|, ..., |k_n-k|\}$ for some $\epsilon(\delta)$, where $k$ is the average of the $k_is$.  Therefore, with this choice of $\vec{\alpha}(\vec{k})$,
\begin{eqnarray*}
|\Lambda(\vec{f})| \leq C 2^{-\epsilon\max \{|k_1  k|, |k_2-k|, ... , |k_i + k|, ..., |k_n-k|\}}   \prod_{j=1}^n\left[  \left| esssup\{ f_j(x): x \in F^j_{k_j}\} \right| 2^{k_j \beta_j} \right].
\end{eqnarray*}
Introduce $\tilde{k}_1 := k_1 - k, \tilde{k}_2=k_2-k, ..., \tilde{k}_i = k_i + k, ..., \tilde{k}_n = k_n -k$ to bound the above expression as

\begin{eqnarray*}
\sum_{\tilde{k}_1, ..., \tilde{k}_{n-1}\in \mathbb{Z}}~ \sum_{k  \in \mathbb{Z}} C 2^{-\epsilon\max \{| \tilde{k}_1|, ..., |\tilde{k}_{n-1}|\} }  \prod_{j=1}^n\left[  \left| esssup\{ f_j(x): x \in F^j_{\tilde{k}_j \pm k}\} \right| 2^{(\tilde{k}_j\pm k) \beta_j} \right].
\end{eqnarray*}
Now apply H\"{o}lder's inequality in $k$ with the following computation valid when $2 \beta_i \geq 1$:

\begin{eqnarray*}
 \sum_k \left[  \left| esssup\{ f_j(x): x \in F^j_{\tilde{k}_j \pm k}\} \right| 2^{(\tilde{k}_j\pm k) \beta_j} \right]& \leq&   \left(\sum_k \left( \left| esssup\{ f_j(x): x \in F^j_{\tilde{k}_j \pm k}\} \right| 2^{(\tilde{k}_j\pm k) \beta_j} \right) ^{\frac{1}{2 \beta_i}} \right)^{2 \beta_i}\\ \\ &\leq& \prod_{j=1}^n \left( \sum_{k_j} \left| esssup\{ f_j(x): x \in F^j_{k_j}\} \right|^{\frac{1}{\beta_j}}  2^{k_j} \right) ^{\beta_j} \\ &\leq& \prod_{j=1}^n || f_j||_{\frac{1}{\beta_j}}.
\end{eqnarray*}
In the last line, we implicitly used for each $1 \leq j \leq n$

\begin{eqnarray*}
\left(\sum_{k_j} \left| esssup\{ f_j(x): x \in F^j_{k_j}\} \right|^{\frac{1}{\beta_j}}  2^{k_j} \right) ^{\beta_j} &\leq& \left(\sum_{k_j} \left| essinf~\{ f_j(x): x \in F^j_{k_j-1}\} \right|^{\frac{1}{\beta_j}}  2^{k_j} \right) ^{\beta_j}\\& \lesssim &\left(\sum_{k_j} || f_j 1_{F^j_{k_j}}||_{\frac{1}{\beta_j}} ^{\frac{1}{\beta}_j} \right) ^{\beta_j}\\& \leq& ||f_j||_{\frac{1}{\beta_j}}.
\end{eqnarray*}

\end{proof}
\begin{lemma}
Fix $n \geq 2$. Assume $\Lambda$ is of generalized restricted weak $\beta$ where $\sum_{j \not=  i } \beta_j = \beta_i$ for some $i \not =n$, $\beta_i \geq \frac{1}{2}, \beta_k >0$ for all $k < n$ and $\beta_n \leq 0$. Assume $\Lambda$ is also of generalized restricted type $\alpha$ for all $\alpha$ in a neighborhood of $\beta$ satisfying $\sum_{j \not = i} \alpha_j = \alpha_i$ where $\alpha_i \geq \frac{1}{2}$.Then the dual form (in the nth function) $T$ satisfies 

\begin{eqnarray*}
|| T(f_1, ..., f_{n-1})||_{\frac{1}{1 - \beta_n}} \leq C \prod_{j=1}^n ||f_j||_{1/\beta_j}  
\end{eqnarray*}
where $C$ depends only on the constants appearing the generalized restricted type estimates near $\beta$. 
\begin{proof}
Fix $f_1, ..., f_{n-1}$. By pre- and post composing with measure preserving transformation, we may assume that $|f_j|$as well as $|T(f_1, ..., f_n))|$ are supported in $[0, \infty)$ and non increasing. We write 

\begin{eqnarray*}
||T(f)||_p^p &=& \int |T(f)(x)|^p dx \\ 
&\leq& \sum_{k \in \mathbb{Z}} \left( 2^{-k} \int_{2^k}^{2^{k+1}} Tf(x) dx \right)^p 2^k \\ 
&\lesssim& \sum_{k \in \mathbb{Z}} 2^{k(1-p)} \left( \sum_{k_`, ..., k_{n-1}} \int_{2^k} ^{2^{k+1}} T(f_j 1_{[2^{k_j}, 2^{k_j+1})}) dx \right)^{1/p}. 
\end{eqnarray*}
Setting $E_n = (0, 2^{k+1})$, we see that

\begin{eqnarray*}
\int_{2^k}^{2^{k+1}} T g(x) dx \leq \int |Tg(x)| 1_{\tilde{E}_n} dx
\end{eqnarray*}
for every tuple $g$ and every major subset $\tilde{E}_n$ of $E_n$. Therefore, by the generalized restricted type estimate, 

\begin{eqnarray*}
|| T(f)||_p^p \leq \sum_k 2^{k(1-p)} \left( \sum_{k_1, ..., k_{n-1}} 2^{k \alpha_n} \prod_{j=1}^{n-1} f_j(2^{k_j}) 2^{k_j \alpha_j} \right)^p. 
\end{eqnarray*}
Using the freedom to choose $\alpha_j$ for each tuple $(k_1, ..., k_{n-1})$, we obtain for $p = \frac{1}{1-\beta_n}$ 

\begin{eqnarray*}
||T(f)||_p^p &\leq& \sum_k 2^{k(1-p)} \left( \sum_{k_1, ...., k_{n-1}} 2^{-\epsilon \max |k_j\pm k| }2^{k \beta_n} \prod_{j=1}^{n-1} f_j(2^{k_j} ) 2^{k_j \beta_j} \right)^p \\ &=& \sum_k 2^{k(1-p)} \left( \sum_{\tilde{k}_1, ...., \tilde{k}_{n-1}} 2^{-\epsilon \max |\tilde{k}_j| }2^{k \beta_n} \prod_{j=1}^{n-1} f_j(2^{\tilde{k}_j \pm k}) 2^{(\tilde{k}_j \pm k) \beta_j} \right)^p \\ &\leq& \sum_k \sum_{\tilde{k}_1, ..., \tilde{k}_{n-1}} 2^{-\tilde{\epsilon} \max | \tilde{k}_j| } \left( \prod_{j=1}^{n-1} f_j(2^{\tilde{k}_j \pm k}) 2^{(\tilde{k}_j \pm k) \beta_j} \right)^p \\ &\lesssim&\sum_{\tilde{k}_1, ..., \tilde{k}_{n-1}} 2^{-\tilde{\epsilon} \max | \tilde{k}_j| } \left( \sum_k  \left( \prod_{j=1}^{n-1} f_j(2^{\tilde{k}_j \pm k}) 2^{(\tilde{k}_j \pm k) \beta_j} \right)^{\frac{1}{2\beta_i -\beta_n}} \right)^{p (2 \beta_i - \beta_n)} \\ &\leq& \sum_{\tilde{k}_1, ..., \tilde{k}_{n-1}} 2^{-\tilde{\epsilon} \max |\tilde{k}_j|} \left( \prod_{j=1}^n || f_j||_{1/ \beta_j} \right)^p \\ &\lesssim& \left( \prod_{j=1}^n || f_j||_{1/ \beta_j} \right)^p.
\end{eqnarray*}
The condition $\beta_i \geq  1/2$ was crucial in the line before H\"{o}lder's inequality was applied.

\end{proof}
\end{lemma}

\section{Essential Time-Frequency Definitions}
We now introduce some essential time-frequency definitions and then prove number of results concerning degenerate multilinear symbols. For the readers' convenience, we include the definitions that will be used extensively in the remainder of this work. 

\begin{definition}
Let $n \geq 1$ and $\sigma \in \{ 0, \frac{1}{3}, \frac{2}{3} \}^n$. We define the shifted $n-$dyadic mesh $D= D^n_\sigma$ to be the collection of cubes of the form 

\begin{eqnarray*}
D^n_\sigma := \left\{ 2^j(k+(0,1)^n + (-1)^j \sigma) : j \in \mathbb{Z}, k \in \mathbb{Z}^n \right\}
\end{eqnarray*}

\end{definition}
Observe that for every cube $Q$, there exists a shifted dyadic cube $Q^\prime$ such that $Q \subseteq  \frac{7}{10} Q^\prime$ and $|Q^\prime| \sim |Q|$; this property clearly follows from verifying the $n=1$ case. The constant $\frac{7}{10}$ is not especially important here. 

\begin{definition}
A subset $D^\prime$ of a shifted $n-$dyadic grid $D$ is called sparse, if for any two cubes $Q, Q^\prime$ in $D$ with $Q \not = Q^\prime$ we have $|Q| < |Q^\prime|$ implies $|10^9 Q| < |Q^\prime|$ and $|Q|=|Q^\prime|$ implies $10^9 Q \cap 10^9 Q^\prime = \emptyset$.
\end{definition}
It is immediate from the above definition that any subset of a shifted $n-$dyadic grid can be split into $O(C^n)$ sparse subsets. 
\begin{definition}

For a given spatial interval $I$, let $\tilde{\chi}_I(x) := \left( 1+ \left( \frac{|x-x_I|}{|I|} \right)^2 \right)^{1/2}$, where $x_I$ is the center of $I$. 
\end{definition}

\begin{definition}
Let $P= (I_P, \omega_P)$ be a tile. A wave packet on $P$ is a function $\Phi_P$ which has Fourier support in $\frac{9}{10} \omega_P$ and obeys the estimate 

\begin{eqnarray*}
|\Phi_P(x)| \lesssim_M  |I_{P}|^{-1/2} \tilde{\chi}^M_{I_{P}}(x)
\end{eqnarray*}
for some fixed large integer $M$.  
Therefore, $\Phi_P$ is $L^2$ normalized and adapted to the Heisenberg box $(I_{P}, \omega_P )$.  
\end{definition}
We next introduce the tile ordering $<$ from \cite{MR2127985}, which is in the spirit of Fefferman or Lacey and Thiele, but different inasmuch as $P^\prime$ and $P$ do not have to intersect. 
\begin{definition}
Let $\sigma = (\sigma_1, \sigma_2, \sigma_3) \in \{0, \frac{1}{3}, \frac{2}{3} \}^3$, and let $1 \leq i \leq 3$. An $i-$tile with shift $\sigma_i$ is a rectangle $P = (I_P, \omega_P)$ with area $1$ and with $I_P \in D_0^1, \omega_P \in D^1_{\sigma_i}$. A tri-tile with shift $\sigma$ is a $3$-tuple $\vec{P} = (P_1, P_2, P_3)$ such that each $P_i$ is an $i-$tile with shift $\sigma_i$, and the $I_{P_i} = I_{\vec{P}}$ are independent of $i$. The frequency cube $Q_{\vec{P}}$ of a tri-tile is defined to be $\prod_{i=1}^3 \omega_{P_i}$. 
\end{definition}

\begin{definition}
A set $\mathbb{P}$ of tri-tiles is called sparse, if all the tri-tiles in $\mathbb{P}$ have the same shift $\sigma$ and the set of frequency cubes $\{Q_{\vec{P}}= (\omega_{P_1}, \omega_{P_2}, \omega_{P_3}): \vec{P} \in \mathbb{P} \}$ 
 is sparse. 
\end{definition}
\begin{definition}
Let $P$ and $P^\prime$ be tiles. We write $P^\prime < P$ if $I_{P^\prime} \subsetneq I_{P}$ and $3 \omega_P \subseteq 3 \omega_{P^\prime}$, and $P^\prime \leq P$ if $P^\prime <P $ or $P^\prime = P$. We write $P^\prime \lesssim P$ if $I_{P^\prime} \subseteq I_P$ and $10^7 \omega_P \subseteq 10^7 \omega_{P^\prime}$. We write $ P^\prime \lesssim^\prime P$ of $P^\prime \lesssim P$ and $P^\prime \not \leq P$. 
\end{definition}

\begin{definition}
A collection $\mathbb{P}$ of tri-tiles is said to have rank 1 if one has the following properties for all $\vec{P}, \vec{P}^\prime \in \mathbb{P}$: 

If $\vec{P} \not = \vec{P}^\prime$, then $P_j \not = P_j^\prime$ for all $j =1,2,3$. 

If $P_j^\prime \leq P_j$ for some $j= 1,2,3$, then $P_i^\prime \lesssim P_i$ for all $1 \leq i \leq 3$. 

If we further assume that $|I_{\vec{P}^\prime} |> 10^9 |I_{\vec{P}}|$, then $P_i^\prime \lesssim^\prime P_i$ for all $i \not = j$.

\end{definition}
\begin{definition}
For any $1 \leq j \leq 3$ and tri-tile $\vec{P}_T \in \mathbb{P}$, define a $j$-tree with top $\vec{P}_T$ to be a collection of tri-tiles $T \subset \mathbb{P}$ such that 

\begin{eqnarray*}
P_j \leq P_{T, j}~for~all~ \vec{P} \in T,
\end{eqnarray*}
where $P_{T, j}$ is the jth component of $\vec{P}_T$. We write $I_T$ and $\omega_{T, j}$ for $I_{\vec{P}_T}$ and $\omega_{P_{T, j}}$ respectively. We say that $T$ is a tree if it is a $j-$tree for some $1 \leq j \leq 3$. 

\end{definition}
We do not require $T$ to contain its top $\vec{P}_T$. 

  \begin{definition}
  Let $1 \leq j \leq 3$. Two trees $T, T^\prime$ are strongly $j-$disjoint if 
  
  $P_j \not = P_j^\prime$ for all $\vec{P} \in T, \vec{P}^\prime \in T^\prime$ 
  
Whenever $\vec{P} \in T, \vec{P}^\prime \in T^\prime$ satisfy $2 \omega_{P_j} \cap 2 \omega_{P_j^\prime} \not = \emptyset$, then $I_{\vec{P}^\prime} \cap I_T = \emptyset$, and similarly with $T$ and $T^\prime$ reversed. 
  \end{definition}
Note that if $T$ and $T^\prime$ are strongly $j-$disjoint, then $I_{\vec{P}} \times 2 \omega_{P_j} \cap I_{\vec{P}^\prime} \times 2 \omega_{P_j^\prime} = \emptyset $ for all $\vec{P} \in T, \vec{P}^\prime \in T^\prime$. 

\begin{definition}
Let $\omega_1$ and $\omega_2$ be intervals. Then write $\omega_1 \subset \subset \omega_2$ provided $|\omega_1| << |\omega_2|$ for some sufficiently large absolute constant and $\omega_1 \subset \omega_2$. 
\end{definition}
 \section{Mixed Estimates for the Scale-1 Hilbert Transform in the Plane}
With these preliminaries out of the way, we now state and prove

  \begin{prop}
 Let $m_0 \in \mathcal{M}_\Gamma(\mathbb{R}^2)$ be supported in $dist(\vec{\xi} , \Gamma) \simeq 1$. Then $ T_{m_0}:L^{p_1}(\mathbb{R}) \times W_{p_2}(\mathbb{R}) \rightarrow L^{\frac{p_1 p_2}{p_1 + p_2}}(\mathbb{R})$ provided $\frac{1}{p_1} + \frac{1}{p_2} <1, 2 < p_2 < \infty$. 

 \end{prop} 
 \begin{proof}
 By standard discretization arguments, see \cite{MR3052499},  it suffices to prove restricted weak-type estimates uniform in neighborhoods near the points $(1/2, 1/2, 0), (0, 1/2, 1/2), (1,0,0)$ and $(0,0,1)$ for the $3-$form defined by

\begin{eqnarray*}
\Lambda_{T_0}(f_1, f_2, f_3) = \sum_{\vec{P} \in \mathbb{P}}\frac{1}{|I_{\vec{P}}|} \langle f_1, \Phi_{P_1,1}\rangle \langle f_2, \Phi_{P_2, 2} \rangle \langle f_3 , \Phi_{|P|, 3}^{lac} \rangle 
\end{eqnarray*}
where $\mathbb{P}$ is a scale-1 collection of tiles. In particular, it suffices to show that for every $(E_1, E_2, E_3)$ such that $E_j \subset \mathbb{R}$ is measurable for each $j \in \{1, 2, 3\}$ and $(f_1, f_2, f_3)$ satisfying $f_j \in X(E_j)~for~j \in \{1, 3\}$ and $|\hat{f}_2 |\leq 1_{E_2}$, there exists $E_1^\prime (\vec{f})$ a major subset of $E_1$ such that

\begin{eqnarray*}
|\Lambda(f_11_{E^\prime_1}, f_2, f_3)| \lesssim_{\vec{p}} |E_1|^{1/p_1} |E_2|^{1-1/p_2} |E_3|^{1-1/p_1 - 1/p_2}
\end{eqnarray*}
for $(p_1, p_2, p_3)$ in neighborhoods of $(1/2, 1/2, 0), (1,0,0), (0,0,1)$ and a similar statement for $(0,1/2, 1/2)$, except with the exceptional set attached to the 1st index. 
Let $\vec{P}_{n_1, n_2, n_3} = \mathbb{P}_{n_1, 1} \cap \mathbb{P}_{n_2, 2} \cap \mathbb{P}_{n_3, 3}$ where for each $j \in \{1, 2, 3\}$ and $n \gtrsim 1$

\begin{eqnarray*}
\mathbb{P}_{n, j}:= \left\{ \vec{P} \in \mathbb{P} : \frac{|\langle f_j, \Phi_{P_j, j} \rangle |}{|I_{\vec{P}}|^{1/2}} \simeq 2^{-n} \right\}.
\end{eqnarray*}
Moreover, let $\Omega = \left\{ M1_{E_1} \geq C |E_1| \right\} \cap \left\{Mf_2 \geq C |E_2|^{1/2} \right\}$ and set $\mathbb{P}^d= \left\{ \vec{P} \in \mathbb{P} : 1+dist(I_{\vec{P}} , \Omega^c)/|I_{\vec{P}}| \simeq 2^{d} \right\}$. Lastly, define $\mathbb{P}^d_{n_1, n_2, n_3} = \mathbb{P}^d \cap \mathbb{P}_{n_1, n_2, n_3}$. 
By triangle inequality and Cauchy-Schwarz, 
\begin{eqnarray*}
&& \left| \Lambda(f_1, f_2, f_3) \right| \\&\leq& \sum_{ d \geq 0} \sum_{n_1, n_2, n_3} \sum_{\vec{P}\in \mathbb{P}^d_{n_1, n_2, n_3}} \left[ \sup_{\vec{P} \in \mathbb{P}_{n_1, n_2, n_3}} \frac{ |\langle f_1, \Phi_{P_1} \rangle | }{|I_{\vec{P}}|^{1/2}} \right] \left[ \sup_{\vec{P}\in \mathbb{P}_{n_1, n_2, n_3}}\frac{  |\langle f_2, \Phi_{P_2} \rangle | }{|I_{\vec{P}}|^{1/2}} \right] \left[ \sup_{\vec{P}\in \mathbb{P}_{n_1, n_2, n_3}} \frac{|\langle f_3, \Phi_{|\vec{P}|} \rangle|}{|I_{\vec{P}}|^{1/2}} \right] |I_{\vec{P}}| \\ &\lesssim&\sum_{ d \geq 0} \sum_{n_1 \geq N_1(d), n_2 \geq N_2(d), n_3 \geq N_3(d)} 2^{-n_1} 2^{-n_2} 2^{-n_3} \sum_{\vec{P} \in \mathbb{P}^d_{n_1, n_2, n_3}} |I_{\vec{P}}| .
\end{eqnarray*}
To prove $T_0: L^{p_1}(\mathbb{R}) \times W_{p_2}(\mathbb{R}) \rightarrow L^{\frac{p_1 p_2}{p_1+p_2}}(\mathbb{R})$, it suffices by weak interpolation to Each $\mathbb{P}_{n_1, n_2, n_3}$ is a collection of scale 1 tiles. The sum over the spatial lengths of all tiles in this collection can be estimated in two ways: 

\begin{eqnarray*}
\sum_{\vec{P} \in \mathbb{P}_{n_1, n_2, n_3}} |I_{\vec{P}}| \lesssim \{ 2^{2n_1} |E_1|, 2^{2n_2} |E_2|\}
\end{eqnarray*}
However, this is not enough to get summability over all three parameters $n_1, n_2, n_3$. We must provide another estimate into the above sum which makes use of $n_3$, which is achieved using information about how many tiles may stack on top of each other. 

\begin{prop}

\begin{eqnarray*}
 \#_{n_1, n_2, n_3}:= \sup_{|I_{\vec{P}}|=1} \left| \left\{ \vec{Q} \in \mathbb{P}_{n_1, n_2, n_3} : I_{\vec{Q}} = I_{\vec{P}}\right\}\right|\lesssim 2^{n_2} || \hat{f}_1||_1
 \end{eqnarray*}
 
  \end{prop}
\begin{proof}

Because $\vec{P} \in \mathbb{P}_{n_1, n_2, n_3}, 2^{-n_2} \lesssim \frac{ |\langle f_2, \Phi_{P_2} \rangle |}{|I_{P_2}|^{1/2}}$. Let the supremum be attained by some interval $I_{P_0}$. Then observe

\begin{eqnarray*}
\#_{n_1, n_2, n_3} 2^{-n_2} \lesssim \sum_{\vec{Q} \in \mathbb{P}_{n_1, n_2, n_3}} \frac{ |\langle f_2, \Phi_{Q_2} \rangle |}{|I_{Q_2}|^{1/2}} \lesssim \sum_{\vec{Q} \in \mathbb{P}_{n_1, n_2, n_3}: I_{\vec{Q}}=I_{P_0}}  \langle |\hat{f}_2|, \tilde{1}_{\omega_{Q_2}} \rangle \leq ||\hat{f}_2||_1.
\end{eqnarray*}
\end{proof}
It follows that $\sum_{\vec{P} \in \mathbb{P}_{n_1, n_2, n_3}} |I_{\vec{P}}| \lesssim 2^{n_2} ||\hat{f}_2||_1 \cdot \sum_{I \in \mathbb{I}_{n_1, n_2, n_3}} |~I~|$ where 

\begin{eqnarray*}
\mathbb{I}_{n_1, n_2, n_3} := \{ I  \in \mathbb{D}: \exists \vec{P} \in \mathbb{P}_{n_1, n_2, n_3}~s.t.~I=I_{\vec{P}}\}
\end{eqnarray*}
 and $\mathbb{D}$ is the collection of dyadic intervals. Moreover, for every $I \in \mathbb{I}_{n_1, n_2, n_3}$, 
\begin{eqnarray*}
\sum_{I \in \mathbb{I}_{n_1,n_2, n_3}} |I| \leq \sum_{I \in \mathbb{I}_{n_3} } |I| \lesssim 2^{n_3} \sum_{ I \in \mathbb{I}_{n_3}} \langle 1_{E_3}, \tilde{1}_I \rangle \lesssim 2^{n_3} |E_3|. 
\end{eqnarray*}
Putting it all together, we have the additional estimate $\sum_{\vec{P} \in \mathbb{P}_{n_1, n_2, n_3}} |I_{\vec{P}}| \lesssim 2^{n_2} 2^{n_3} |E_2||E_3|$, which enables us to write down for any $(\theta_1, \theta_2, \theta_3)$ subject to the requirement $0 \leq \theta_1, \theta_2, \theta_3 \leq1$ and $\theta_1+\theta_2+\theta_3=1$

\begin{eqnarray*}
|\Lambda(f_1, f_2, f_3)| \lesssim \sum_{n_1, n_2, n_3}2^{-n_1(1-2\theta_1)} 2^{-n_2(1-\theta_2 -2 \theta_3)} 2^{-n_3(1-\theta_2)} |E_1|^{\theta_1} |E_2|^{\theta_2 + \theta_3} |E_3|^{\theta_2}.
\end{eqnarray*}
For summability, we must impose the additional requirement that $ \theta_1<1/2$ and $\theta_2 + 2\theta_3 <1$.

\subsection{Restricted Weak Type Estimates}
By rescaling,  we may assume $|E_3|=1$.  Note that the natural size restrictions are then $2^{-n_1} \lesssim 2^d |E_1|^\alpha$ for any $0 \leq \alpha  \leq 1$, $2^{-n_2} \lesssim 2^{d} |E_2|^{1/2}$, and $2^{-n_3} \lesssim 2^{-\tilde{N} d}$.  Fixing $(\theta_1, \theta_2, \theta_3)$ satisfying $0 \leq \theta_1, \theta_2, \theta_3 \leq 1$ and $\theta_1 + \theta_2+ \theta_3=1$ with $\theta_1 <1/2$,  the summation gives 

\begin{eqnarray*}
\left| \Lambda(f_1, f_2, f_3)\right| \lesssim \sum _{d \geq0}\sum_{n_1\geq N_1(d)} \sum_{ n_2 \geq N_2(d)} \sum_{ n_3 \geq N_3(d)} 2^{-n_1 (1-2\theta_1)} 2^{-n_2 (1- \theta_2 -2\theta_3)} 2^{-n_3 (1-\theta_2)} |E_1| ^{\theta_1} |E_2|^{\theta_2 + \theta_3}
\end{eqnarray*}
Provided $\theta_1 <1/2, \theta_2 < 1, \theta_2 + 2\theta_3 <1$, we have the upper bound 

\begin{eqnarray*}
\left| \Lambda(f_1, f_2, f_3)\right| \lesssim \min\left\{ |E_1|^{\theta_1}, |E_1|^{1-\theta_1} \right\}  |E_2|^{1/2+\theta_2/2}.
\end{eqnarray*}
To produce restricted weak estimates in a neighborhood of $(1/2, 1/2, 0)$, set $\theta_1 = 1/2-\epsilon, \theta_2 = \epsilon, \theta_3 = 1/2 - \epsilon$. To do the same in neighborhoods of $(1, 0 , 0)$ and $(0,0,1)$ use $\theta_1 = 2\epsilon, \theta_2 = 1-3\epsilon, \theta_3 = \epsilon$. By interpolation, it suffices to prove estimates in a neighborhood of $(0,1/2, 1/2)$. To this end, assume $|E_1|=1$ and the exceptional set $\Omega$ attached to $f_1$ satisfies

\begin{eqnarray*}
\Omega \supset \left\{Mf_2\geq C |E_2|^{1/2} \right\} \bigcup \left\{ M1_{E_3} \geq C  |E_3| \right\}.
\end{eqnarray*}
As before, this exceptional set will have an acceptable size provided $C$ is sufficiently large. 
The natural size restrictions are $2^{-n_3} \lesssim 2^d |E_3|^{\alpha}$ for any $0 \leq \alpha \leq 1$, $2^{-n_2} \lesssim 2^d |E_2|^{1/2}$, and $2^{-n_1} \lesssim 2^{-\tilde{N} d}$. A similar calculation then yields for $\theta_1 <1/2, \theta_2 < 1, \theta_2 + 2 \theta_3 <1$, 

\begin{eqnarray*}
\left| \Lambda(f_1 1_{\Omega^c}, f_2, f_3) \right| \lesssim  |E_2|^{1/2 + \theta_2/2} \min\left\{ |E_3|^{\theta_2}, |E_3|^{1-\theta_2} \right\}.
\end{eqnarray*}
Choosing $\theta_1 = 1/2 - \epsilon, \theta_2 = 3\epsilon, \theta_3 = 1/2 -2 \epsilon $ yields the desired estimate for $\Lambda(f_1, f_2, f_3)$ near $(0, 1/2, 1/2)$ and therefore, by interpolation, produces the desired mixed estimates for $T_{m_0}$.

\end{proof}
 \section{Mixed Estimates for the Generic Hilbert Transform in the Plane}
The proceeding argument uses the fact that each tree consists of only one tile and the number of scale-1 tiles stacking on top of each other is limited by the relevant size parameter. The problem with extending this line of argument to the general case is that for a given strongly-disjoint collection of trees $\mathbb{T}_{3, n_3}$, there is no reasonable bound for  $\left| \left| \sum_{T\in \mathbb{T}_{3, n_3}} 1_{I_T}  \right|\right|_\infty$. 
 Hence, the proof has to move from stacking of trees to stacking of individual tiles above (and below) a certain \emph{point} in time. Indeed, strong disjointness ensures that at each time the frequency projections of the relevant tiles with time concentration $I_{\vec{P}}$ intersecting a shared point are all disjoint. Now, we prove mixed estimates for generic bilinear degenerate symbols.   
 
 \begin{theorem}
Let $m \in \mathcal{M}_\Gamma(\mathbb{R}^2)$.  Then  $T_m:L^{p_1}(\mathbb{R}) \times W_{p_2}(\mathbb{R}) \rightarrow L^{\frac{p_1 p_2}{p_1 + p_2}}(\mathbb{R})$ whenever 

\begin{eqnarray*}
\frac{1}{p_1} + \frac{1}{p_2} <1, 2<p_2 < \infty.
\end{eqnarray*}

\end{theorem} 
\begin{proof}
Omitting standard details, it suffices to obtain bounds for the model sum

\begin{eqnarray*}
\Lambda(f_1, f_2, f_3) =\sum_{\vec{Q} \in \mathbb{Q}_d} \frac{ \langle f_1, \Phi^1_{Q_1} \rangle \langle  f_2*\eta_{\omega_{Q_2}}, \tilde{\Phi}^2_{Q_2} \rangle \langle f_3, \Phi^{3,lac}_{\vec{Q}} \rangle}{|I_{\vec{Q}}|^{1/2}}.
\end{eqnarray*}
Setting

\begin{eqnarray*}
Size_3 (f_3) = \sup_{I \subset \mathbb{R}} \frac{1}{|I| } \int_\mathbb{R} 1_{E_3} (x) \tilde{1}_{I}(x) dx, 
\end{eqnarray*}
we may obtain for each $n_3$ a collection of disjoint dyadic intervals $\mathbb{I}_{n_3}$ with the property that 
\begin{eqnarray*}
\frac{1}{|I|} \int_I |f_3(x)| dx\gtrsim 2^{-n_3}
\end{eqnarray*}
for every $I \in \mathbb{I}_{n_3}$ and such that each $I$ is maximal with respect the set of all dyadic intervals enjoying the above property. Therefore, for every $I \in \mathbb{I}_{n_3}$ and $n_3 < m_3$, there exists a unique $J \in \mathbb{I}_{m_3}$ such that $I \subseteq J$. This filtration then gives rise to a partition of bi-tiles in the usual way: note that for every $d \geq 0$ there is an integer $n_0(d)$ such that for no integer $n_3 < n_0(d)$ is there a dyadic interval $I$ with the property that $I \subset J$ for some $J \in \mathbb{I}_{n_3}$ and $1+\frac{dist(I, \Omega^c)}{|I|} \simeq 2^d. $
Indeed, if $I \subset J$ with $J \in \mathbb{I}_{n_3}$, then 

\begin{eqnarray*}
2^{-n_3} &\lesssim & \frac{1}{|J| } \int_\mathbb{R} 1_{E_3}(x) \tilde{1}_{J} (x) dx \\&\lesssim& 2^d \inf_{x \in 2^d J} M 1_{E_3}( x) \\ &\leq& 2^d \inf_{x \in 2^d I} M 1_{E_3} (x)  \\ &\leq&  2^d \sup_{x \in \Omega^c} M1_{E_3}(x).
\end{eqnarray*}
It is routine to control the above display by an acceptable quantity by enlarging our exceptional set $\Omega$. 
Therefore, we may start our decomposition at $n_3 \geq n_0(d)$. Let 

\begin{eqnarray*}
\mathbb{Q}_{d, n_0(d)} &:=& \left\{ \vec{Q} \in \mathbb{Q} : I_{\vec{Q}} \subset \bigcup_{I \in \mathbb{I}_{ n_0(d)}} I \right\}\\
\mathbb{Q}_{d, n_0(d)+1} &:=& \left\{  \vec{Q} \in \mathbb{Q} \cap \mathbb{Q}_{3, n_0(d)}^c : I_{\vec{Q}} \subset \bigcup_{I \in \mathbb{I}_{ n_0(d)+1}} I \right\} . 
\end{eqnarray*}
Inductively define $\mathbb{Q}_{d, m} := \left\{  \vec{Q} \in \mathbb{Q} \cap \left[ \bigcup_{n \leq m-1} \mathbb{Q}_{d, n}\right]^c : I_{\vec{Q}} \subset \bigcup_{I \in \mathbb{I}_{ m}} I \right\} . $ Therefore, 

\begin{eqnarray*}
\mathbb{Q} = \bigcup_{n \geq n_0(d)} \mathbb{Q}_{d, n}
\end{eqnarray*}
where the union is disjoint and for every tree $T$ consisting of bi-tiles in $\mathbb{Q}_{3, n}$ sitting in some $I \in \mathbb{I}_{ n}$

\begin{eqnarray*}
\left( \frac{ \sum_{\vec{Q} \in T} | \langle f_3, \Phi_{\vec{Q}}^{3, lac} \rangle|^{2}}{|I_T|} \right) \lesssim \sup_{ \vec{Q} \in T}  \frac{1}{|I_{\vec{Q}}|}\int _\mathbb{R} |f_3(x)| \tilde{1}_{I_{\vec{Q}}}(x)dx \lesssim 2^{-n}.
\end{eqnarray*}
Now, for each collection of disjoint dyadic intervals $\mathbb{I}_n$, we construct adapted sizes in the first and second indices. For a given collection $\mathbb{I}$ and collection of bi-tiles $\tilde{\mathbb{Q}}$ for which $I_{\vec{Q}} \subset  \bigcup_{I \in \mathbb{I}}I$, let 
\begin{eqnarray*}
Size_1^{\mathbb{I}}( \tilde{\mathbb{Q}}) := \sup_{T \subset \tilde{\mathbb{Q}}: I_T \subset \bigcup_{I \in \mathbb{I}} I} \frac{1}{|I_T|^{1/2}} \left( \sum_{Q \in T} | \langle f_1, \Phi^1_{Q_1} \rangle|^2 \right)^{1/2}.
\end{eqnarray*}
Next, we define 

\begin{eqnarray*}
Size_2^{\mathbb{I}} (\tilde{\mathbb{Q}}) := \sup_{T \subset \tilde{\mathbb{Q}} : I_T \subset \bigcup_{I \in \mathbb{I}_{n_3}} \mathbb{I}} \frac{1}{|I_T|^{1/2}}\left( \sum_{\vec{Q} \in T} \left| \langle f_2*\eta_{\omega_{Q_2}}, \tilde{\Phi}^{2}_{Q_2} \rangle \right|^2 \right)^{1/2}. 
\end{eqnarray*}
This procedure yields a decomposition of the tiles $\mathbb{Q}$ into a disjoint union of sub collections $\mathbb{Q}_{3, n_3}$, where each $\vec{Q} \in \mathbb{Q}_{3, n_3}$ has the property that $I_{\vec{Q}} \subset \bigcup_{I \in \mathbb{I}} I$. Moreover, 

\begin{eqnarray*}
 \sum_{\vec{Q} \in \mathbb{Q}_d} \frac{ |\langle f_1, \Phi^1_{Q_1} \rangle |\langle f_2*\eta_{\omega_{Q_2}}, \tilde{\Phi}^2_{Q_2} \rangle| \langle f_3, \Phi^{3,lac}_{\vec{Q}} \rangle|}{|I_{\vec{Q}}|^{1/2}} &=& \sum_{n_3}\sum_{\vec{Q} \in \mathbb{Q}_d \cap \mathbb{Q}_{3, n_3}}  \frac{ \langle f_1, \Phi^1_{Q_1} \rangle \langle f_2*\eta_{\omega_{Q_2}}, \tilde{\Phi}^2_{Q_2} \rangle| \langle f_3, \Phi^{3,lac}_{\vec{Q}} \rangle|}{|I_{\vec{Q}}|^{1/2}}.
\end{eqnarray*}
As usual, we may break down each subcollection $\mathbb{Q}_{3, n_3}$ according to the same BHT type stopping time argument now done with respect to the new localized sizes. Clearly, the strongly disjoint trees can be grouped according to the interval $I \in \mathbb{I}_{n_3}$ containing the top of the tree $I_T$. Denote this collection of trees $\mathbb{T}_{1, n_1}(\mathbb{Q}_{3, n_3}) [ I]$ and the collection of bi-tiles $(\mathbb{Q}_{3, n_3})_{n_1}^{n_2} (I)$. Putting it all together therefore yields

\begin{eqnarray*}
\mathbb{Q} = \bigcup_{n_3} \mathbb{Q}_{3, n_3} = \bigcup_{n_3} \left[ \bigcup_{n_1, n_2} (\mathbb{Q}_{3, n_3})_{n_1, n_2} \right] = \bigcup_{n_3} \left[ \bigcup_{n_1, n_2} \bigcup_{I \in \mathbb{I}_{n_3}} (\mathbb{Q}_{3, n_3})_{n_1}^{ n_2}(I) \right] 
\end{eqnarray*}
(One sorts the tiles on each interval $I \in \mathbb{I}_{n_3}$ separately.) Each set $(\mathbb{Q}_{3, n_3}) _{n_1, n_2} (I)$ can be further decomposed into a collection of disjoint trees (modulo harmless modifications). Putting it all together yields

\begin{eqnarray*}
&& \sum_{\vec{Q} \in \mathbb{Q}_d \cap \mathbb{Q}_{3, n_3}}  \frac{ |\langle f_1, \Phi^1_{Q_1} \rangle |\langle f_2*\eta_{\omega_{Q_2}}, \tilde{\Phi}^2_{Q_2} \rangle| \langle f_3, \Phi^{3,lac}_{\vec{Q}} \rangle|}{|I_{\vec{Q}}|^{1/2}} \\&=& \sum_{\vec{Q} \in \mathbb{Q}_d \cap (\mathbb{Q}_{3, n_3})_{n_1}^{n_2}}  \frac{ |\langle f_1, \Phi^1_{Q_1} \rangle |\langle f_2*\eta_{\omega_{Q_2}}, \tilde{\Phi}^2_{Q_2} \rangle| \langle f_3, \Phi^{3,lac}_{\vec{Q}} \rangle|}{|I_{\vec{Q}}|^{1/2}}  \\&=& \sum_{I \in \mathbb{I}_{n_3}}\sum_{n_1,n_2} ~\sum_{\vec{Q} \in \mathbb{Q}_d \cap (\mathbb{Q}_{3, n_3})_{n_1}^{ n_2} (I)} ~ \frac{ |\langle f_1, \Phi^1_{Q_1} \rangle |\langle f_2*\eta_{\omega_{Q_2}}, \tilde{\Phi}^2_{Q_2} \rangle| \langle f_3, \Phi^{3,lac}_{\vec{Q}} \rangle|}{|I_{\vec{Q}}|^{1/2}} .
\end{eqnarray*}

\subsection{Degenerate Tree Estimate}
If $T$ is a $2-tree$ in $(\mathbb{Q}_{3, n_3})_{n_1}(I)$, then

\begin{eqnarray*}
&& \sum_{\vec{Q} \in T} \frac{ |\langle f_1, \Phi^1_{Q_1} \rangle |\langle f_2*\eta_{\omega_{Q_2}}, \tilde{\Phi}^2_{Q_2} \rangle| \langle f_3, \Phi^{3,lac}_{\vec{Q}} \rangle|}{|I_{\vec{Q}}|^{1/2}}\\& \leq& |I_T| \frac{\left( \sum_{\vec{Q} \in T} |\langle f_1, \Phi^1_{Q_1} \rangle|^2 \right)^{1/2}}{|I_T|^{1/2}} \left[\sup_{\vec{Q} \in T} \frac{\langle f_2*\eta_{\omega_{Q_2}}, \tilde{\Phi}^{2,\infty}_{Q_2} \rangle}{|I_{\vec{Q}}|} \right] \frac{ \left(\sum_{\vec{Q} \in T} |\langle f_3, \Phi_{\vec{Q}}^{3, lac} \rangle|^2 \right)^{1/2}}{|I_T|^{1/2}} \\ &\lesssim& |I_T| 2^{-n_1} 2^{-n_2} 2^{-n_3} .
\end{eqnarray*}
If T is a $1-tree$ in $(\mathbb{Q}_{3, n_3} )^{n_2} (I)$, then 

\begin{eqnarray*}
&& \sum_{\vec{Q} \in T} \frac{ |\langle f_1, \Phi^1_{Q_1} \rangle |\langle f_2*\eta_{\omega_{Q_2}}, \tilde{\Phi}^2_{Q_2} \rangle| \langle f_3, \Phi^{3,lac}_{\vec{Q}} \rangle|}{|I_{\vec{Q}}|^{1/2}} \\ &\leq&|I_T| \left[\sup_{\vec{Q}\in T} \frac{ |\langle f_1, \Phi_{Q_1}^{1,\infty} \rangle|}{|I_{\vec{Q}}|} \right]\left( \frac{ \sum_{\vec{Q} \in T} \langle f_2*\eta_{\omega_{Q_2}}, \tilde{\Phi}_{Q_2}^{2, \infty} \rangle ^{1+\epsilon} }{|I_T|} \right)^{1/(1+\epsilon)} \left( \frac{ \sum_{\vec{Q} \in T} | \langle f_3, \Phi_{\vec{Q}}^{3, lac} \rangle|^{(1+\epsilon)/\epsilon }}{|I_T|}\right)^{\epsilon/(1+\epsilon)} \\ &\lesssim& |I_T| 2^{-n_1} 2^{-n_2} 2^{-n_3}.
\end{eqnarray*}
\subsection{Degenerate Energy Estimate}
It is routine to observe

\begin{eqnarray*}
\sum_{I \in \mathbb{I}_{n_3}} ~\sum_{T \in \mathbb{T}_{1, n_1}(\mathbb{Q}_{3, n_3}) [I]} |I_T| \lesssim 2^{2n_1} |E_1|; \sum_{I \in \mathbb{I}_{n_3}}~ \sum_{T \in \mathbb{T}_{2, n_2}(\mathbb{Q}_{2, n_2}) [I]} |I_T| \lesssim 2^{2n_2} |E_2|
\end{eqnarray*}
This follows from the usual $TT^*$ argument and the fact that collection $ \bigcup_{I \in \mathbb{I}_{n_3}} \bigcup_{T \in \mathbb{T}_{1, n_1}(I)}\left\{ T\right\}, \bigcup_{I \in \mathbb{I}_{n_3}} \bigcup_{T \in \mathbb{T}_{2, n_2}(I)}\left\{ T\right\}$ are both strongly-disjoint collections of trees. However, we have the following additional story: 
\begin{eqnarray*}
\sum_{I \in \mathbb{I}_{n_3}}~ \sum_{T \in \mathbb{T}_{2, n_2}(\mathbb{Q}_{2, n_2}) [I]} |I_T| &\lesssim& 2^{2n_2} \sum_{I \in \mathbb{I}_{n_3}}~ \sum_{T \in \mathbb{T}_{2, n_2}(\mathbb{Q}_{2, n_2}) [I]}  |\langle f_2 *\eta_{\omega_{Q_2}} , \tilde{\Phi}_{Q_2}^{2} \rangle|^2   \\  &\lesssim& 2^{2n_2} \sum_{I \in \mathbb{I}_{n_3}}~ \sum_{T \in \mathbb{T}_{2, n_2}(\mathbb{Q}_{2, n_2}) [I]}  |\langle f_2 *\eta_{\omega_{Q_2}} , \tilde{\Phi}_{Q_2}^{2,\infty} \rangle| \left[ \sup_{T \in \mathbb{T}_{2, n_2} (\mathbb{Q}_{2, n_2} ) [ I]} \sup_{\vec{Q} \in T} \left|\langle f_2 , \Phi_{Q_2} ^{2, 1} \rangle\right| \right] \\ &\lesssim& 2^{n_2}  \sum_{I \in \mathbb{I}_{n_3}} \sum_{T \in \mathbb{T}_{2, n_2}(I)} \sum_{\vec{Q} \in T}  \left| \left| f_2*\eta_{\omega_{Q_2}} \right| \right|_{L^\infty(\mathbb{R})} |I_{\vec{Q}}|\\&\leq& 2^{n_2}  \sup_{I_ \in \mathbb{I}_{n_3}}  \left[ \int_\mathbb{R} \sum_{T \in \mathbb{T}_{2, n_2} (I)} \sum_{\vec{Q} \in T} \left| \left| \hat{f}_2 \hat{\eta}_{\omega_{Q_2}}  \right| \right| _{L^1(\mathbb{R})} 1_{I_{\vec{Q}}} dx \right] \left[ \sum_{I \in \mathbb{I}_{n_3}}|~I~| \right]\\& \lesssim& 2^{n_2} 2^{n_3} || \hat{f}_2||_{1} |E_3| \\ &=& 2^{n_2} 2^{n_3} |E_2| \cdot |E_3|.
\end{eqnarray*}
  Therefore, 

\begin{eqnarray*}
\left| \Lambda(f_1, f_2, f_3) \right|& \leq&  \sum_{d, n_1, n_2, n_3} \left| \Lambda_{d,n_1, n_2, n_3} (f_1, f_2, f_3)\right|\\& \lesssim&  \sum _{d \geq0}\sum_{n_1\geq N_1(d)} \sum_{ n_2 \geq N_2(d)} \sum_{ n_3 \geq N_3(d)} 2^{-n_1} 2^{-n_2} 2^{-n_3} \min \left\{ 2^{2n_1}|E_1| , 2^{n_2} 2^{n_3} |E_2| \cdot |E_3| , 2^{2n_2} |E_2|  \right\}.
\end{eqnarray*}

\subsection{Restricted Weak Type Estimates}
By scaling invariance,  assume $|E_3|=1$. By enlarging $\Omega$ if necessary, we may ensure for fixed $\alpha >>1$
that
 \begin{eqnarray*}
 \Omega \supset \left\{  M 1_{E_1} \gtrsim |E_1| \right\} \bigcup \left\{ M f_2 \gtrsim |E_2|^{1/2} \right\}.
 \end{eqnarray*}
 is an acceptable exceptional set large enough implicit constants. 
 Note that the natural size restrictions are then $2^{-n_1} \lesssim 2^d |E_1|^\alpha$ for any $0 \leq \alpha  \leq 1$, $2^{-n_2} \lesssim 2^{d} |E_2|^{1/2}$, and $2^{-n_3} \lesssim 2^{-\tilde{N} d}$.  Fixing $(\theta_1, \theta_2, \theta_3)$ satisfying $0 \leq \theta_1, \theta_2, \theta_3 \leq 1$ and $\theta_1 + \theta_2+ \theta_3=1$ with $\theta_1 <1/2$,  the summation gives 

\begin{eqnarray*}
\left| \Lambda(f_1, f_2, f_3)\right| \lesssim \sum _{d \geq0}\sum_{n_1\geq N_1(d)} \sum_{ n_2 \geq N_2(d)} \sum_{ n_3 \geq N_3(d)} 2^{-n_1 (1-2\theta_1)} 2^{-n_2 (1- \theta_2 -2\theta_3)} 2^{-n_3 (1-\theta_2)} |E_1| ^{\theta_1} |E_2|^{\theta_2 + \theta_3}
\end{eqnarray*}
Provided $\theta_1 <1/2, \theta_2 < 1, \theta_2 + 2\theta_3 <1$, we have the upper bound 

\begin{eqnarray*}
\left| \Lambda(f_1, f_2, f_3)\right| \lesssim \min\left\{ |E_1|^{\theta_1}, |E_1|^{1-\theta_1} \right\}  |E_2|^{1/2+\theta_2/2}.
\end{eqnarray*}
To produce restricted weak estimates in a neighborhood of $(1/2, 1/2, 0)$, set $\theta_1 = 1/2-\epsilon, \theta_2 = \epsilon, \theta_3 = 1/2 - \epsilon$. To do the same in neighborhoods of $(1, 0 , 0)$ and $(0,0,1)$ use $\theta_1 = 2\epsilon, \theta_2 = 1-3\epsilon, \theta_3 = \epsilon$. By interpolation, it suffices to prove estimates in a neighborhood of $(0,1/2, 1/2)$. To this end, assume $|E_1|=1$ and the exceptional set $\Omega$ attached to $f_1$ satisfies

\begin{eqnarray*}
\Omega \supset \left\{Mf_2 \gtrsim |E_2|^{1/2} \right\} \bigcup \left\{ M1_{E_3} \gtrsim |E_3| \right\}.
\end{eqnarray*}
As before, this exceptional set will have an acceptable size provided the implicit constants appearing in the above display are taken sufficiently large. 
The natural size restrictions are $2^{-n_3} \lesssim 2^d |E_3|^{\alpha}$ for any $0 \leq \alpha \leq 1$, $2^{-n_2} \lesssim 2^d |E_2|^{1/2}$, and $2^{-n_1} \lesssim 2^{-\tilde{N} d}$. A similar calculation as before yields for $\theta_1 <1/2, \theta_2 < 1, \theta_2 + 2 \theta_3 <1$, 

\begin{eqnarray*}
\left| \Lambda(f_1 1_{\Omega^c}, f_2, f_3) \right| \lesssim  |E_2|^{1/2 + \theta_2/2} \min\left\{ |E_3|^{\theta_2}, |E_3|^{1-\theta_2} \right\}.
\end{eqnarray*}
Choosing $\theta_1 = 1/2 - \epsilon, \theta_2 = 3\epsilon, \theta_3 = 1/2 -2 \epsilon $ yields the desired estimate for $\Lambda(f_1, f_2, f_3)$ near $(0, 1/2, 1/2)$ and therefore, by interpolation, the desired mixed estimates for $T_m$.

\end{proof}
Before moving on to the next section, we should remark that the symmetry of trilinear-form associated to the $BHT$ is not preserved in the degenerate case. Now, there is one "bad" index that does not provide frequency localization in the other indices. Hence, knowing the projection of a degenerate tri-tile onto the "bad" index does not uniquely determine that tri-tile. 

\section{$LWL$-Type Mixed Estimates for $B[a_1, a_2]$}
Our goal in the next 3 sections is to show 
\begin{theorem}\label{PT3}
$B[a_1, a_2]: L^{p_1} (\mathbb{R})\times W_{p_2}(\mathbb{R}) \times L^{p_3}(\mathbb{R}) \rightarrow L^{\frac{1}{\frac{1}{p_1} + \frac{1}{p_2} +\frac{1}{p_3}}}(\mathbb{R})$ provided
\begin{eqnarray*}
1 < p_1, p_2  < \infty, \frac{1}{p_1} + \frac{1}{p_2} <1, \frac{1}{p_2} +\frac{1}{p_3} <1, 2< p_2 <\infty.
 \end{eqnarray*}
In particular, $B[a_1, a_2]$ has mixed estimates into $L^r(\mathbb{R})$ for all $1/2< r <\infty$. 
\end{theorem}
 It would appear that $LWL$-type mixed estimates  for $B[a_1, a_2]$  are considerably more difficult to prove than the $WLW-$type variants shown in \S{7}. 

 \begin{proof}
 \subsection{Discretization}

To obtain mixed estimates for 
\begin{eqnarray*}
B[a_1, a_2] : \ (f_1, f_2, f_3) \mapsto \int_{\mathbb{R}^3} a_1(\xi_1,\xi_2) a_2(\xi_2, \xi_3) \hat{f}_1(\xi_1) \hat{f}_2(\xi_2) \hat{f}_3(\xi_3) e^{2 \pi i x (\xi_1 + \xi_2 + \xi_3)} d\vec{\xi},
\end{eqnarray*}
we adopt the general philosophy applied in \cite{MR2127985} to deal with the Biest, which is to localize the multiplier into distinct regions and then to bound the contribution of each region individually. 

These regions are 
\begin{eqnarray*}
\mathcal{R}^{1,1}_1&=&\left\{ -\xi_1 < \xi_2 , \xi_2 <  -\xi_3 , |\xi_1 + \xi_2| << |\xi_2 + \xi_3| \right\}\\ \mathcal{R}^{2,1}_1 &=& \left\{ -\xi_1 > \xi_2, \xi_2 < -\xi_3 , |\xi_1 + \xi_2| << |\xi_2 + \xi_3| \right\}\\  \mathcal{R}^{1,2}_1 &=& \left\{ -\xi_1 < \xi_2, \xi_2> -\xi_3,  |\xi_1 + \xi_2| << |\xi_2 + \xi_3|\right\} \\  \mathcal{R}^{2,2}_1 &=& \left\{-\xi_1 > \xi_2, \xi_2> -\xi_3, |\xi_1 + \xi_2| << |\xi_2 + \xi_3| \right\}  \\
\mathcal{R}^{1,1}_2&=&\left\{ -\xi_1 < \xi_2 , \xi_2 <  -\xi_3 , |\xi_1 + \xi_2| \simeq |\xi_2 + \xi_3| \right\}\\ \mathcal{R}^{2,1}_2 &=& \left\{ -\xi_1 > \xi_2, \xi_2 < -\xi_3 , |\xi_1 + \xi_2| \simeq |\xi_2 + \xi_3| \right\}\\  \mathcal{R}^{1,2}_2 &=& \left\{ -\xi_1 < \xi_2, \xi_2> -\xi_3,  |\xi_1 + \xi_2| \simeq  |\xi_2 + \xi_3|\right\} \\  \mathcal{R}^{2,2}_2 &=& \left\{-\xi_1 > \xi_2, \xi_2> -\xi_3, |\xi_1 + \xi_2| \simeq |\xi_2 + \xi_3| \right\}   \\ 
\mathcal{R}^{1,1}_3&=&\left\{ -\xi_1 < \xi_2 , \xi_2 <  -\xi_3 , |\xi_1 + \xi_2| >> |\xi_2 + \xi_3| \right\}\\ \mathcal{R}^{2,1}_3 &=& \left\{ -\xi_1 > \xi_2, \xi_2 < -\xi_3 , |\xi_1 + \xi_2| >> |\xi_2 + \xi_3| \right\}\\  \mathcal{R}^{1,2}_3 &=& \left\{ -\xi_1 < \xi_2, \xi_2> -\xi_3,  |\xi_1 + \xi_2| >> |\xi_2 + \xi_3|\right\} \\  \mathcal{R}^{2,2}_3 &=& \left\{-\xi_1 > \xi_2, \xi_2> -\xi_3, |\xi_1 + \xi_2| >> |\xi_2 + \xi_3| \right\} .
\end{eqnarray*}
It turns out that a wide range of $L^p$ estimates exists for multipliers localized to $\mathcal{R}_2^{1,1}, \mathcal{R}_2^{2,1}, \mathcal{R}_2^{1,2}, \mathcal{R}_2^{2,2}$, where the multiplier is adapted to $\{-\xi_1 = \xi_2= -\xi_3 \}$. Indeed, by results in \cite{MR1887641}, one checks the existence of restricted weak estimates near the desired extremal points in the collections $E_0, E_1, E_2$. Therefore, by symmetry, it will suffice to consider a generic multiplier of the form $\tilde{1}_{\mathcal{R}_1^{1,1}} (\xi_1, \xi_2, \xi_3) a_1(\xi_1, \xi_2) a_2(\xi_2, \xi_3)$, where $\tilde{1}_{\mathcal{R}_1^{1,1}} \equiv 1$ on a region of shape $\mathcal{R}_1^{1,1}$ and is supported inside a region of shape $\mathcal{R}_1^{1,1}$.  Carving $a_1(\xi_1, \xi_2) 1_{-\xi_1 < \xi_2}(\xi_1, \xi_2)$ and $a_2(\xi_2, \xi_3) 1_{\xi_2 <-\xi_3}(\xi_2, \xi_3)$ using Whitney  decompositions and then expanding bump functions adapted to Whitney cubes using double fourier series yields

\begin{eqnarray*}
&&a_1(\xi_1, \xi_2) 1_{\{ -\xi_1 < \xi_2\}}(\xi_1, \xi_2) = \sum_{(\gamma, \gamma^\prime) \in \{ 0, \frac{1}{3}, \frac{2}{3} \}^2} \sum_{k, k^\prime \in \mathbb{Z}} \sum_{\vec{Q} \in \mathbb{Q}^{\gamma, \gamma^\prime}} c_k \tilde{c}_{k^\prime} \hat{\eta}_{Q_1, 1}^{\gamma, k} (\xi_1) \hat{\eta}_{Q_2, 2}^{\gamma^\prime , k^\prime} (\xi_2) \\ &&
a_2(\xi_2, \xi_3) 1_{\{ \xi_2 < -\xi_3\}}(\xi_2, \xi_3) = \sum_{(\gamma, \gamma^\prime) \in \{ 0, \frac{1}{3}, \frac{2}{3} \}^2} \sum_{k, k^\prime \in \mathbb{Z}} \sum_{\vec{Q} \in \mathbb{Q}^{\gamma, \gamma^\prime}} d_k \tilde{d}_{k^\prime} \hat{\eta}_{Q_1, 1}^{\gamma, k} (\xi_2) \hat{\eta}_{Q_2, 2}^{\gamma^\prime , k^\prime} (\xi_3), 
\end{eqnarray*}
where $Q_2 \simeq Q_1 + C|\vec{Q}|$ and $\hat{\eta}^i_{I}$ is a bump function adapted in the H\"{o}rmander-Mikhlin sense to the interval $I$. Hence, omitting the dependence on the shifts $\gamma, \gamma^\prime$ and oscillation parameters $k, k^\prime$, it suffices to handle for generic bumps functions $\hat{\eta}_{I, j}$ the symbol 

\begin{eqnarray*}
&&\tilde{1}_{\mathcal{R}_1^{1,1}}(\xi_1, \xi_2, \xi_3) m(\xi_1, \xi_2, \xi_3)\\ &=&\tilde{1}_{\mathcal{R}_1^{1,1}}(\xi_1, \xi_2, \xi_3) \sum\left[   \sum_{\vec{Q} \in \mathcal{Q}}~ \sum_{\vec{P} \in \mathcal{P}} \hat{\eta}_{P_1, 1}(\xi_1) \hat{\eta}_{P_2, 2}(\xi_2)  \hat{\eta}_{Q_1, 3} (\xi_2) \hat{\eta}_{Q_2.4}(\xi_3) \right]\\ &=&\tilde{1}_{\mathcal{R}_1^{1,1}}(\xi_1, \xi_2, \xi_3) \sum \left[   \sum_{\vec{Q} \in \mathcal{Q}}~ \sum_{\vec{P} \in \mathcal{P} : |I_{\vec{P}}| \lesssim |I_{\vec{Q}}|} \hat{\eta}_{P_1, 1}(\xi_1) \hat{\eta}_{P_2, 2}(\xi_2)  \hat{\eta}_{Q_1, 3} (\xi_2) \hat{\eta}_{Q_2.4}(\xi_3) \right] .
\end{eqnarray*}
Furthermore, because $\tilde{1}_{\mathcal{R}_1^{1,1}} \equiv 1$ inside a region of the same shape, for large enough implicit constant

\begin{eqnarray*}
&&\tilde{1}_{\mathcal{R}_1^{1,1}}(\xi_1, \xi_2, \xi_3) \sum\left[   \sum_{\vec{Q} \in \mathcal{Q}}~ \sum_{\vec{P} \in \mathcal{P} : |I_{\vec{P}}| <<  |I_{\vec{Q}}|} \hat{\eta}_{P_1, 1}(\xi_1) \hat{\eta}_{P_2, 2}(\xi_2)  \hat{\eta}_{Q_1, 3} (\xi_2) \hat{\eta}_{Q_2.4}(\xi_3) \right] \\ &=& \sum\left[   \sum_{\vec{Q} \in \mathcal{Q}}~ \sum_{\vec{P} \in \mathcal{P} : |I_{\vec{P}}|  << |I_{\vec{Q}}|} \hat{\eta}_{P_1, 1}(\xi_1) \hat{\eta}_{P_2, 2}(\xi_2)  \hat{\eta}_{Q_1, 3} (\xi_2) \hat{\eta}_{Q_2.4}(\xi_3) \right]. 
\end{eqnarray*}
Lastly, note

\begin{eqnarray*}
\tilde{1}_{\mathcal{R}_1^{1,1}}(\xi_1, \xi_2, \xi_3) \sum \left[   \sum_{\vec{Q} \in \mathcal{Q}}~ \sum_{\vec{P} \in \mathcal{P} : |I_{\vec{P}}| \simeq  |I_{\vec{Q}}|} \hat{\eta}_{P_1, 1}(\xi_1) \hat{\eta}_{P_2, 2}(\xi_2)  \hat{\eta}_{Q_1, 3} (\xi_2) \hat{\eta}_{Q_2.4}(\xi_3) \right] 
\end{eqnarray*}
is adapted in the Mikhlin-H\"{o}rmander sense to $\{-\xi_1 = \xi_2 = -\xi_3\} \subset \mathbb{R}^3$ and so satisfies $L^p$ estimates. Therefore, it suffices to produce mixed estimates for generic symbols of the form 

\begin{eqnarray*}
 \sum_{\vec{Q} \in \mathcal{Q}}~ \sum_{\vec{P} \in \mathcal{P} : |I_{\vec{P}}|  << |I_{\vec{Q}}|} \hat{\eta}_{P_1, 1}(\xi_1) \hat{\eta}_{P_2, 2}(\xi_2)  \hat{\eta}_{Q_1, 3} (\xi_2) \hat{\eta}_{Q_2.4}(\xi_3).
\end{eqnarray*}
Dualizing and completing yields the $4-$form

\begin{eqnarray*}
\Lambda(f_1, f_2, f_3, f_4) :=  \int_\mathbb{R} \sum_{\vec{Q} \in \mathcal{Q}} \sum_{\vec{P} \in \mathcal{P} : |I_{\vec{P}}| << |I_{\vec{Q}}|} \left[ f_4*\eta_{-P_1, 4} f_1*\eta_{P_1, 1} \right] *\eta_{|\vec{Q}|, 0}^{lac} \cdot  f_2*\eta_{P_2, 2}*  \eta_{Q_1, 3}\cdot  f_3*  \eta_{Q_2, 4}dx .
\end{eqnarray*}
Discretizing in time then gives that $\Lambda(f_1, f_2, f_3,f_4)$ can be written as a sum of expressions of the form 

\begin{eqnarray*}
 \int_0^1 \int_0^1   \sum_{\vec{Q} \in \mathbb{Q}}~  \sum_{\vec{P} \in \mathbb{P} : \omega_{P_2} \supset \supset \omega_{Q_1}}\frac{1}{|I_{\vec{Q}}|^{1/2}|I_{\vec{P}}|^{1/2}}   \langle f_4, \Phi^\alpha_{-P_1,4} \rangle \langle f_1, \Phi^\alpha_{P_1,1} \rangle \left\langle \Phi_{P_3,5}^{\alpha} , \Phi^{\alpha^\prime}_{Q_3, 6} \right\rangle \langle f_2 *\eta_{P_2, 2}, \Phi^{\alpha^\prime}_{Q_1,2} \rangle \langle f_3, \Phi^{\alpha^\prime}_{Q_2,3} \rangle d \alpha d \alpha^\prime,
\end{eqnarray*} 
where $\mathbb{P}$ and $\mathbb{Q}$ are collections of tri-tiles adapted to the degenerate line $\{ \xi_1 + \xi_2 = 0\}$ when viewed only in their first two entries.  In particular, for each $\vec{P} = (P_1, P_2, P_3 ) \in \mathbb{P}$ and $\vec{Q} = (Q_1, Q_2, Q_3) \in \mathbb{Q}$, $P_j= (I_{\vec{P}}, \omega_{P_j})$, $Q_j = (I_{\vec{Q}}, \omega_{Q_j})$ is a tile for $j\in \{1, 2, 3 \}$, $\Phi_{T,j}, $ is a wave-packet on the tile $T$ for each $j \in \{ 1, 2, 3, 4, 5,6\}$, $\omega_{P_3}= [-|I_{\vec{P}}|^{-1} /2, |I_{\vec{P}}|^{-1}/2 ]$, and $\omega_{Q_3} = [c|I_{\vec{Q}}|^{-1}, C |I_{\vec{Q}}|^{-1} ]$ for some $0<c<C$ fixed.   
\section{Estimates for a Toy Model}
It will be useful for us to first prove estimates for the simpler model $\Lambda_{Toy}^{k_0}(f_1, f_2, f_3, f_4)$ defined by 

\begin{eqnarray*}
  \sum_{\vec{Q} \in \mathbb{Q}} ~\sum_{\vec{P} \in \mathbb{P} : |I_{\vec{P}}| =2^{-k_0} |I_{\vec{Q}}|, \omega_{P_2} \supset \supset \omega_{Q_1}}\frac{1}{|I_{\vec{Q}}| |I_{\vec{P}}|}  \left|  \langle f_4, \Phi_{-P_1,0} \rangle \langle f_1, \Phi_{P_1,1} \rangle \langle \tilde{1}_{I_{\vec{P}}}, \tilde{1}_{I_{\vec{Q}}}\rangle \langle f_2, \Phi_{Q_1,2} \rangle \langle f_3, \Phi_{Q_2,3} \rangle \right|
\end{eqnarray*}
with an operatorial norm that grows like $2^{2k_0}$. As $k_0 \simeq 0$ corresponds to a model adapted to $\{-\xi_1 = \xi_2  = - \xi_3 \} \subset \mathbb{R}^3$, the statement only needs to be proven for $k_0 >>1$. 
To this end, observe that the main contribution to the above sum occurs for those pairs of tri-tiles $(\vec{Q}, \vec{P}) \in \mathbb{Q} \times \mathbb{P}$ for which $I_{\vec{P}} \subset I_{\vec{Q}}$. Indeed, we now make this heuristic rigorous by showing that it suffices to produce estimates for generic sums of the above form for which $I_{\vec{P}} \subset I_{\vec{Q}}$ also holds. Begin by noting that whenever $|I_{\vec{P}}| = 2^{-k_0} |I_{\vec{Q}}|$, 

\begin{eqnarray*}
\langle \Phi^{n-l}_{|\vec{P}|}, \Phi_{\vec{Q}}^{lac} \rangle \lesssim \frac{1}{|I_{\vec{P}}|^{1/2} |I_{\vec{Q}}|^{1/2}}  \langle\tilde{1}_{I_{\vec{P}}}, \tilde{1}_{I_{\vec{Q}}} \rangle \lesssim_N \frac{|I_{\vec{P}}|^{1/2}}{|I_{\vec{Q}}|^{1/2}} \frac{1}{1+\left( \frac{ dist(I_{\vec{P}}, I_{\vec{Q}})}{|I_{\vec{Q}}|}\right)^N}.
\end{eqnarray*}
Therefore, denote $\mathbb{P}_{\vec{Q}}(l) = \left\{ \vec{P} \in \mathbb{P} : 1+\frac{ dist(I_{\vec{P}}, I_{\vec{Q}})}{|I_{\vec{Q}}|} \simeq 2^l \right\}$ and $\mathbb{Q}_{\vec{P}}(l)= \left\{ \vec{Q} \in \mathbb{Q} : 1+\frac{ dist(I_{\vec{P}}, I_{\vec{Q}})}{|I_{\vec{Q}}|} \simeq 2^l \right\}$. Then 
\begin{eqnarray*}
  &&\sum_{\vec{Q} \in \mathbb{Q}} ~\sum_{\vec{P} \in \mathbb{P} : |I_{\vec{P}}| = 2^{-k_0} |I_{\vec{Q}}|, \omega_{P_2} \supset \omega_{Q_1}}\frac{1}{|I_{\vec{Q}}|^{1/2}|I_{\vec{P}}|^{1/2}}  \left|  \langle f_4, \Phi_{-P_1,0} \rangle \langle f_1, \Phi_{P_1,1} \rangle \left\langle \Phi_{|\vec{P}|}^{n-l} , \Phi^{lac}_{|\vec{Q}|} \right \rangle \langle f_2, \Phi_{Q_1,2} \rangle \langle f_3, \Phi_{Q_2,3} \rangle \right| \\ &\lesssim&  \sum_{l \geq 0} \frac{1}{1+ l^N}  \sum_{\vec{P} \in \mathbb{P}} ~\sum_{\vec{Q} \in \mathbb{Q}_{\vec{P}}(l) : |I_{\vec{Q}}| =2^{k_0} |I_{\vec{P}}|,  \omega_{Q_1} \subset \omega_{P_2}}\frac{1}{|I_{\vec{Q}}|}  \left|  \langle f_4, \Phi_{-P_1,0} \rangle \langle f_1, \Phi_{P_1,1} \rangle \langle f_2, \Phi_{Q_1,2} \rangle \langle f_3, \Phi_{Q_2,3} \rangle \right| \\ &\lesssim& \sum_{l \geq 0} \frac{1+l^M}{1+ l^N}  \sum_{\vec{P} \in \mathbb{P}} ~\sum_{\vec{Q} \in \mathbb{Q}: |I_{\vec{Q}}| =2^{k_0} |I_{\vec{P}}|,  \omega_{Q_1} \subset \omega_{P_2}, I_{\vec{Q}} \supset I_{\vec{P}}}\frac{1}{|I_{\vec{Q}}|}  \left|  \langle f_4, \Phi_{-P_1,0} \rangle \langle f_1, \Phi_{P_1,1} \rangle \langle f_2, \Phi^l_{Q_1,2} \rangle \langle f_3, \Phi^l_{Q_2,3} \rangle \right|
\end{eqnarray*}
for some $1 << M << N$. It therefore suffices to prove generic mixed estimates for $\tilde{\Lambda}^{k_0}_{Toy}(f_1, f_2,f_3, f_4)$ given by
\begin{eqnarray*}
 \sum_{\vec{Q} \in \mathbb{Q}} ~\sum_{\vec{P} \in \mathbb{P}: |I_{\vec{P}}| =2^{-k_0} |I_{\vec{P}}|,  \omega_{P_2} \supset \supset \omega_{Q_1}, I_{\vec{P}} \subset I_{\vec{Q}}}\frac{1}{|I_{\vec{Q}}|} \left|  \langle f_4, \Phi_{-P_1,0} \rangle \langle f_1, \Phi_{P_1,1} \rangle \langle f_2, \Phi_{Q_1,2} \rangle \langle f_3, \Phi_{Q_2,3} \rangle \right|.
\end{eqnarray*}

\subsection{Sizes and Energies}
By scaling invariance, we shall assume $|E_4|=1$. Moreover, let

\begin{eqnarray*}
\Omega := \left\{ M1_{E_1} \geq C |E_1| \right\}  \bigcup \left\{ M1_{E_3} \geq C |E_3| \right\} \bigcup \left\{ M f_2 \geq C |E_2|^{1/2} \right\}. 
\end{eqnarray*}
As usual, choose $C \in \mathbb{R}$ large enough to ensure $|\Omega| \leq 1/2$ in which case $\tilde{E}_4 := E_4 \cap \Omega^c$ is a major subset of $E_4$. 
\begin{definition}
For $d, \tilde{d} \in \mathbb{Z}_{ \geq0}$, let
 $\mathbb{P}^d = \left\{ \vec{P} \in \mathbb{P}: 1+ \frac{dist(I_{\vec{P}}, \Omega^c)}{|I_{\vec{P}}|} \simeq 2^d \right\}; \mathbb{Q}^{\tilde{d}} := \left\{ \vec{Q} \in \mathbb{Q} : 1+ \frac{dist(I_{\vec{Q}}, \Omega^c)}{|I_{\vec{Q}}|} \simeq 2^{\tilde{d}} \right\}$. 
 \end{definition}
 \begin{definition}
For each measurable set $E \subset \mathbb{R}$,  dyadic interval $I \in \mathcal{D}$, and $j \in \{1,4\}$, set
\begin{eqnarray*}
SIZE(E, I) := \frac{1}{|I|} \int_E \tilde{1}_{I} dx = \sup_{\omega \in \mathcal{D} : |\omega| \cdot |I| =1 } \left[ \frac{1}{|I|}  \sum_{\vec{P} \in \mathbb{P} : |I_{\vec{P}}| \lesssim 2^{-k_0} |I|, I_{\vec{P}} \subset I, \omega_{P_2} \supset \omega} \langle f_1, \Phi_{P_1, 1} \rangle|^2 \right]^{1/2}.
\end{eqnarray*}
\end{definition}

\begin{definition}
For each $n_1 \in \mathbb{Z}$, let $\mathbb{I}^d_{n_1,1}$ denote the set of dyadic intervals $\{I\}$ that are maximal with respect to the property

\begin{eqnarray*}
SIZE(E_1, I) &\geq &2^{-n_1} 
\end{eqnarray*}
Similarly, for each $n_4 \in \mathbb{Z}$, let $\mathbb{I}^d_{n_4,4}$ denote the set of dyadic intervals $\{I\}$ that are maximal with respect to the property

\begin{eqnarray*}
SIZE(E_4 \cap \Omega^c, I) &\geq& 2^{-n_4}.
\end{eqnarray*}

\end{definition}

Therefore, $\left\{ \mathbb{I}_{n_1}\right\}_{n_1 \geq 0}$ and $\left\{ \mathbb{I}_{n_4} \right\}_{n_4 \geq 0}$ generate two decompositions of $\mathbb{Q}$ using the recursive definitions 
\begin{eqnarray*}
\mathbb{Q}_{n_1,1}  &:=& \left\{ \vec{Q} \in \mathbb{Q} : I_{\vec{Q}} \subset I \in \mathbb{I}_{n_1,1 } \right\}\cap \left[ \bigcup_{N_1(d)\leq  \tilde{n}_1 < n_1} \left\{ \vec{Q} \in \mathbb{Q}: I_{\vec{Q}} \subset I \in \mathbb{I}_{\tilde{n}_1,1} \right\} \right]^c \\ 
\mathbb{Q}_{n_4, 4} &:=&  \left\{ \vec{Q} \in \mathbb{Q} : I_{\vec{Q}} \subset I \in \mathbb{I}_{n_4,4 } \right\}\cap \left[ \bigcup_{N_4(d)\leq  \tilde{n}_4 < n_4} \left\{ \vec{Q} \in \mathbb{Q}: I_{\vec{Q}} \subset I \in \mathbb{I}_{\tilde{n}_4,4} \right\} \right]^c .
\end{eqnarray*}
Lastly, define 

\begin{eqnarray*}
\mathbb{Q}^{\tilde{d}}_{n_1, n_4}:= \mathbb{Q}_{n_1,1} \cap \mathbb{Q}_{n_4,4} \cap \mathbb{Q}^{\tilde{d}}.
\end{eqnarray*}
It follows that $\mathbb{Q} = \bigcup_{ \tilde{d} \geq 0} \bigcup_{n_1 \geq N_1(\tilde{d})} \bigcup_{n_4 \geq N_4(\tilde{d})} \mathbb{Q}^{ \tilde{d}}_{n_1, n_4}$. 
Setting $\mathbb{I}_{n_1, n_4} = \left\{ I \cap J : I \in \mathbb{I}_{n_1, 1}, J \in \mathbb{I}_{n_4, 4} \right\}$ and $\mathbb{Q}^{\tilde{d}}_{n_1, n_4} [I] = \left\{ \vec{Q} \in \mathbb{Q}^{ \tilde{d}}_{n_1, n_4} : I_{\vec{Q}} \subset I \right\}$, note

\begin{eqnarray*}
\mathbb{Q}^{ \tilde{d}}_{n_1, n_4} = \bigcup_{I \in \mathbb{I}_{n_1, n_4}} \mathbb{Q}^{ \tilde{d}}_{n_!, n_4} [I].
\end{eqnarray*}

\begin{definition}
A collection of (dyadic) intervals $\mathcal{P}:= \{\omega_P\}$ is said to be lacunary provided about $\Omega \in \mathbb{R}$ provided\begin{eqnarray*}
 \omega_P - \Omega_T \subset \left[ |\omega_P| 2^{-C_*}, |\omega_P| 2^{C_*} \right] ~\forall ~\omega_P \in \mathcal{P}
 \end{eqnarray*}
for some fixed $C_{*}>>1$. 
\end{definition}

\begin{definition}
For a given $I \in \mathbb{I}_{n_1, n_4}$,  $j \in \{2,3\}$, let
\begin{eqnarray*}
SIZE_{n_1, n_4}^{\tilde{d}}(f_j, I):= \sup_{T : I_T \subset I, } \frac{1}{|I_T|} \left( \sum_{\vec{Q} \in T \cap \mathbb{Q}_{n_1, n_4}^{\tilde{d}}} |\langle f_j, \Phi_{Q_2} \rangle|^2 \right)^{1/2}
\end{eqnarray*}
 generate another decomposition localized to $I$ associated to the indices $n_2, n_3$ as follows: 

\begin{eqnarray*}
\mathbb{Q}^{ \tilde{d}}_{n_1, n_4} [I]=\bigcup_{n_2} \left[ \mathbb{Q}^{ \tilde{d}}_{n_1, n_4}  [I] \right]_{n_2,2}= \bigcup_{n_3} \left[ \mathbb{Q}_{n_1, n_4} [I] \right]_{n_3,3}.\\
\end{eqnarray*}

\end{definition}
Putting it all together using the standard BHT-type strongly disjoint tree decomposition from \cite{MR2127985}, say, we have tile decompositions for $\mathbb{Q}$:

\begin{eqnarray*}
\mathbb{Q} &=& \bigcup_{\tilde{d} \geq 0} \mathbb{Q}^{\tilde{d}} = \bigcup_{ \tilde{d} \geq 0}\bigcup_{n_1, n_4}  \bigcup_{I \in \mathbb{I}_{n_1, n_4}}\mathbb{Q}^{ \tilde{d}}_{n_1, n_4} [I]   = \bigcup_{\tilde{d} \geq 0}\bigcup_{n_1, n_4}   \bigcup_{I \in \mathbb{I}_{n_1, n_4}}\bigcup_{n_2, n_3}\left[ \mathbb{Q}^{\tilde{d}}_{n_1, n_4}[I] \right]_{n_2}^{n_3}
\end{eqnarray*}
where

\begin{eqnarray*}
 \left[\mathbb{Q}^{ \tilde{d}}_{n_1, n_4}[I] \right]_{n_2}^{n_3} :=  \left[ \mathbb{Q}^{\tilde{d}}_{n_1, n_4} [I] \right]_{n_2}\cap  \left[ \mathbb{Q}^{ \tilde{d}}_{n_1, n_4} [I] \right]^{n_3} = \bigcup_{T_2 \in  \mathcal{T}\left[ \mathbb{Q}^{\tilde{d}}_{n_1, n_4} [I] \right]_{n_2}} \bigcup_{T_3 \in \mathcal{T} \left[ \mathbb{Q}^{ \tilde{d}}_{n_1, n_4}[I] \right]^{n_3}}     T_2 \cap T_3
 \end{eqnarray*}
  and the collection of trees $\mathcal{T}\left[ \mathbb{Q}^{\tilde{d}}_{n_1, n_4} [I] \right]_{n_2}$ and $\mathcal{T} \left[ \mathbb{Q}^{ \tilde{d}}_{n_1, n_4}[I] \right]^{n_3}$ have subcollections $\mathcal{T}_*\left[ \mathbb{Q}^{\tilde{d}}_{n_1, n_4} [I] \right]_{n_2}$ and $\mathcal{T}_* \left[ \mathbb{Q}^{ \tilde{d}}_{n_1, n_4}[I] \right]^{n_3}$ which are strongly $2-$ and $3-$disjoint respectively and for which the following energy-type estimates hold:

\begin{eqnarray*}
\sum_{T_2 \in \mathcal{T}_* \left[  \mathbb{Q}^{ \tilde{d}}_{n_1, n_4} [I] \right]_{n_2}} |I_{T_2}| &\lesssim& 2^{2n_2} |E_2| \\
\\ 
\sum_{T_3 \in \mathcal{T}_* \left[ \mathbb{Q}^{ \tilde{d}}_{n_1, n_4}[I] \right]^{n_3}} |I_{T_3}| &\lesssim& 2^{2n_3} |E_3| .
\end{eqnarray*} 
To summarize, we have assembled a decomposition of the set $\mathbb{P} \times \mathbb{Q}:$
\begin{eqnarray*}
\mathbb{P} \times \mathbb{Q} = \bigcup_{\tilde{d} \geq 0}\bigcup_{n_1, n_4}   \bigcup_{I \in \mathbb{I} _{n_1, n_4}}\bigcup_{n_2, n_3}  \left(  \mathbb{P} \times \left[ \mathbb{Q}^{ \tilde{d}}_{n_1, n_4}[I] \right]_{n_2}^{n_3} \right).
\end{eqnarray*}

\subsection{Splitting $\tilde{\Lambda}^{k_0}_{Toy}$}
We now break apart our model as follows: 

\begin{eqnarray*}
&& \tilde{\Lambda}^{k_0}_{Toy}(f_1, f_2, f_3, f_4)\\ &=& \sum_{\vec{Q} \in \mathbb{Q}} ~\sum_{\vec{P} \in \mathbb{P}: |I_{\vec{P}}| =2^{-k_0} |I_{\vec{Q}}|, \omega_{P_2} \supset \supset \omega_{Q_1}, I_{\vec{P}} \subset I_{\vec{Q}}} \frac{1}{|I_{\vec{Q}}|}    \left|  \langle f_4, \Phi_{P_1,4} \rangle \langle f_1, \Phi_{-P_1,1} \rangle \langle f_2, \Phi_{Q_1,2} \rangle \langle f_3, \Phi_{Q_2,3} \rangle \right|\\ &=& \sum_{ \tilde{d} \geq 0} \sum_{n_1, n_2, n_3, n_4}  \sum_{I \in \mathbb{I}_{n_1, n_4} } \sum_{\vec{Q} \in \left[ \mathbb{Q}^{ \tilde{d}}_{n_1, n_4}[I] \right]_{n_2}^{n_3} }\sum_{\vec{P} \in \mathbb{P}^d, |I_{\vec{P}}| =2^{-k_0} |I_{\vec{Q}}|, \omega_{P_2} \supset \supset\omega_{Q_1}, I_{\vec{P}} \subset I_{\vec{Q}}}\Lambda_{\vec{Q}, \vec{P}}(f_1, f_2, f_3, f_4),\end{eqnarray*}
where 

\begin{eqnarray*}
\Lambda_{\vec{Q}, \vec{P}}(f_1, f_2, f_3, f_4) := \frac{1}{|I_{\vec{Q}}|}    \left|  \langle f_4, \Phi_{P_1,4} \rangle \langle f_1, \Phi_{-P_1,1} \rangle \langle f_2, \Phi_{Q_1,2} \rangle \langle f_3, \Phi_{Q_2,3} \rangle \right|. 
\end{eqnarray*}

\subsection{Toy Model Tree Estimate}
\begin{lemma}
Let $T$ be a $\mathbb{Q}-$tree satisfying $T \subset \left[[\mathbb{Q}^{ \tilde{d}}_{n_1, n_4}[I] \right]_{n_2}^{n_3}$. Then 
\begin{eqnarray*}
\sum _{\vec{Q} \in T}~ \sum_{\vec{P} \in \mathbb{P}^d, |I_{\vec{P}}| =2^{-k_0} |I_{\vec{Q}}|, \omega_{P_2} \supset \supset \omega_{Q_1}, I_{\vec{P}} \subset I_{\vec{Q}}}\Lambda_{\vec{Q}, \vec{P}}(f_1, f_2, f_3, f_4) \lesssim 2^{-n_1} 2^{-n_2} 2^{-n_3} 2^{-n_4} |I_T|.
\end{eqnarray*}
\end{lemma}
\begin{proof}

WLOG, assume $\left\{Q_1\right\}_{\vec{Q} \in T}$ are overlapping as the case when $\left\{Q_2 \right\}_{\vec{Q} \in T}$ is overlapping is similar. Then it suffices to note that for large enough implicit constant in $|I_{\vec{P}}| << |I_{\vec{Q}}|$, the assumption that $\vec{Q}$ lie on a single tree ensures that 

\begin{eqnarray*}
\left\{\omega_{P_2}\right\}_{\vec{P} \in \mathbb{P}: |I_{\vec{P}}| =2^{-k_0} |I_{\vec{Q}}|, \omega_{P_2} \supset \supset \omega_{Q_1}, I_{\vec{P}} \subset I_{\vec{Q}}~ for~some~ \vec{Q} \in T}
\end{eqnarray*}
is overlapping. Therefore, $\left\{\omega_{P_1}\right\}_{\vec{P} \in \mathbb{P}: |I_{\vec{P}}| =2^{-k_0} |I_{\vec{Q}}|, \omega_{P_2} \supset \supset \omega_{Q_1}, I_{\vec{P}} \subset I_{\vec{Q}}~ for~some~ \vec{Q} \in T} := \Omega\{ \mathcal{T} \}$ forms a lacunary sequence and setting $\mathbb{P}_{n_1, n_2}= \left\{ \mathbb{P}\in \mathbb{P} : \exists \vec{Q} \in \mathbb{Q}_{n_1, n_2} , I_{\vec{Q}} \supset I_{\vec{P}},  |I| =2^{k_0} |I_{\vec{P}}| \right\}$ yields
\begin{eqnarray*}
&& \sum _{\vec{Q} \in T}~ \sum_{\vec{P} \in \mathbb{P}: |I_{\vec{P}}| = 2^{-k_0}|I_{\vec{Q}}|, \omega_{P_2} \supset \omega_{Q_1}, I_{\vec{P}} \subset I_{\vec{Q}}}\Lambda_{\vec{Q}, \vec{P}}(f_1, f_2, f_3, f_4)  \\ &\leq &\sup_{\vec{Q} \in T}  \left[ \frac{1}{|I_{\vec{Q}}|} \sum_{\vec{P} \in \mathbb{P}^d:  I_{\vec{P}} \subset I_{\vec{Q}}, \omega_{P_1} \in \Omega\{\mathcal{T}\}}  |\langle f_4, \Phi_{-P_1, 4} \rangle|^2 \right]^{1/2} \left(\frac{1}{|I_T|} \sum_{\vec{P} \in \mathbb{P}^d:  I_{\vec{P}} \subset I_T, \omega_{P_2} \in \Omega\{ \mathcal{T}\}}  |\langle f_1, \Phi_{-P_1, 1} \rangle|^2 \right)^{1/2}   \\ &\times& \left[ \sup_{\vec{Q} \in T} \frac{ |\langle f_2, \Phi_{Q_1}^2 \rangle|}{|I_{\vec{Q}}|^{1/2}} \right]\left(\frac{1}{|I_T|} \sum_{\vec{Q} \in T_2}  |\langle f_3, \Phi_{Q_2} \rangle|^2 \right)^{1/2} \cdot |I_T| \\ &\lesssim& \left[ \sup_{\vec{P} \in \mathbb{P}_{n_1, n_4}} \frac{1}{|I_{\vec{P}}|}\int_{E_1} \tilde{1}_{I_{\vec{P}}} dx\right] \cdot \left[ \sup_{\vec{P} \in \mathbb{P}_{n_1, n_4}} \frac{1}{|I_{\vec{P}}|}\int_{E_4\cap \Omega^c}  \tilde{1}_{I_{\vec{P}}} dx\right] 2^{-n_2} 2^{-n_3} |I_T|\\&\lesssim&2^{2k_0}  2^{-n_2} 2^{-n_3} \left[ \sup_{I \in \mathbb{I}_{n_1, n_4}} \frac{1}{|I|}\int_{E_1} \tilde{1}_{I} dx\right] \cdot \left[ \sup_{I \in \mathbb{I}_{n_1, n_4}} \frac{1}{|I|}\int_{E_4\cap \Omega^c}  \tilde{1}_{I} dx\right] \\&\lesssim&  2^{2k_0}2^{-n_1} 2^{-n_2} 2^{-n_3} 2^{-n_4} |I_T|.
\end{eqnarray*}

\end{proof}

\subsection{Toy Model Energy Estimate}
The following lemma forms the core of our analysis and is one of the main reasons why we have mixed estimates. 
\begin{lemma}
The following estimate holds:
\begin{eqnarray*}
\sum_{T_2 \in\mathcal{T}_* \left[ \mathbb{Q}^{ \tilde{d}}_{n_1, n_4} [I] \right]_{n_2}} |I_{T_2}| \lesssim 2^{n_2}|E_2| \cdot \left| I \right|.
\end{eqnarray*}
\end{lemma}

\begin{proof}
By definition, for every $T_2 \in \mathcal{T}_* \left[ \mathbb{Q}^{ \tilde{d}}_{n_1, n_4}[I] \right]_{n_2}$, we have

\begin{eqnarray*}
|I_{T_2}| \lesssim 2^{2n_2} \sum_{\vec{Q} \in T_2} |\langle f_2, \Phi_{Q_1, 2} \rangle|^2
\end{eqnarray*}
where the tiles $Q_2$ are lacunary around some top frequency. Using strong $2-$disjointness of the trees $\left\{T_2 \right\} $ yields

\begin{eqnarray*}
\sum_{T_2\in \mathcal{T}_* \left[ \mathbb{Q}^{ \tilde{d}}_{n_1, n_4}[I] \right]_{n_2}} |I_{T_2}|&\lesssim& 2^{2n_2}\sum_{T_2 \in  \mathcal{T}_* \left[ \mathbb{Q}^{\tilde{d}}_{n_1, n_4}[I] \right]_{n_2}}\sum_{\vec{Q} \in T_2} |\langle f_2, \Phi_{Q_1, 2} \rangle|^2\\ &\lesssim& 2^{n_2} \sum_{T_2 \in \mathcal{T}_* \left[ \mathbb{Q}^{ \tilde{d}}_{n_1, n_4}[I] \right]_{n_2}}\sum_{\vec{Q} \in T_2} |\langle f_2, \Phi^\infty_{Q_1, 2} \rangle| \\ &\leq& 2^{n_2} \sum_{T_2 \in  \mathcal{T}_* \left[ \mathbb{Q}^{ \tilde{d}}_{n_1, n_4}[I] \right]_{n_2}}\sum_{\vec{Q} \in T_2} \left| \left|  f_2 * \eta_{\omega_{Q_1}} \right| \right|_\infty |I_{\vec{Q}}| \\&\leq& 2^{n_2}  \int_I \sum_{T_2 \in  \mathcal{T}_*\left[ \mathbb{Q}^{ \tilde{d}}_{n_1, n_4}[I] \right]_{n_2}}\sum_{\vec{Q} \in T_2} \left| \left|  \hat{f}_2 1_{\omega_{Q_1}} \right| \right|_11_{I_{\vec{Q}}}(x) dx \\ &\lesssim &2^{n_2} |E_2| \cdot |I|.
\end{eqnarray*}
\end{proof}
We shall also need the following elementary result:
\begin{lemma}

\begin{eqnarray*}
\sum_{I \in \mathbb{I}_{n_1, n_4}} |I| &\leq &\sum_{I \in \mathbb{I}_{n_1}} |I| \lesssim \left| \left\{ M1_{E_1} \gtrsim 2^{-n_1} \right\} \right| \lesssim 2^{n_1} |E_1|\\
\sum_{I \in \mathbb{I}_{n_1, n_4}} |I| &\leq& \sum_{I \in \mathbb{I}_{n_4}} |I| \lesssim \left| \left\{ M1_{E_4} \gtrsim 2^{-n_4} \right\} \right| \lesssim 2^{n_4}.
\end{eqnarray*}
\end{lemma}
\begin{proof}
Immediate from the definitions. 
\end{proof}

Putting it all together, we find

\begin{eqnarray*}
&& \sum_{\vec{Q} \in \left[ \mathbb{Q}^{ \tilde{d}}_{n_1, n_4}[I] \right]_{n_2}^{n_3} }~~\sum_{\vec{P} :|I_{\vec{P}}| =2^{-k_0} |I_{\vec{Q}}|, \omega_{P_2} \supset \supset\omega_{Q_1}, I_{\vec{P}} \subset I_{\vec{Q}}}\Lambda_{\vec{Q}, \vec{P}}(f_1, f_2, f_3, f_4) \\ &\lesssim& 2^{-n_1} 2^{-n_2} 2^{-n_3} 2^{-n_4} \min \{ 2^{2n_2} |E_2 \cap I | , 2^{n_2} |E_2| \cdot |I| , 2^{2n_3} |E_3 \cap I| \},
\end{eqnarray*}
and using the two separate estimates for $\sum_{I \in \mathbb{I}_{n_1, n_4}} |I|$ yields 
\begin{eqnarray*}
&&\tilde{ \Lambda}^{k_0}_{Toy} (f_1, f_2, f_3, f_4)\\ &\lesssim& \sum_{ \tilde{d} \geq 0} \sum_{n_1 \geq N_1(d)} \sum_{n_2 \geq N_2(\tilde{d}) }\sum_{n_3 \geq N_3(\tilde{d}) } \sum_{n_4 \geq N_4(d)} 2^{-n_1} 2^{-n_2} 2^{-n_3} 2^{-n_4} \min \{ 2^{2n_2} |E_2| , 2^{n_2} |E_2| \min \{ 2^{n_4} ,  2^{n_1} |E_1|\}, 2^{2n_3} |E_3| \}.
\end{eqnarray*}
The natural size restrictions are easily seen to be $2^{-n_1} \lesssim \min \{1, 2^d |E_1|\}$, $2^{-n_2} \lesssim \min \{ 2^{\tilde{d}} |E_2|^{1/2}, 2^{\tilde{d}} |E_2|\}, 2^{-n_3} \lesssim \min \{ 1, 2^{\tilde{d}} |E_3| \}, 2^{-n_4} \lesssim 2^{-\tilde{N} d}$. Therefore, the above sum is summable, and a range of mixed estimates are available by interpolation. As the numerics are similar to that found in the main model, we postpone a detailed examination of exactly what these estimates are.

\section{Main Model}
Recall the main model $\Lambda(f_1, f_2, f_3, f_4)$ given by the formula

\begin{eqnarray*}
  \sum_{\vec{Q} \in \mathbb{Q}} ~ \sum_{\vec{P} \in \mathbb{P} : |I_{\vec{P}}| << |I_{\vec{Q}}|, \omega_{P_2} \supset \supset \omega_{Q_1}}\frac{1}{|I_{\vec{Q}}|^{1/2}|I_{\vec{P}}|^{1/2}}   \langle f_4, \Phi_{-P_1,4} \rangle \langle f_1, \Phi_{P_1,1} \rangle \left\langle \Phi_{|\vec{P}|}^{n-l} , \Phi^{lac}_{|\vec{Q}|} \right\rangle \langle f_2 *\eta_{P_2, 2}, \Phi_{Q_1,2} \rangle \langle f_3, \Phi_{Q_2,3} \rangle.
\end{eqnarray*}
Before proceeding with details, let us pause to sketch the idea of the remaining proof. The main difficulty in bounding the above expression turns on the factor $\langle f_2 *\eta_{P_2, 2}, \Phi_{Q_2, 1} \rangle$ because mixing $\vec{Q}$ and $\vec{P}$ in this way may work against the orthogonality introduced by $\Phi^{lac}_{|\vec{Q}|}$. Therefore, our goal is to split $\mathbb{Q}$ and $\mathbb{P}$ using a more cumbersome decomposition than in the toy model before. Then, we shall see rewrite the above expression localized to a single $\mathbb{Q}-$tree as a sum of two main pieces, say $A$ and $B$. Then $A$ is  paracomposition, which gives us the desired orthogonality and therefore estimates. The remainder $B$ can be written as an infinite sum over rapidly decaying pieces, each one of which can be reformulated as a toy model $\Lambda_{Toy}^{k_0}$ for some parameter $k_0>> 1 $ as before. Hence, our proceeding work will enable us to successfully estimate $B$, and with it, the main model $\Lambda$. 

To begin, assume by scaling invariance that $|E_4|=1$. Moreover, let 

\begin{eqnarray*}
\Omega:= \left\{ M1_{E_1} \geq C |E_1| \right\} \bigcup \left\{ M1_{E_3} \geq C |E_3| \right\} \bigcup \left\{ M f_2 \geq C|E_2|^{1/2} \right\}. 
\end{eqnarray*}
Recall $\mathbb{P}^d = \left\{ \vec{P} \in \mathbb{P}: 1+\frac{dist(I_{\vec{P}}, \Omega^c)}{|I_{\vec{P}}|} \simeq 2^d\right\}; \mathbb{Q}^{\tilde{d}} := \left\{ \vec{Q} \in \mathbb{Q} : 1+ \frac{dist(I_{\vec{Q}}, \Omega^c)}{|I_{\vec{Q}}|} \simeq 2^{\tilde{d}} \right\}$.

\begin{definition}
For each $\mathbb{Q}-$tree $T$ with top frequency centered around $c_{\omega_{T}}$ in the $1st$ index, let $\Phi^{n-l, T}_{\vec{P}}(x) := \hat{\eta}_{P_2,2}(c_{\omega_T}) \Phi^{n-l}_{|\vec{P}|}(x)$. 
\end{definition}
\begin{definition}
For each $\mathbb{Q}-$tree $T$, let $\mathbb{P}(T ) :=\left\{ \vec{P} \in \mathbb{P}: \omega_{P_2} \supset \supset \omega_{Q_1}~for~some~ \vec{Q} \in T \right\}$.
\end{definition}
It is important to observe that $\Omega(\mathbb{P}(T))_1 := \left\{ \omega_{P_1} : \vec{P} \in \mathbb{P}(T) \right\}$ is lacunary about $c_{\omega_T}$ for every $\mathbb{Q}-$tree $T$. 

\begin{definition}
For every dyadic interval $I \in \mathcal{D}$ and measurable set $E \subset \mathbb{R}$, let 

\begin{eqnarray*}
SIZE(E, I) := \frac{1}{|I|} \int_E \tilde{1}_I (x) dx. 
\end{eqnarray*}
\end{definition}
\begin{definition}
For each $n_1\geq 0$, let $\mathbb{I}_{n_1, 1}$ denote the collection of dyadic intervals maximal with respect to the property 
\begin{eqnarray*}
SIZE(E_1, I) \geq 2^{-n_1}. 
\end{eqnarray*}
Similarly, for each $n_4 \geq 0$, let $\mathbb{I}_{n_4, 4}$ denote those intervals maximal with respect to the property

\begin{eqnarray*}
SIZE(E_4, I) \geq 2^{-n_4}. 
\end{eqnarray*}

\end{definition}
As before, $\left\{\mathbb{I}_{n_1, 1} \right\}$ and $\left\{ \mathbb{I}_{n_4, 4} \right\}$ generate two decompositions of $\mathbb{P}$ using the recursive definitions  

\begin{eqnarray*}
\mathbb{P}_{n_1, 1} &:=& \left\{ \vec{P} \in \mathbb{P} : I_{\vec{P}} \subset I \in \mathbb{I}_{n_1, 1} \right\} \cap \left[ \bigcup_{ \tilde{n}_1 < n_1} \left\{ \vec{P} \in \mathbb{P}: I_{\vec{P}} \subset I \in \mathbb{I}_{\tilde{n}_1,1} \right\} \right]^c\\ \mathbb{P}_{n_4, 4} &:=& \left\{ \vec{P} \in \mathbb{P} : I_{\vec{P}} \subset I \in \mathbb{I}_{n_4, 4} \right\} \cap \left[ \bigcup_{ \tilde{n}_4 < n_4} \left\{ \vec{P} \in \mathbb{P}: I_{\vec{P}} \subset I \in \mathbb{I}_{\tilde{n}_4,4} \right\} \right]^c
\end{eqnarray*}
Then define $\mathbb{P}^d_{n_1, n_4} := \mathbb{P}_{n_1, 1} \cap \mathbb{P}_{n_4, 4} \cap \mathbb{P}^d$. 

\begin{definition}
For a given dyadic interval $I \in \mathcal{D}$, let
\begin{eqnarray*}
 SIZE^{d, \tilde{d}} _{n_1, n_4} (f_1, f_4, I) &:=& \sup_{T: I_T = I} \frac{1}{|I|^{1/2}} \left[ \sum_{\vec{Q} \in \mathbb{Q}^{\tilde{d}}\cap T}  \left| \left \langle   \sum_{\vec{P} \in  \mathbb{P}^d_{n_1, n_4} \cap \mathbb{P}(T)} \frac{1}{|I_{\vec{P}}|^{1/2}} \langle f_4, \Phi_{-P_1, 4} \rangle \langle f_1, \Phi_{P_1,1} \rangle\Phi^{n-l,T}_{\vec{P}} , \Phi^{lac}_{|\vec{Q}|} \right \rangle\right|^2 \right. \\ && ~~~~~~~~~~~~~~~+ \left.  \sum_{\vec{Q} \in \mathbb{Q}^{\tilde{d}}\cap T}  \left| \left \langle   \sum_{\vec{P} \in \tilde{\tilde{T}}(Q)\cap \mathbb{P}^d_{n_1, n_4}} \frac{1}{|I_{\vec{P}}|^{1/2}}\langle f_1, \Phi_{-P_1} \rangle \langle f_4, \Phi_{P_1} \rangle \Phi^{n-l}_{|\vec{P}| } , \Phi^{lac}_{|\vec{Q}|} \right \rangle\right|^2 \right]^{1/2} .
\end{eqnarray*}
\end{definition}
 
 \begin{definition}
For each $n_0 \in \mathbb{Z}$, let $\left\{ \left[\mathbb{I}^{d, \tilde{d}}_{n_1, n_4} \right]_{n_0} \right\}$ be the collection of dyadic intervals that are maximal with respect to the property that

\begin{eqnarray*}
SIZE_{n_1,  n_4}^{d, \tilde{d}}(f_1, f_4,I ) \geq 2^{-n_0}.
\end{eqnarray*}
\end{definition}
As with the $\vec{P}-$ tiles and $\left\{\mathbb{I}_{n_1,1}\right\}$, the collection $\left\{ \left[\mathbb{I}^{d, \tilde{d}}_{n_1, n_4} \right]_{n_0} \right\}$ generates a decomposition of the $\mathbb{Q}^{\tilde{d}}-$tiles by the familiar recursive formula 

\begin{eqnarray*}
\left[ \mathbb{Q}^{d, \tilde{d}}_{n_1, n_4}\right]_{n_0} &:=& \left\{ \vec{Q} \in \mathbb{Q}^{\tilde{d}} : I_{\vec{Q}} \subset I \in \left[ \mathbb{I}^{d, \tilde{d}}_{n_1, n_4} \right]_{n_0} \right\} \cap \left[ \bigcup_{ \tilde{n}_0 < n_0} \left\{ \vec{Q} \in \mathbb{Q}^{\tilde{d}}: I_{\vec{P}} \subset I \in \left[ \mathbb{I}^{d, \tilde{d}}_{n_1, n_4} \right]_{\tilde{n}_0 }\right\} \right]^c.
\end{eqnarray*}
Putting it all together, we have

\begin{eqnarray*}
\mathbb{Q} = \bigcup_{\tilde{d} \geq 0} \mathbb{Q}^{\tilde{d}} = \bigcup_{\tilde{d} \geq 0} \bigcup_{n_1, n_4} \bigcup_{n_0} \left[ \mathbb{Q}^{d, \tilde{d}}_{n_1, n_4} \right]_{n_0} = \bigcup_{\tilde{d} \geq 0} \bigcup_{n_1, n_4} \bigcup_{n_0} \bigcup_{ I  \in\left[\mathbb{I}^{d, \tilde{d}}_{n_1, n_4} \right]_{n_0} }\left[ \mathbb{Q}^{d, \tilde{d}}_{n_1, n_4} \right]_{n_0} [I],
\end{eqnarray*}
where 

\begin{eqnarray*}
\left[ \mathbb{Q}^{d, \tilde{d}}_{n_1, n_4} \right]_{n_0} [I] :=\left\{ \vec{Q} \in \left[ \mathbb{Q}^{d, \tilde{d}}_{n_1, n_4} \right]_{n_0}: I_{\vec{Q}} \subset I \right\}.
\end{eqnarray*}
Lastly, we need to introduce 
\begin{definition}
For any $\vec{Q} \in \mathbb{Q}$, let $\Phi_{Q_1, 2 , a}$, $\Phi_{Q_1,2, b}$, and $\Phi_{Q_1, 2, c}$ be defined by the identities

\begin{eqnarray*}
\hat{\Phi}_{Q_1, 2, a}(\xi) &=& \hat{\Phi}_{Q_1, 2}(\xi) \\ 
\hat{\Phi}_{Q_1, 2, b}(\xi) &=& (\xi- c_{\omega_{Q_2}}) |I_{\vec{Q}}| \hat{\Phi}_{Q_1, 2}(\xi)\\
\hat{\Phi}_{Q_1, 2, c}(\xi) &=& (\xi - c_{\omega_{Q_2}})^2 |I_{\vec{Q}}|^2 \hat{\Phi}_{Q_1, 2} (\xi). 
\end{eqnarray*}

\end{definition}
It is simple to check that $\Phi_{Q_1, 2, a}, \Phi_{Q_2, 2, b},$ $\Phi_{Q_1, 2, c}$ are all $L^2$-normalized wave packets on the tile $Q_1$. 

\begin{definition}
For a given $I \in   \mathbb{I}_{n_1, n_4} $ and $j \in \{2, 3\}$, let

\begin{eqnarray*}
&& SIZE^{\tilde{d},2}_{n_1, n_4} (f_2 , I) \\&=& \sup_{T \subset \mathbb{Q}^{\tilde{d}}_{n_1, n_4}  : I_T \subset I}\frac{1}{|I_T|^{1/2}}\left[  \left( \sum_{\vec{Q} \in T } |\langle f_2, \Phi_{Q_{1}, 2,a} \rangle |^2 \right)^{1/2}+ \left( \sum_{\vec{Q} \in T} |\langle f_2, \Phi_{Q_{1},2,b} \rangle |^2 \right)^{1/2}+ \left( \sum_{\vec{Q} \in T } |\langle f_2, \Phi_{Q_{1}, 2,c} \rangle |^2 \right)^{1/2} \right] \\ 
&&SIZE^{\tilde{d}, 3}_{n_1, n_4} (f_3 , I)  = \sup_{T \subset \mathbb{Q}^{\tilde{d}}_{n_1, n_4}  : I_T \subset I}\frac{1}{|I_T|^{1/2}} \left( \sum_{\vec{Q} \in T } |\langle f_3, \Phi_{Q_{2}, 3} \rangle |^2 \right)^{1/2}
\end{eqnarray*}
where the supremum in $SIZE^{\tilde{d},2}_{n_1, n_4}$ is over all $2-$trees $T \subset \mathbb{Q}^{\tilde{d}}_{n_1, n_4}$ for which $I_T \subset \bigcup_{I \in \mathbb{I}_{n_1, n_4}} I $ and the supremum in $SIZE^{\tilde{d}, 3}_{n_1, n_4}$ is over all $1-$trees $T \subset \mathbb{Q}^{\tilde{d}}_{n_1, n_4}$ for which $I_T \subset \bigcup_{I \in \mathbb{I}_{n_1, n_4}} I$. 

\end{definition}
This decomposes $\mathbb{Q}^{\tilde{d}}_{n_1, n_4} [I]$ into a union of trees corresponding to each of the $2$ sizes for indices $2$ and $3$, i.e. 

\begin{eqnarray*}
 \mathbb{Q}^{\tilde{d}}_{n_1, n_4} [I] &=& \bigcup_{n_2}  \mathbb{Q}^{ \tilde{d}}_{n_1, n_4} [I]_{n_2, 2} = \bigcup_{T_2 \in \mathcal{T} \left\{\mathbb{Q}^{ \tilde{d}}_{n_1, n_4}  [I] _{n_2,2}\right\}} T_2 \\ 
\mathbb{Q}^{ \tilde{d}}_{n_1, n_4}  [I] &=& \bigcup_{n_3} \mathbb{Q}^{ \tilde{d}}_{n_1, n_4}  [I]_{n_3, 3} = \bigcup_{T_3 \in \mathcal{T} \left\{  \mathbb{Q}^{ \tilde{d}}_{n_1, n_4}  [I]_{n_3,3} \right\}} T_3.
\end{eqnarray*}
Finish by setting 

\begin{eqnarray*}
\left[  \mathbb{Q}^{\tilde{d}}_{n_1, n_4} \right]_{n_2}^{n_3} :=\bigcup_{I \in  \mathbb{I}_{n_1, n_4}}\left[ \left( \bigcup_{T_2 \in \mathcal{T} \left\{ \mathbb{Q}^{\tilde{d}}_{n_1, n_4}  [I] _{n_2,2}\right\}} T_2 \right)\bigcap \left(  \bigcup_{T_3 \in \mathcal{T} \left\{  \mathbb{Q}^{ \tilde{d}}_{n_1, n_4}  [I]_{n_3,3} . \right\}} T_3 \right) \right].
\end{eqnarray*}

\subsection{Tree Localization}
We have worked to decompose $\mathbb{Q}^{\tilde{d}} \times \mathbb{P}$ into a union of trees with useful properties. Let us now fix $ \tilde{d}, n_1, n_4, n_0, I, n_2$ and let $T_2 \in \mathcal{T} \left\{ \mathbb{Q}^{\tilde{d}}_{n_1, n_4}  [I] _{n_2,2}\right\}$, say, and try to estimate $\Lambda_{T_2}(f_1, f_2, f_3, f_4)$ defined by 

\begin{eqnarray*}
 \sum_{\vec{Q} \in T_2} ~ \sum_{\vec{P} \in \mathbb{P}(T_2)}\frac{1}{|I_{\vec{Q}}|^{1/2}|I_{\vec{P}}|^{1/2}}   \langle f_4, \Phi_{-P_1,4} \rangle \langle f_1, \Phi_{P_1,1} \rangle \left\langle \Phi_{|\vec{P}|}^{n-l} , \Phi^{lac}_{|\vec{Q}|} \right\rangle \langle f_2 *\eta_{P_2, 2}, \Phi_{Q_1,2} \rangle \langle f_3, \Phi_{Q_2,3} \rangle.
\end{eqnarray*}
Recall that $T_2$ is a $\mathbb{Q}-$tree comes equipped with a center frequency $c_{\omega_T}$. Inspired by the argument of J. Jung in \cite{2013arXiv1311.1574J}, we define $R^T_{\vec{P}}(\xi)$ by
\begin{eqnarray*}
\hat{\eta}_{P_2,2} (\xi) := \hat{\eta}_{P_2, 2} (c_{\omega_T}) + \hat{\eta}^{(1)}(c_{\omega_T})(\xi-c_{\omega_T}) + \frac{1}{2} \hat{\eta}^{(2)}(c_{\omega_T})(\xi- c_{\omega_T})^2 +  R^{T}_{\vec{P}}(\xi) .
\end{eqnarray*}
Now rewrite 

\begin{eqnarray*}
&& \langle f_2 * \eta_{P_2, 2} , \Phi_{Q_1} \rangle \\&=& \langle \hat{f}_2  \hat{\eta}_{P_2, 2}, \hat{\Phi}_{Q_1, 2} \rangle \\ &=&   \hat{\eta}_{P_2, 2} (c_{\omega_T}) \langle f_2, \Phi_{Q_1, 2} \rangle +\hat{\eta}^{(1)}_{P_2, 2}(c_{\omega_T})  \langle \hat{f}_2, (\cdot - c_{\omega_T}) \hat{\Phi}_{Q_1, 2}  \rangle+\frac{\hat{\eta}^{(2)}_{P_2, 2}(c_{\omega_T})}{2}\langle \hat{f}_2 , (\cdot -c_{\omega_T})^2 \hat{\Phi}_{Q_1, 2} \rangle + \langle \hat{f}_2 R^T_{\vec{P}} , \hat{\eta}_{Q_1, 2} \rangle \\ &=& I_a + I_b + I_c + II. 
\end{eqnarray*}
Therefore, $\Lambda(f_1, f_2, f_3, f_4) =\left[ \Lambda_{I_a}+\Lambda_{I_b} + \Lambda_{I_c} \right] (f_1, f_2, f_3, f_4) + \Lambda_{II} (f_1, f_2, f_3, f_4)$, where

\begin{eqnarray*}
&& \Lambda_{I_a}(f_1, f_2, f_3, f_4)\\
&=&  \sum_{\vec{Q} \in T_2} ~ \left\langle\sum_{\vec{P} \in \mathbb{P}(T_2)} \frac{ \langle f_4, \Phi_{-P_1,4} \rangle \langle f_1, \Phi_{P_1,1} \rangle }{|I_{\vec{P}}|^{1/2}} \hat{\eta}_{P_2,2}(c_{\omega_T}) \Phi_{\vec{P}, 5}^{n-l,T} , \Phi^{lac}_{|\vec{Q}|} \right\rangle \frac{ \langle f_2 , \Phi_{Q_1,2} \rangle \langle f_3, \Phi_{Q_2,3} \rangle}{|I_{\vec{Q}}|^{1/2}} \\ 
&& \Lambda_{I_b}(f_1, f_2, f_3, f_4) \\ &=&  \sum_{\vec{Q} \in T_2} ~ \left\langle\sum_{\vec{P} \in \mathbb{P}(T_2)} \frac{ \langle f_4, \Phi_{-P_1,4} \rangle \langle f_1, \Phi_{P_1,1} \rangle }{|I_{\vec{P}}|^{-1/2}} \hat{\eta}^{(1)}_{P_2,2}(c_{\omega_T}) \Phi_{\vec{P}, 5}^{n-l,T}  , \Phi^{lac}_{|\vec{Q}|} \right\rangle \frac{ \langle \hat{f}_2 ,  (\cdot - c_{\omega_{T_2}})\hat{\Phi}_{Q_1,2} \rangle \langle f_3, \Phi_{Q_2,3} \rangle}{|I_{\vec{Q}}|^{1/2}}\\ 
&& \Lambda_{I_c}(f_1, f_2, f_3, f_4) \\&=& \frac{1}{2} \sum_{\vec{Q} \in T_2} ~  \left\langle\sum_{\vec{P} \in \mathbb{P}(T_2)} \frac{ \langle f_4, \Phi_{-P_1,4} \rangle \langle f_1, \Phi_{P_1,1} \rangle }{|I_{\vec{P}}|^{-3/2}}\hat{\eta}_{P_2,2}^{(2)}(c_{\omega_T})\Phi_{\vec{P},5}^{n-l,T} , \Phi^{lac}_{|\vec{Q}|} \right\rangle \frac{ \langle \hat{f}_2 ,  (\cdot - c_{\omega_{T_2}})^2\hat{\Phi}_{Q_1,2} \rangle \langle f_3, \Phi_{Q_2,3} \rangle }{|I_{\vec{Q}}|^{1/2}}
\end{eqnarray*}
and 
\begin{eqnarray*}
&&\Lambda_{II}(f_1, f_2, f_3, f_4) \\&=&  \sum_{\vec{Q} \in T_2} ~ \sum_{\vec{P} \in  \mathbb{P}(T_2)}\frac{1}{|I_{\vec{Q}}|^{1/2}|I_{\vec{P}}|^{1/2}}   \langle f_4, \Phi_{-P_1,4} \rangle \langle f_1, \Phi_{P_1,1} \rangle \left\langle \Phi_{|\vec{P}|}^{n-l} , \Phi^{lac}_{|\vec{Q}|} \right\rangle \langle \hat{f}_2 \cdot R^T_{\vec{P}}  \cdot \hat{\eta}_{\omega_{Q_1}}, \hat{\Phi}_{Q_1,2} \rangle \langle f_3, \Phi_{Q_2,3} \rangle.
\end{eqnarray*}
Furthermore, we may rewrite $(\xi- c_{\omega_{T_2}}) =( \xi - c_{\omega_{Q_2}} )+ (c_{\omega_{Q_2}} - c_{\omega_{T_2}})$, in which case 

\begin{eqnarray*}
&& \Lambda_{I_b}(f_1, f_2, f_3, f_4) \\&=& \sum_{\vec{Q} \in T_2} ~ \frac{1}{|I_{\vec{Q}}|^{1/2}}  \left\langle\sum_{\vec{P} \in \mathbb{P}(T_2)} \frac{ \langle f_4, \Phi_{-P_1,4} \rangle \langle f_1, \Phi_{P_1,1} \rangle }{|I_{\vec{P}}|^{1/2}}\Phi_{|\vec{P}|,2}^{n-l,T} |I_{\vec{P}}|, \Phi^{lac}_{|\vec{Q}|} \right\rangle \langle \hat{f}_2 ,  (\cdot - c_{\omega_{Q_2}})\hat{\Phi}_{Q_1,2} \rangle \langle f_3, \Phi_{Q_2,3} \rangle \\ &+&  \sum_{\vec{Q} \in T_2} ~ \frac{c_{\omega_{Q_2}} - c_{\omega_{T_2}}}{|I_{\vec{Q}}|^{1/2}}  \left\langle\sum_{\vec{P} \in \mathbb{P}(T_2)} \frac{ \langle f_4, \Phi_{-P_1,4} \rangle \langle f_1, \Phi_{P_1,1} \rangle }{|I_{\vec{P}}|^{1/2}} \Phi_{\vec{P},5}^{n-l,T}|I_{\vec{P}}| , \Phi^{lac}_{|\vec{Q}|} \right\rangle \langle f_2 , \Phi_{Q_1,2} \rangle \langle f_3, \Phi_{Q_2,3} \rangle \\ &:=& \Lambda_{I_{b,1}}(f_1, f_2, f_3, f_4) + \Lambda_{I_{b,2}}(f_1, f_2, f_3, f_4).
\end{eqnarray*}
By construction, 

\begin{eqnarray*}
 \Lambda_{I_{b,1}}(f_1, f_2, f_3, f_4)  = \sum_{\vec{Q} \in T_2} ~ \frac{1}{|I_{\vec{Q}}|^{1/2}}  \left\langle\sum_{\vec{P} \in \mathbb{P}(T_2)} \frac{ \langle f_4, \Phi_{-P_1,4} \rangle \langle f_1, \Phi_{P_1,1} \rangle }{|I_{\vec{P}}|^{1/2}}\Phi_{\vec{P},5}^{n-l,T}  |I_{\vec{P}}|,\frac{ \Phi^{lac}_{|\vec{Q}|}}{|I_{\vec{Q}}|} \right\rangle \langle f_2 , \Phi_{Q_2,2,b} \rangle \langle f_3, \Phi_{Q_2,3} \rangle.
\end{eqnarray*}
Furthermore, 

\begin{eqnarray*}
 \left| \Lambda_{I_{b,2}}(f_1, f_2, f_3, f_4) \right| \leq  \sum_{\vec{Q} \in T_2} ~ \frac{1}{|I_{\vec{Q}}|^{1/2}} \left|  \left\langle\sum_{\vec{P} \in \mathbb{P}(T_2)} \frac{ \langle f_4, \Phi_{-P_1,4} \rangle \langle f_1, \Phi_{P_1,1} \rangle }{|I_{\vec{P}}|^{1/2}} \Phi_{\vec{P},5}^{n-l,T}|I_{\vec{P}}| , \frac{ \Phi^{lac}_{|\vec{Q}|} }{|I_{\vec{Q}}|}\right\rangle \right| \left| \langle f_2 , \Phi_{Q_1,2} \rangle \langle f_3, \Phi_{Q_2,3} \rangle \right|.
\end{eqnarray*}
To handle $\Lambda_{I_c}(f_1, f_2, f_3, f_4)$, we rewrite 

\begin{eqnarray*}
(\xi- c_{\omega_{T_2}})^2 = ((\xi -c_{\omega_{Q_2}}) + (c_{\omega_{Q_2}} - c_{\omega_{T_2}}))^2= (\xi - c_{\omega_{Q_2}})^2 + 2 (\xi- c_{\omega_{Q_2}})(c_{\omega_{Q_2} }- c_{\omega_{T_2}}) + (c_{\omega_{Q_2} }- c_{\omega_{T_2}})^2.
\end{eqnarray*}
It is a straightforward matter to decompose $\Lambda_{I_c}(f_1, f_2, f_3, f_4)$ into three terms corresponding to the above display and then to bound each using previous observations. Each of the three terms can be majorized by a generic expression of the form

\begin{eqnarray*}
 \sum_{\vec{Q} \in T_2} ~ \frac{1}{|I_{\vec{Q}}|^{1/2}} \left|  \left\langle\sum_{\vec{P} \in \mathbb{P}(T_2)} \frac{ \langle f_4, \Phi_{-P_1,4} \rangle \langle f_1, \Phi_{P_1,1} \rangle }{|I_{\vec{P}}|^{1/2}} \tilde{\Phi}_{\vec{P}}^{n-l,T}|I_{\vec{P}}| ^2, \frac{ \Phi^{lac}_{|\vec{Q}|} }{|I_{\vec{Q}}|^2}\right\rangle \right| \left| \langle f_2 , \Phi_{Q_1,2,j} \rangle \langle f_3, \Phi_{Q_2,3} \rangle \right|,
\end{eqnarray*}
 where $j \in \{a,b,c\}$. In considering the remainder $R_{\vec{P}}^T(\xi)$, expand using fourier series
\begin{eqnarray*}
R^T_{\vec{P}}(\xi)  \hat{\eta}_{\omega_{Q_1}} (\xi)= \sum_{\lambda \in \mathbb{Z}} c^{\lambda}_{\vec{P}, \vec{Q}} \hat{\eta}^\lambda_{\omega_{Q_1}}(\xi)
\end{eqnarray*}
where the sequence $\left| c^\lambda_{\vec{P}, \vec{Q}}\right| \lesssim \frac{|I_{\vec{P}}|^3}{|I_{\vec{Q}}|^3} \frac{1}{1+|\lambda|^{\tilde{N}}}$. Indeed, this is a straightforward consequence of the definition of $R^T_{\vec{P}}(\xi)$, and in the support of $R^T_{\vec{P}}(\xi)  \hat{\eta}_{\omega_{Q_1}} (\xi)$, $|\xi- c_{\omega_T}| \lesssim \frac{1}{|I_{\vec{Q}}|}$. Plugging into our formula yields 

\begin{eqnarray*}
&& \Lambda_{II}(f_1, f_2, f_3, f_4)  \\ &=&  \sum_{\vec{Q} \in T_2} ~ \sum_{\vec{P} \in \mathbb{P}^d_{n_1, n_4} \cap \mathbb{P}(T_2)}   \sum_{\lambda \in \mathbb{Z}} \frac{c_{\vec{P}, \vec{Q}}^\lambda}{|I_{\vec{Q}}|^{1/2}|I_{\vec{P}}|^{1/2}}   \langle f_4, \Phi_{-P_1,4} \rangle \langle f_1, \Phi_{P_1,1} \rangle \left\langle \Phi_{|\vec{P}|}^{n-l} , \Phi^{lac}_{|\vec{Q}|} \right\rangle \langle f_2,\Phi^\lambda_{Q_1,2} \rangle \langle f_3, \Phi_{Q_2,3} \rangle
\end{eqnarray*}
so that
\begin{eqnarray*}
&&\left|  \Lambda_{II}(f_1, f_2, f_3, f_4) \right| \\ &\leq & \sum_{\lambda \in \mathbb{Z}} \frac{1}{1+|\lambda|^{\tilde{N}}} \sum_{\vec{Q} \in T_2} ~ \sum_{\vec{P} \in \mathbb{P}^d_{n_1, n_4} \cap \mathbb{P}(T_2)}   \frac{|I_{\vec{P}}|^{5/2}}{|I_{\vec{Q}}|^{7/2}}  \left| \langle f_4, \Phi_{-P_1,4} \rangle \langle f_1, \Phi_{P_1,1} \rangle \left\langle \Phi_{|\vec{P}|}^{n-l} , \Phi^{lac}_{|\vec{Q}|} \right\rangle \langle f_2,\Phi^\lambda_{Q_1,2} \rangle \langle f_3, \Phi_{Q_2,3} \rangle \right|. 
\end{eqnarray*}
At this stage, we do not need to analyze $\Lambda_{T_2, II}(f_1, f_2, f_3, f_4)$ any further.

\subsection{Tree Estimates} 
\subsubsection{$\Lambda_{I_a}$ Tree Estimate}
\begin{lemma}\label{TLa}
The following $I_a$ type size estimate holds: for any $0 < \theta <1$, 

\begin{eqnarray*}
&&\left( \sum_{\vec{Q} \in T_2}\left| \left\langle \sum_{\vec{P} \in \mathbb{P} : |I_{\vec{P}}| << |I_{\vec{Q}}|}  \frac{ \langle f_1, \Phi_{-P_1,1} \rangle \langle f_4, \tilde{\Phi}_{P_1,4} \rangle }{|I_{\vec{P}}|^{1/2}} \Phi_{|\vec{P}|,a}^{n-l, T} , \Phi^{lac}_{|\vec{Q}|}\right \rangle \right|^2 \right)^{1/2} \\&\lesssim_\theta& \left[  \sup_{\vec{Q} \in \mathcal{T}_2} \frac{1}{|I_{\vec{Q}}|} \int_{E_1} \tilde{1}_{I_{\vec{Q}}} dx \right] ^{1-\theta}\left[ \sup_{\vec{Q} \in \mathcal{T}_2} \frac{1}{|I_{\vec{Q}}|} \int_{E_4 \cap \Omega^c} \tilde{1}_{I_{\vec{Q}}} dx \right]^\theta .
\end{eqnarray*}

\end{lemma}
\begin{proof}
We include the proof taken from \cite{MR2127985} for the reader's convenience. For each $\vec{Q} \in T$, set 

\begin{eqnarray*}
a^T_{\vec{Q}} = \left\langle \sum_{\vec{P} \in \mathbb{P} : |I_{\vec{P}}| << |I_{\vec{Q}}|}  \frac{ \langle f_1, \Phi_{-P_1,1} \rangle \langle f_4, \tilde{\Phi}_{P_1,4} \rangle }{|I_{\vec{P}}|^{1/2}} \Phi_{|\vec{P}|,a}^{n-l, T} , \Phi^{lac}_{|\vec{Q}|}\right \rangle.
\end{eqnarray*}
By John-Nirenberg, it suffices to show 

\begin{eqnarray*}
\left| \left| \left( \sum_{\vec{Q} \in T} | a^T_{\vec{Q}} |^2 \frac{1_{I_{\vec{Q}}}}{|I_{\vec{Q}}|} \right)^{1/2} \right| \right|_{L^{1, \infty}(I_T)} \lesssim |I_T| \sup_{\vec{Q} \in T} \left[ \frac{  \int_{E_1} \tilde{1}_{I_{\vec{Q}}} }{|I_{\vec{Q}}|} \right]^{1-\theta}  \left[ \frac{ \int_{E_4\cap \Omega^c} \tilde{1}_{I_{\vec{Q}}} dx}{|I_{\vec{Q}}|} \right]^\theta.
\end{eqnarray*}
We may assume that $T$ contains its top $P_T$, in which case we may reduce to 

\begin{eqnarray*}
\left| \left| \left( \sum_{\vec{Q} \in T} | a^T_{\vec{Q}} |^2 \frac{1_{I_{\vec{Q}}}}{|I_{\vec{Q}}|} \right)^{1/2} \right| \right|_{L^{1, \infty}(I_T)}  \lesssim  \left[  \int_{E_1} \tilde{1}_{I_T} dx \right]^{1-\theta}  \left[  \int_{E_4 \cap \Omega^c} \tilde{1}_{I_T} dx \right]^\theta. 
\end{eqnarray*}
Fix $T$. First consider the relatively easy case when $f_1$ vanishes on $5I_T$. In this case, we shall prove the stronger estimate

\begin{eqnarray*}
|a_{\vec{Q}}^T|& \lesssim&  |I_{\vec{Q}}|^{-1/2} \left[  \int_{E_1} \tilde{1}_{I_{\vec{Q}}} dx \right]^{1-\theta}  \left[  \int_{E_4\cap \Omega^c} \tilde{1}_{I_{\vec{Q}}} dx \right]^\theta \\ &\lesssim&   |I_{\vec{Q}}|^{-1/2} \left( \frac{|I_{\vec{Q}}|}{|I_T|} \right)^{M(1-\theta)}  \left[  \int_{E_1} \tilde{1}_{I_T} dx \right]^{1-\theta}  \left[  \int_{E_4\cap \Omega^c} \tilde{1}_{I_T} dx \right]^\theta .
\end{eqnarray*}
The claim then follows by square-summing in $\vec{Q}$. To prove the claim, fix $\vec{Q} \in T$ and estimate

\begin{eqnarray*}
|a_{\vec{Q}}^T | \lesssim |I_{\vec{Q}}|^{-1/2} \sum_{\vec{P} \in \mathbb{P}(T): |I_{\vec{P}}| \lesssim |I_{\vec{Q}}|} |\langle f_1, \Phi_{-P_1, 1} \rangle | | \langle f_4, \Phi_{P_1, 4} \rangle| \int_\mathbb{R} \frac{ \tilde{1}_{I_{\vec{P}}}}{|I_{\vec{P}}|} \tilde{1}_{I_{\vec{Q}}} dx.
\end{eqnarray*}
Interchanging the sum and the integral and applying Cauchy-Schwarz gives

\begin{eqnarray*}
|a_{\vec{Q}}^T| \lesssim |I_{\vec{Q}}|^{-1/2} \int_\mathbb{R} |S_1 f_1| |S_4 f_4| \tilde{1}_{I_{\vec{Q}}} dx,
\end{eqnarray*}
where each $S_1, S_4$ is a Calderon-Zygmund operator with localized estimates on $\tilde{1}_{I_{\vec{Q}}}$. Therefore, the claim is true and we are done when $f_1$ vanishes on $5I_T$. Similarly, we are done if $f_4$ vanishes on $5I_T$. Hence, it suffices to consider the case when both $E_1, E_4$ are supported inside $5I_T$. In this situation, it suffices to prove

\begin{eqnarray*}
\left| \left| \left( \sum_{\vec{Q} \in T} \left| \left \langle \sum_{\vec{P} \in \mathbb{P}(T)} \frac{ \langle f_1, \Phi_{-P_1, 1} \rangle \langle f_4, \Phi_{P_1, 4} \rangle}{|I_{\vec{P}}|} \Phi^{n-l}_{|\vec{P}|}, \Phi_{\vec{Q}}^{lac} \right\rangle \right|^2 \frac{\tilde{1}_{I_{\vec{Q}}}}{|I_{\vec{Q}}|} \right)^{1/2} \right| \right|_{L^{1, \infty}(\mathbb{R})} \lesssim_\theta |E_1|^{1-\theta} |E_4\cap \Omega^c|^\theta.
\end{eqnarray*}
However, this follows from the $L^1 \rightarrow L^{1,\infty}$ estimates for Calderon-Zygmund operators together with the estimate $\sum_{\vec{P} \in \mathbb{P}(T)}  |\langle f_1, \Phi_{-P_1, 1} \rangle| |\langle f_4, \Phi_{P_1, 4} \rangle| \lesssim_\theta |E_1|^{1-\theta} |E_4 \cap \Omega^c|^\theta$.

\end{proof}

To handle $\Lambda_{I_a}(f_1, f_2, f_3, f_4)$, note that the $\mathbb{Q}-$tree $T_2$ must be overlapping in either the first or second index. Without loss of generality, assume that $T_2$ is overlapping in the 1st index. Then the Biest size estimate gives
 \begin{eqnarray*}
 && | \Lambda_{T_2, I}(f_1, f_2, f_3, f_4)| \\&\leq& \left( \sum_{\vec{Q} \in T_2}\left| \left\langle \sum_{\vec{P} \in  \mathbb{P}(T_2) } \frac{ \langle f_1, \Phi_{-P_1,1} \rangle \langle f_4, \tilde{\Phi}_{P_1,4} \rangle }{|I_{\vec{P}}|^{1/2}} \Phi_{|\vec{P}|}^{n-l, T} , \Phi^{lac}_{|\vec{Q}|} \right \rangle \right|^2 \right)^{1/2} \left[ \sup_{\vec{Q} \in T_2} \frac{ |\langle f_2, \Phi_{Q_1,2} \rangle|}{|I_{\vec{Q}}|^{1/2}} \right]\left( \sum_{\vec{Q} \in T_2} |\langle f_3, \Phi_{Q_2,3} \rangle|^2 \right)^{1/2} \\ &\lesssim_\theta&  2^{-n_1(1-\theta)} 2^{-n_4\theta}2^{-n_2} 2^{-n_3} |I_{T_2}|, 
\end{eqnarray*}
for any $0 < \theta <1$.  
\subsubsection{$\Lambda_{I_b}, \Lambda_{I_c}$ Tree Estimates}
To contend with $\Lambda_{I_b}(f_1, f_2, f_3, f_4)$ and $\Lambda_{I_c}(f_1, f_2, f_3, f_4)$, we verify

\begin{lemma}\label{TLbc}
The following $I_b, I_c$ type size estimate holds:  for any $\epsilon>0$ and $0 < \theta <1$,

\begin{eqnarray*}
&&\left( \sum_{\vec{Q} \in T_2}\left| \left\langle \sum_{\vec{P} \in \mathbb{P} : |I_{\vec{P}}| << |I_{\vec{Q}}| } \frac{1}{|I_{\vec{P}}|^{\epsilon}} \frac{ \langle f_1, \Phi_{-P_1,1} \rangle \langle f_4, \tilde{\Phi}_{P_1,4} \rangle }{|I_{\vec{P}}|^{1/2}} \Phi_{|\vec{P}|}^{n-l, T} , \Phi^{lac}_{|\vec{Q}|} |I_{\vec{Q}}|^{\epsilon} \right \rangle \right|^2 \right)^{1/2} \\&\lesssim_{\epsilon, \theta}& \left[  \sup_{\vec{Q} \in \mathcal{T}_2} \frac{1}{|I_{\vec{Q}}|} \int_{E_1} \tilde{1}_{I_{\vec{Q}}} dx \right] ^{1-\theta}\left[ \sup_{\vec{Q} \in \mathcal{T}_2} \frac{1}{|I_{\vec{Q}}|} \int_{E_4 \cap \Omega^c} \tilde{1}_{I_{\vec{Q}}} dx \right]^\theta .
\end{eqnarray*}
\end{lemma}
\begin{proof}
The standard Biest size proof handles the cases when either $E_1$ or $E_4$ vanishes on $5I_T$. Hence, it suffices to assume $E_1, E_4 \subset 5I_T$ and show 

\begin{eqnarray*}
\left| \left| \sum_{\vec{Q} \in T}  \int_\mathbb{R} \left[ \sum_{\vec{P} \in \mathbb{P}(T) : |I_{\vec{P}}| << |I_{\vec{Q}}|}\frac{  | \langle f_1, \Phi_{-P_1, 1} \rangle \langle f_4, \Phi_{P_1, 4} \rangle|}{|I_{\vec{P}}|} \tilde{1}_{I_{\vec{P}}} \frac{ \tilde{1}_{I_{\vec{Q}}}}{|I_{\vec{Q}}|^{1/2}} \right] dx \frac{1_{I_{\vec{Q}}}}{|I_{\vec{Q}}|^{1/2}} \frac{|I_{\vec{P}}|^\epsilon}{|I_{\vec{Q}}|^\epsilon} \right| \right|_{L^1(\mathbb{R})} \lesssim_\theta |E_1|^{1-\theta} |E_4\cap \Omega^c|^{\theta}.
\end{eqnarray*}
By the triangle inequality, it is enough to show
\begin{eqnarray*}
\left| \left| \sum_{\vec{P} \in \mathbb{P}(T)} ~\sum_{\vec{Q} \in T : |I_{\vec{Q}}| = 2^{-k_0} |I_{\vec{P}}|} \int_\mathbb{R} \frac{  | \langle f_1, \Phi_{-P_1, 1} \rangle \langle f_4, \Phi_{P_1, 4} \rangle|}{|I_{\vec{P}}|} \tilde{1}_{I_{\vec{P}}} \frac{ \tilde{1}_{I_{\vec{Q}}}}{|I_{\vec{Q}}|^{1/2}} dx \frac{1_{I_{\vec{Q}}}}{|I_{\vec{Q}}|^{1/2}}  \right| \right|_{L^1(\mathbb{R})} \lesssim_\theta |E_1|^{1-\theta} |E_4\cap \Omega^c|^{\theta}
\end{eqnarray*}
with an implicit constant independent of $k_0 \geq 0$. However, this is immediate from the observation that the above display can be bounded by

\begin{eqnarray*}
\left| \left| \sum_{\vec{P} \in \mathbb{P}(T)} ~  | \langle f_1, \Phi_{-P_1, 1} \rangle \langle f_4, \Phi_{P_1, 4} \rangle| \frac{\tilde{1}_{2^{k_0}I_{\vec{P}}}}{|2^{k_0} I_{\vec{P}}|}  \right| \right|_{L^1(\mathbb{R})}=\sum_{\vec{P} \in \mathbb{P}(T)} ~  | \langle f_1, \Phi_{-P_1, 1} \rangle \langle f_4, \Phi_{P_1, 4} \rangle|  \lesssim_\theta |E_1|^{1-\theta} |E_4\cap \Omega^c|^{\theta}.
\end{eqnarray*}

\end{proof}

\subsection{Energy Estimates}

Before proceeding to proving generalized restricted type mixed estimates for $\Lambda$, we must record 

 \begin{lemma}\label{EL10}

\begin{eqnarray*}
\sum_{I \in  \mathbb{I}_{n_1, n_4}} |I| \lesssim \min \{ 2^{n_1} |E_1|, 2^{n_4} \}. 
\end{eqnarray*}

\end{lemma}
\begin{proof}
Immediate from the construction of $\mathbb{I}_{n_1, n_4}$. 
\end{proof}
Moreover, we have
\begin{lemma}\label{EL11}

The following energy estimate holds:
\begin{eqnarray*}
\sum_{I \in  \mathbb{I}_{n_1, n_4}} \sum_{T \in  \mathcal{T} \left\{ \mathbb{Q}^{ \tilde{d}}_{n_1, n_4}  [I]_{n_2, 2}\right\}} |I_{T_2}|& \lesssim_\epsilon& \min \left\{ 2^{2n_2} |E_2| ,2^{n_2} |E_2| \min \{ 2^{n_1} |E_1|, 2^{n_4} \} \right\}.
\end{eqnarray*}

\end{lemma}
\begin{proof}
The fact that 
\begin{eqnarray*}
\sum_{I \in  \mathbb{I}_{n_1, n_4}} \sum_{T \in  \mathcal{T} \left\{ \mathbb{Q}^{ \tilde{d}}_{n_1, n_4}  [I]_{n_2, 2}\right\}} |I_{T_2}| \lesssim 2^{2n_2} |E_2|
\end{eqnarray*}
is by now standard, and so its proof is omitted. Therefore, it suffices to establish 
\begin{eqnarray*}
\sum_{I \in  \mathbb{I}_{n_1, n_4}}  \sum_{T \in  \mathcal{T} \left\{ \mathbb{Q}^{ \tilde{d}}_{n_1, n_4}  [I]_{n_2, 2}\right\}}|I_{T_2}| \lesssim_\epsilon 2^{n_2} |E_2|\min \left\{ 2^{n_1} |E_1|, 2^{n_4}\right \}.
\end{eqnarray*}
To this end, use an argument similar to that found in the toy model section to see
\begin{eqnarray*}
 \sum_{T \in  \mathcal{T} \left\{ \mathbb{Q}^{ \tilde{d}}_{n_1, n_4}  [I]_{n_2, 2}\right\}} |I_{T_2}| &\lesssim& 2^{2n_2}  \sum_{T \in  \mathcal{T} \left\{ \mathbb{Q}^{ \tilde{d}}_{n_1, n_4}  [I]_{n_2, 2}\right\}}  \sum_{\vec{Q} \in T_2} |\langle f_2, \Phi_{Q_1, 2} \rangle|^2 \\ &\lesssim& 2^{n_2} |E_2| \cdot |I|. 
\end{eqnarray*}
Therefore, an application of Lemma \ref{EL10} yields the claim. 

\end{proof}

\begin{lemma}\label{SRL}
The following size restrictions hold: $2^{-n_1} \lesssim 2^{\tilde{d}} |E_1|, 2^{-n_2} \lesssim 2^{\tilde{d}} |E_2|, 2^{-n_3} \lesssim 2^{\tilde{d}} |E_3|, 2^{-n_4} \lesssim 2^{-\tilde{N}\tilde{d}}$.
\end{lemma}

\begin{proof}
These are by now routine estimates, so the details are omitted. 
\end{proof}

It is now straightforward to observe 
\begin{lemma}\label{ML12}
For fixed $ \tilde{d} \geq 0, n_1 \geq N_1(\tilde{d}), n_2 \geq N_2(\tilde{d}), n_3 \geq N_3(\tilde{d}), n_4 \geq N_4(\tilde{d}) $ and any $0 < \theta < 1$
\begin{eqnarray*}
&& \left|  \sum_{\vec{Q} \in\left[ \mathbb{Q}^{ \tilde{d}}_{n_1, n_4}  \right]_{n_2}^{n_3}} \sum_{\vec{P} \in \mathbb{P}^d_{n_1, n_4} }\Lambda_{\mathbb{Q}, \mathbb{P}}(f_1, f_2, f_3, f_4) \right| \\ &\lesssim_\theta& 2^{-n_1(1-\theta) } 2^{-n_2} 2^{-n_3} 2^{-n_4 \theta} \min \left\{ 2^{2n_2} |E_2|,  2^{n_2} |E_2| \cdot \min \{ 2^{n_1} |E_1|, 2^{n_4} \} , 2^{2n_3} |E_3| \right\} \\ &+& \sum_{\lambda \in \mathbb{Z}} \frac{1}{1+|\lambda|^{\tilde{N}}} \sum_{\vec{Q} \in \left[ \mathbb{Q}^{ \tilde{d}}_{n_1, n_4}\right]_{n_2}^{n_3}} ~ \sum_{\vec{P} \in \mathbb{P} :  \omega_{P_2} \supset \supset \omega_{Q_1}}  \frac{ |I_{\vec{P}}|^2}{|I_{\vec{Q}}|^4}   \left|  \langle f_4, \Phi_{-P_1,4} \rangle \langle f_1, \Phi_{P_1,1} \rangle \left\langle \tilde{1}_{I_{\vec{P}}}, \tilde{1}_{I_{\vec{Q}}}\right\rangle \langle f_2, \Phi^\lambda_{Q_1,2} \rangle \langle f_3, \Phi_{Q_2,3} \rangle \right|.
\end{eqnarray*}
\end{lemma}

\begin{proof}
Combine Lemmas \ref{TLa}, \ref{TLbc}, and \ref{EL11}. 
\end{proof}
An immediate consequence of the Lemma \ref{ML12} is
\begin{corollary}
The following estimate holds:
\begin{eqnarray*}
&&\left|  \Lambda(f_1, f_2, f_3, f_4)\right| \\ &\lesssim_\theta&  \sum_{ \tilde{d}, \vec{n} \geq \vec{N}(\tilde{d})} 2^{-n_1(1-\theta)}  2^{-n_2} 2^{-n_3}2^{-n_4 \theta } \min \left\{ 2^{2n_2} |E_2| , 2^{n_2} |E_2| \cdot \min \{ 2^{n_1} |E_1|, 2^{n_4} \} , 2^{2n_3} |E_3| \right\} \\ &+& \sum_{\lambda \in \mathbb{Z}} \frac{1}{1+|\lambda|^{\tilde{N}}} \sum_{\vec{Q} \in \mathbb{Q}} ~ \sum_{\vec{P} \in \mathbb{P} :  \omega_{P_2} \supset \supset \omega_{Q_1}}  \frac{|I_{\vec{P}}|^2}{|I_{\vec{Q}}|^4}  \left|  \langle f_4, \Phi_{-P_1,4} \rangle \langle f_1, \Phi_{P_1,1} \rangle \left\langle \tilde{1}_{I_{\vec{P}}}, \tilde{1}_{I_{\vec{Q}}} \right\rangle \langle f_2, \Phi^\lambda_{Q_1,2} \rangle \langle f_3, \Phi_{Q_2,3} \rangle \right| \\ &:=& \Lambda_I(f_1, f_2, f_3, f_4) + \Lambda_{II} (f_1, f_2, f_3, f_4).  
\end{eqnarray*}
\end{corollary}
\begin{proof}
Using the triangle inequality, sum the previous estimate in Lemma \ref{ML12} over the size restrictions $\tilde{d} \geq 0, n_1 \geq N_1(\tilde{d}), n_2 \geq N_2(\tilde{d}), n_3 \geq N_3(\tilde{d}), n_4 \geq N_4(\tilde{d})$ in Lemma \ref{SRL}. 
\end{proof}
Note that $\Lambda_{II} (f_1, f_2, f_3, f_4)$ may be rewritten as
\begin{eqnarray*}
\sum_{\lambda \in \mathbb{Z}}  \sum_{ k_0 >>1 }\frac{2^{-3k_0}}{1+|\lambda|^{\tilde{N}}}\sum_{\vec{Q} \in \mathbb{Q}} ~ \sum_{\vec{P} \in \mathbb{P} : |I_{\vec{P}}| =2^{-k_0} |I_{\vec{Q}}|, \omega_{P_2} \supset \supset \omega_{Q_1}} \frac{  \left|  \langle f_4, \Phi_{-P_1,4} \rangle \langle f_1, \Phi_{P_1,1} \rangle \left\langle \tilde{1}_{I_{\vec{P}}} , \tilde{1}_{I_{\vec{Q}}} \right\rangle \langle f_2, \Phi^\lambda_{Q_1,2} \rangle \langle f_3, \Phi_{Q_2,3} \rangle \right| }{|I_{\vec{Q}}| |I_{\vec{P}}|}.
\end{eqnarray*} 
Therefore, to handle $\Lambda_{II}(f_1, f_2, f_3, f_4)$, it suffices to obtain restricted weak-type estimates for $\Lambda^{k_0}_{Toy}$ with operational bounds $O(2^{2k_0})$. However, this was already accomplished with the toy model decomposition.

\subsection{Mixed Generalized Restricted Weak Type Estimates for $\Lambda_I(f_1, f_2, f_3, f_4)$}
 It so happens that the global energy bound

\begin{eqnarray*}
\sum_{I \in  \mathbb{I}_{n_1, n_4}} \sum_{T \in  \mathcal{T} \left\{  \mathbb{Q}^{ \tilde{d}}_{n_1, n_4}[I]_{n_2, 2}\right\} } |I_{T_2}| \lesssim 2^{2n_2} |E_2|
\end{eqnarray*} 
is not needed to produce estimates. It follows that for any $0 \leq \theta_1, \theta_2 \leq 1$ satisfying $\theta_1 + \theta_2 =1$ along with $0 < \theta <1$ and $0 \leq \gamma \leq 1$, 

\begin{eqnarray*}
&& \Lambda_I(f_1, f_2, f_3, f_4) \\&\lesssim_{\theta}& \sum_{ \tilde{d} \geq 0} \sum_{\vec{n} \geq \vec{N}(\tilde{d})}   2^{-n_1(1-\theta)} 2^{-n_2} 2^{-n_3} 2^{-n_4 \theta}\min \left\{ 2^{2n_2} |E_2|,  2^{n_2} |E_2| \cdot \min \{ 2^{n_1} |E_1|, 2^{n_4} \} , 2^{2n_3} |E_3| \right\} \\ &\leq& \sum_{ \tilde{d} \geq 0} \sum_{\vec{n} \geq \vec{N}(\tilde{d})}  2^{-n_1(1-\theta)} 2^{-n_2} 2^{-n_3}2^{-n_4 \theta} \left[ 2^{n_2} |E_2| \cdot 2^{n_1(1-\gamma)} |E_1|^{1-\gamma} 2^{n_4 \gamma} \right]^{\theta_1}  \left[  2^{2n_3} |E_3| \right]^{\theta_2} \\ &=&  \sum_{ \tilde{d} \geq 0} \sum_{\vec{n} \geq \vec{N}(\tilde{d})} 2^{-n_1 (1-\theta- \theta_1(1-\gamma)) } 2^{-n_2 (1-\theta_1) } 2^{-n_3(1-2\theta_2)}  2^{-n_4 (\theta -\theta_1 \gamma)}|E_1|^{(1-\gamma) \theta_1}|E_2|^{\theta_1}    |E_3|^{\theta_2}.
\end{eqnarray*}
The conditions for summability are $\theta + \theta_1(1-\gamma),   \theta_1 <1, 2\theta_2 <1$ and $\theta_1 \gamma < \theta$, in which case the above display boils down to
\begin{eqnarray*}
\Lambda_I(f_1, f_2, f_3, f_4) \lesssim  \min\{ |E_1|^{(1-\gamma) \theta_1}, |E_1|^{1-\theta} \}   |E_2|^{1/2 + \theta_1/2} \min \{ |E_3|^{\theta_2}, |E_3|^{1-\theta_2}\} .
\end{eqnarray*}
Now let $0 < \epsilon <<1$. Choosing $ \theta_1= 1/2+\epsilon, \theta_2 =1/2-\epsilon,  \theta = \gamma = 1-\epsilon$ followed by $\theta_1= 1/2-\epsilon, \theta_2=1/2+\epsilon, \theta =\gamma= \epsilon$ yields restricted weak-type estimates in neighborhoods near

\begin{eqnarray*}
E_0 &=& \left\{\left(0, \frac{1}{2},\frac{1}{2},0\right), \left(1, \frac{1}{2}, \frac{1}{2}, -1\right)\right\}.
\end{eqnarray*}
Moreover, letting $ \theta_1 =1-\epsilon, \theta_2 = \epsilon , \theta =\gamma=1- \epsilon$ yields estimates in neighborhoods near 
\begin{eqnarray*}
E_2 &=& \left\{ (0,1, 1, 0), (0, 1, 0,1) \right\}.
\end{eqnarray*}
Lastly, letting $ \theta_1 =1-\epsilon, \theta_2 = \epsilon , \theta =\gamma= \epsilon$ yields estimates in neighborhoods near

\begin{eqnarray*}
E_1&=& \left\{ (1, 1, 1, -1), (1,1,0, 0) \right\}.
\end{eqnarray*}
To finish, recall the generic toy model estimate 
\begin{eqnarray*}
&& \Lambda^{k_0}_{Toy} (f_1, f_2, f_3, f_4)\\ &\lesssim&2^{2k_0} \sum_{ \tilde{d}, \vec{n} \geq \vec{N}(\tilde{d})}  2^{-n_1} 2^{-n_2} 2^{-n_3} 2^{-n_4} \min \{ 2^{n_2} |E_2| \min \{ 2^{n_4} ,  2^{n_1} |E_1|\}, 2^{2n_3} |E_3| \} 
\end{eqnarray*}
with size restrictions $2^{-n_1} \lesssim \min \{1, 2^d |E_1|\}$, $2^{-n_2} \lesssim \min \{ 2^{\tilde{d}} |E_2|^{1/2}, 2^{\tilde{d}} |E_2|\}, 2^{-n_3} \lesssim \min \{ 1, 2^{\tilde{d}} |E_3| \}, 2^{-n_4} \lesssim 2^{-\tilde{N} d}$. To estimate $\Lambda^{k_0}_{Toy}(f_1, f_2, f_3, f_4)$,  it suffices to note

\begin{eqnarray*}
\Lambda^{k_0}_{Toy}(f_1, f_2, f_3, f_4) \lesssim \sum_{ \tilde{d}, \vec{n} \geq \vec{N}(\tilde{d})} 2^{-n_1(1-\theta_1(1-\gamma))} 2^{-n_2(1 - \theta_1)} 2^{-n_3 (1-2 \theta_2)} 2^{-n_4(1-\gamma \theta_1)}|E_1|^{(1-\gamma)\theta_1} |E_2|^{ \theta_1} |E_3|^{\theta_2}.
\end{eqnarray*}
It is a simple matter to see that generalized restricted weak estimates are available in neighborhoods near $E_0, E_1, E_2$.  Interpolating and applying symmetry then yields

\begin{eqnarray*}
B[a_1, a_2]: L^{p_1}(\mathbb{R}) \times W_{p_2}(\mathbb{R}) \times L^{p_3}(\mathbb{R}) \rightarrow L^{\frac{1}{\frac{1}{p_1} +\frac{1}{p_2} + \frac{1}{p_3}}}(\mathbb{R}) 
\end{eqnarray*}
provided $\frac{1}{p_1} + \frac{1}{p_2} <1, \frac{1}{p_2} + \frac{1}{p_3} <1$, $2 < p_2 < \infty$. 
\end{proof}

\small  
\nocite{*}
\bibliography{MER}{}

\begin{thebibliography}{10}

\bibitem{2015arXiv151104948B}
C.~{Benea} and C.~{Muscalu}.
\newblock {Multiple Vector Valued Inequalities via the Helicoidal Method}.
\newblock {\em ArXiv e-prints}, November 2015.

\bibitem{MR1809116}
Michael Christ and Alexander Kiselev.
\newblock Maximal functions associated to filtrations.
\newblock {\em J. Funct. Anal.}, 179(2):409--425, 2001.

\bibitem{2013arXiv1311.1574J}
J.~{Jung}.
\newblock {Iterated trilinear Fourier integrals with arbitrary symbols}.
\newblock {\em ArXiv e-prints}, November 2013.

\bibitem{MR3286565}
Robert Kesler.
\newblock Mixed estimates for degenerate multi-linear operators associated to
  simplexes.
\newblock {\em J. Math. Anal. Appl.}, 424(1):344--360, 2015.

\bibitem{MR1491450}
Michael Lacey and Christoph Thiele.
\newblock {$L\sp p$} estimates on the bilinear {H}ilbert transform for
  {$2<p<\infty$}.
\newblock {\em Ann. of Math. (2)}, 146(3):693--724, 1997.

\bibitem{MR1689336}
Michael Lacey and Christoph Thiele.
\newblock On {C}alder\'on's conjecture.
\newblock {\em Ann. of Math. (2)}, 149(2):475--496, 1999.

\bibitem{MR2293255}
Michael~T. Lacey.
\newblock {\em Issues related to {R}ubio de {F}rancia's {L}ittlewood-{P}aley
  inequality}, volume~2 of {\em NYJM Monographs}.
\newblock State University of New York, University at Albany, Albany, NY, 2007.

\bibitem{MR2221256}
C.~Muscalu, T.~Tao, and C.~Thiele.
\newblock The bi-{C}arleson operator.
\newblock {\em Geom. Funct. Anal.}, 16(1):230--277, 2006.

\bibitem{MR3052499}
Camil Muscalu and Wilhelm Schlag.
\newblock {\em Classical and multilinear harmonic analysis. {V}ol. {II}},
  volume 138 of {\em Cambridge Studies in Advanced Mathematics}.
\newblock Cambridge University Press, Cambridge, 2013.

\bibitem{MR1887641}
Camil Muscalu, Terence Tao, and Christoph Thiele.
\newblock Multi-linear operators given by singular multipliers.
\newblock {\em J. Amer. Math. Soc.}, 15(2):469--496, 2002.

\bibitem{MR1981900}
Camil Muscalu, Terence Tao, and Christoph Thiele.
\newblock A counterexample to a multilinear endpoint question of {C}hrist and
  {K}iselev.
\newblock {\em Math. Res. Lett.}, 10(2-3):237--246, 2003.

\bibitem{MR2127984}
Camil Muscalu, Terence Tao, and Christoph Thiele.
\newblock {$L\sp p$} estimates for the biest. {I}. {T}he {W}alsh case.
\newblock {\em Math. Ann.}, 329(3):401--426, 2004.

\bibitem{MR2127985}
Camil Muscalu, Terence Tao, and Christoph Thiele.
\newblock {$L\sp p$} estimates for the biest. {II}. {T}he {F}ourier case.
\newblock {\em Math. Ann.}, 329(3):427--461, 2004.

\bibitem{MR3329857}
Camil Muscalu, Terence Tao, and Christoph Thiele.
\newblock Multi-linear multipliers associated to simplexes of arbitrary length.
\newblock In {\em Advances in analysis: the legacy of {E}lias {M}. {S}tein},
  volume~50 of {\em Princeton Math. Ser.}, pages 346--401. Princeton Univ.
  Press, Princeton, NJ, 2014.

\bibitem{MR2881301}
Richard Oberlin, Andreas Seeger, Terence Tao, Christoph Thiele, and James
  Wright.
\newblock A variation norm {C}arleson theorem.
\newblock {\em J. Eur. Math. Soc. (JEMS)}, 14(2):421--464, 2012.

\bibitem{MR2199086}
Christoph Thiele.
\newblock {\em Wave packet analysis}, volume 105 of {\em CBMS Regional
  Conference Series in Mathematics}.
\newblock Published for the Conference Board of the Mathematical Sciences,
  Washington, DC; by the American Mathematical Society, Providence, RI, 2006.

\end{thebibliography}
\bibliographystyle{plain}

 \end{document}